\documentclass[notitlepage,11pt,reqno]{amsart}
\usepackage[foot]{amsaddr}
\usepackage{amssymb,nicefrac,bm,upgreek,mathtools,verbatim}
\usepackage[final]{hyperref}
\usepackage[mathscr]{eucal}
\usepackage{dsfont}
\usepackage[normalem]{ulem}
\usepackage{amsopn,esint}
\usepackage[pagewise]{lineno}
\usepackage[T1]{fontenc}
\usepackage[utf8]{inputenc}%for appendix
\usepackage{apptools}%for appendix
\AtAppendix{\counterwithin{lemma}{section}} %for appendix

\usepackage[margin=1in]{geometry}
\newcommand{\stkout}[1]{\ifmmode\text{\sout{\ensuremath{#1}}}\else\sout{#1}\fi}
\newtheorem{lemma}{Lemma}[section]
\newtheorem{theorem}{Theorem}[section]

\theoremstyle{definition}
\newtheorem{definition}{Definition}[section]

\newtheorem{example}{Example}[section]

\newtheorem{remark}{Remark}[section]
\numberwithin{theorem}{section}
\numberwithin{equation}{section}

\hypersetup{
  colorlinks=true,
  citecolor=dmagenta,
  linkcolor=dblue,
  urlcolor = dmagenta,
  anchorcolor = dblue,
  frenchlinks=false,
  pdfborder={0 0 0},
  naturalnames=false,
  hypertexnames=false,
  breaklinks}
\usepackage[capitalize,nameinlink]{cleveref}

\usepackage[abbrev,msc-links,nobysame,citation-order]{amsrefs}
\crefname{section}{Section}{Sections}
\crefname{subsection}{Section}{Sections}
\crefname{condition}{Condition}{Conditions}
\crefname{hypothesis}{Hypothesis}{Conditions}
\crefname{assumption}{Assumption}{Assumptions}
\crefname{lemma}{Lemma}{Lemmas}
\crefname{fact}{Fact}{Facts}

\Crefname{figure}{Figure}{Figures}

\crefformat{equation}{\textup{#2(#1)#3}}
\crefrangeformat{equation}{\textup{#3(#1)#4--#5(#2)#6}}
\crefmultiformat{equation}{\textup{#2(#1)#3}}{ and \textup{#2(#1)#3}}
{, \textup{#2(#1)#3}}{, and \textup{#2(#1)#3}}
\crefrangemultiformat{equation}{\textup{#3(#1)#4--#5(#2)#6}}%
{ and \textup{#3(#1)#4--#5(#2)#6}}{, \textup{#3(#1)#4--#5(#2)#6}}%
{, and \textup{#3(#1)#4--#5(#2)#6}}

\Crefformat{equation}{#2Equation~\textup{(#1)}#3}
\Crefrangeformat{equation}{Equations~\textup{#3(#1)#4--#5(#2)#6}}
\Crefmultiformat{equation}{Equations~\textup{#2(#1)#3}}{ and \textup{#2(#1)#3}}
{, \textup{#2(#1)#3}}{, and \textup{#2(#1)#3}}
\Crefrangemultiformat{equation}{Equations~\textup{#3(#1)#4--#5(#2)#6}}%
{ and \textup{#3(#1)#4--#5(#2)#6}}{, \textup{#3(#1)#4--#5(#2)#6}}%
{, and \textup{#3(#1)#4--#5(#2)#6}}

\crefdefaultlabelformat{#2\textup{#1}#3}
\newcommand{\vertiii}[1]{{\left\vert\kern-0.25ex\left\vert\kern-0.25ex\left\vert #1
    \right\vert\kern-0.25ex\right\vert\kern-0.25ex\right\vert}}
\newcommand{\cH}{{\mathcal{H}}}  % Hamiltonian
\newcommand{\sD}{\mathscr{D}}
\newcommand{\RR}{\mathds{R}}

\newcommand{\D}{\mathrm{d}}

\usepackage{color}
\definecolor{dmagenta}{rgb}{.4,.1,.5}
\definecolor{dblue}{rgb}{.0,.4,.7}
\definecolor{mblue}{rgb}{.0,.0,.7}
\definecolor{ddblue}{rgb}{.0,.0,.4}
\definecolor{dred}{rgb}{.5,.0,.0}
\definecolor{dgreen}{rgb}{.0,.5,.0}
\definecolor{Eeom}{rgb}{.0,.0,.5}

\allowdisplaybreaks

\newcommand{\ttl}{\Large Initial boundary value problem for 1D scalar balance laws with strictly convex flux}
\begin{document}
\title[Explicit formula for balance laws in the quarter plane]{\ttl}

\author{Manas R. Sahoo$^1,$ Abhrojyoti Sen$^2$ and Manish Singh$^1$}
\address{$^1$School of Mathematical Sciences, National Institute of Science Education and Research, An OCC of Homi Bhabha National Institute, Bhubaneswar, P.O. Jatni, Khurda, Odisha 752050, India. Email: manas@niser.ac.in, manish@niser.ac.in}
%\author{Abhrojyoti Sen$^2$}
\address{$^2$Department of Mathematics, Indian Institute of Science Education and Research, Dr. Homi Bhabha Road, Pune 411008, India. Email: abhrojyoti.sen@acads.iiserpune.ac.in.}

%\author{Manish Singh$^1$}

%\author{Abhrojyoti Sen$^1$ and Anupam Sen$^2$}

%\address{$^1$Department of Mathematics, Indian Institute of Science Education and Research, Dr. Homi Bhabha Road, Pune 411008, India. Email: abhrojyoti.sen@acads.iiserpune.ac.in.
%\vspace{.05pt}
%$^2$Centre for Applicable Mathematics, Tata Institute of Fundamental Research, Bengaluru, 560065, India. Email: anupam21@tifrbng.res.in}

\begin{abstract}
A Lax-{O}le\u\i nik type explicit formula for 1D scalar balance laws has been recently obtained for the pure initial value problem by Adimurthi et al. in \cite{manish}. In this article, by introducing a suitable boundary functional, we establish a Lax-{O}le\u\i nik type formula for the initial-boundary value problem. For the pure initial value problem, the solution for the corresponding Hamilton-Jacobi equation turns out to be the minimizer of a functional on the set of curves known as {\em $h$-curves}. In the present situation, part of the {\em h-curve} joining any two points in the quarter plane may cross the boundary $x=0$. This phenomenon breaks the simplicity of the minimization process through the boundary functional compared to the case of conservation laws.
Moreover, this complicates the verification of the boundary condition in the sense of Bardos, le Roux, and N\'{e}d\'{e}lec \cite{Bardos}. To verify the boundary condition, the boundary points are classified into three types depending on the structure of the minimizers at those points.  Finally, by introducing characteristic triangles, we construct generalized characteristics and show that the explicit solution is entropy admissible.
\end{abstract}

%%%%%%%%%%%%%%%%%%%%%%%%%%%%%%%%%%%%%%%%%%%%%%%%%%%%%%%%%%%%%%%%%%%%%%%%%%%%%%%%
\maketitle
\tableofcontents
%%%%%%%%%%%%%%%%%%%%%%%%%%%%%%%%%%%%%%%%%%%%%%%%%%%%%%%%%%%%%%%%%%%%%%%%%%%%%%%%

\section{Introduction and main results}
In this article, we are interested in studying the initial boundary value problem for the scalar balance laws in one space dimension. We consider the balance law
\begin{align}\label{equation 1}
u_t+f(u)_x=\alpha(t)u,\,\,\,\,\,\, x>0,\,\, t>0,
\end{align} 
adjoined with the initial and boundary data
\begin{align}\label{IBD}
    u(x,0)=u_0(x),\,\,\, u(0,t)=u_b(t),
\end{align}
where $u_0 \in L^{\infty}([0,\infty))$ and $u_b \in BV([0, \infty)).$
%are bounded measurable functions, and furthermore $u_b$ is locally \textcolor{blue}{BV}. 
Moreover, we set the following assumptions on the flux $f$ and source term $\alpha$ in \eqref{equation 1}:

$f: \mathbb{R}\to \mathbb{R}$ is a $C^2$ strictly convex function with superlinear growth, i.e., $\lim_{|u| \to \infty} \frac{f(u)}{|u|}=\infty$ and $\alpha \in L^{\infty}([0, \infty); \mathbb{R}).$

Prior to discussing the theoretical developments, let us mention that from an application viewpoint, the population dynamics \cite{sharpe1911problem}, fluid flow in a vessel \cite{mackey1993multistability}, as well as bacterial culture development in an ongoing process of fermentation \cite{gyllenberg1982nonlinear} have been successfully studied using the balance law \eqref{equation 1}.

When $\alpha=0,$ the equation \eqref{equation 1} becomes a scalar conservation law and the corresponding Cauchy problem has been extensively studied in one and higher space dimensions (see \cites{Leflochbook, serrebook, Bressanbook} and references cited  therein). One of the important aspects of their study is to obtain an explicit representation of solutions to the scalar conservation laws. The work of Hopf \cite{hopf} is the first significant result in this direction. He considered the initial value problem for the Burgers' equation (set $f(u)=u^2/2$ and $\alpha(t)=0$ in \eqref{equation 1}) in one space dimension and obtained an explicit formula via vanishing viscosity method. In a subsequent work \cite{lax}, Lax generalized this result and obtained the explicit representation of solutions for any strictly convex flux $f$ with superlinear growth by introducing a variational principle. Joseph, in his work \cite{Joseph88} first considered the Burgers' equation in the quarter plane, and obtained the explicit formula by introducing a new boundary functional and utilizing the vanishing viscosity method. In \cite{joseph}, Joseph and Gowda extended the prior result and studied the initial boundary value problem for the scalar conservation laws (set $\alpha=0$ in \eqref{equation 1}) in the framework of Lax. It is shown that the explicit representation of solutions to the initial boundary value problem depends on the minimization of two kinds of functionals-initial and boundary functionals and the relation between them at the boundary.
Recently, Adimurthi et al. \cite{manish} extended the results of Lax \cite{lax} and derived the Lax-{O}le\u\i nik formula to the initial value problem for the scalar balance laws \eqref{equation 1}. In this current contribution, we address the question concerning the explicit representation of solutions to the initial boundary value problem associated with \eqref{equation 1}. Although, in our case, we broadly follow the path of \cite{joseph}, a large amount of new and different strategies are required to employ the variational approach. One of the primary reasons for these difficulties is that we have less control over the behaviour of the characteristic curves in the quarter plane which happen to be straight lines in the $\alpha=0$ case.  We will describe this phenomenon in detail in the subsequent sections.

For the initial boundary value problem, one of the key issues is to define the correct notion of boundary condition as the solution may not satisfy the boundary data in the strong sense at the boundary. We show that the derived explicit formula for the solution satisfies the initial condition in the strong sense and a weak form of boundary condition in the sense of Bardos- le Roux- N\'{e}d\'{e}lec \cite{Bardos}. Let $u_b(t)$ be a boundary data given in \eqref{IBD}, then the weak form of boundary condition in the sense of  Bardos- le Roux- N\'{e}d\'{e}lec for the solution $u(x, t)$ is the following:
\begin{equation*}
\sup_{k \in I(u(0+,t),u_b(t))} Sgn \big[\big(u(0+,t)-k\big)\big(f(u(0+,t))-f(k)\big)\big]=0, \,\, \textnormal{for almost every $t>0$},
\end{equation*}
where
\[I(u(0+,t), u_b(t))=\big[\min\{u(0+,t), u_b(t)\}, \max\{u(0+,t), u_b(t)\} \big].\] For strictly convex $f,$ the above condition can be simplified (see \cite{lefloch}) to an equivalent form,\\
\begin{equation}\label{boundary condition}
    \begin{aligned}
    \begin{cases}
    \textnormal {either}\,\,\, u(0+, t)=\bar{u_b}(t),\\
    \textnormal{or}\,\,\,\, f^{\prime}(u(0+,t)))\leq 0 \,\,\, \textnormal{and} \,\,\, f(u(0+, t)) \geq f(\bar{u_b}(t)).
    \end{cases}
    \end{aligned}
\end{equation}
Here $\bar{u_b}(t)=\max\{u_b(t), \lambda_f\},$  where $\lambda_f$ is the point where $f^{\prime}$ changes its sign. 
\begin{remark}
    Note that Bardos et al. \cite{Bardos} derived the weak form of boundary condition by considering a general quasi-linear hyperbolic equation in a bounded domain $\Omega \times [0, T],$ where $\Omega \subset \mathbb{R}^d$ with a piecewise regular boundary. The above mentioned form of the boundary condition \eqref{boundary condition} is suitably reformulated by LeFloch \cite{lefloch} (also see \cite{F10} for two boundaries) to establish the existence and uniqueness of scalar conservation laws with strictly convex flux in the quarter plane.
\end{remark}
In the subsequent subsection, we provide a concise overview of prior research endeavours that have examined balance laws.
\subsection{Background of balance laws}
The general form of a Cauchy problem of 1D scalar balance laws reads
\begin{align}\label{general balance laws}
    &u_t+f(u)_x=g(x, t, u),\\
    &u(x,0)=u_0(x)\nonumber.
    \end{align}
Dafermos \cite{MR457947} studied the structure of the solution of the above Cauchy problem for more general flux with source field, where $f\in C^2, f_{uu}>0,$ $g\in C^1$ and also $f,g$ depends on time and space variables. However, in \cite{MR457947} the construction procedure of the solution was not discussed. The main tool for studying the structure of the solution is to study the generalized characteristics associated with the corresponding solution. The author deals with a solution $u(x,t)$ which is in the class of $BV_{loc}$ on $\mathbb{R}\times [0, \infty)$ and $u(\cdot, t)$ has locally bounded variation in $x$ on $\mathbb{R}$ for any fixed time  $t\in [0, \infty).$
Regarding the well-posedness theory for the initial boundary value problem, we refer to the work of Colombo and Rosini \cite{CR05} who considered the system of equations of type \eqref{general balance laws} in a domain $\Omega=\left\{(x,t)\in \mathbb{R}^2 \,|\, x\geq \psi(t)\,\, \text{and}\,\, t\geq t_0 \right\}$ for $t_0\in \mathbb{R}$ and $\psi\in C^0\left([t, \infty), \mathbb{R}\right)$ and also see \cites{CR07, CG10} for the system of balance laws with non-characteristic boundary. Moreover, in \cite{CR15} Colombo and Rossi established the well-posedness theory for a more general version of \eqref{general balance laws} in the multidimensional bounded domain $\Omega$ with non-zero boundary data $u_b \in \left(BV\cap L^{\infty}\right)\left(\Omega \times [0, T], \mathbb{R}\right)$ via obtaining rigorous TV bounds on the solution. Recently, in \cite{ABBCN23}, the authors showed that the solution of \eqref{general balance laws} with the source depending only on $(x,t)$ and having integral bound can behave wildly and some conjectures on the regularity of solutions are posed.

Another physically important equation is the balance laws with a moving source, for instance, 
 \[u_t+f(u)_x=C(x)h(u).\] The existence theory for the above equation via the vanishing viscosity method has been studied by Dias and LeFloch in \cite{MR1924285} and considering different appropriate conditions on the source field $C(x)h(u),$ the large time behaviour of the solution has been studied by Liu \cite{MR913412}. There are many pieces of literature where the large-time behaviour has been extensively studied, possibly for the general source term $g(x,t, u).$ Depending on the choice of method, suitable conditions on the flux and the source are taken in different situations, for some of the works we refer to \cites{MR895101, MR1184059, MR1286920, MR1245070, MR1270661, MR1693629, MR2172699, MR1751427, MR1432096, MR1427730, MR1808451} and the references therein. We emphasize that the works mentioned above did not consider the construction of explicit formulas. Even when the source depends only on $(x,t)$ (but not on the unknown), the question is not well settled. 

  In the case of balance laws, because of the wild behavior of the characteristics, getting the Lax-{O}le\u\i nik type representation for the solution is not always possible for a large class of balance laws. On the other hand, the balance laws are mathematically significant as the balance term breaks the symmetry of the equation. Therefore, it is challenging to obtain a rarefaction solution in contrast to the case of conservation law (zero balance term). In the most basic example of \eqref{equation 1}, our goal is to investigate the potential of establishing a Lax-{O}le\u\i nik type formula for the initial-boundary value problem.
 %We wish to underline how mathematically difficult even the model problem \eqref{equation 1} is. 
 In the following subsection, we describe the main contributions of this paper in detail.
\subsection{Main results} As mentioned earlier, several authors studied the scalar balance laws from the perspective of the structure of the solution, asymptotic behaviour etc. However, even for the initial value problem, the explicit construction of a solution was unavailable. Adimurthi et al.\cite{manish} first studied the Lax-{O}le\u\i nik type formula for the pure initial value problem.  %This makes the proof of explicit formula and the structure of solution more involved and interesting. 
In this paper, we extend the previous results of \cite{manish} to the initial boundary value problem. The balance laws \eqref{equation 1} can be rephrased as
\begin{align}\label{new conservation laws}
    w_t+ e^{-\beta(t)}f\left(e^{\beta(t)}w\right)_x=0,
\end{align}
where $w=e^{-\beta(t)}u$ and $\beta(t)=\int_0^t \alpha(\theta)d\theta.$ 
Furthermore, the counterpart of the conservation laws \eqref{new conservation laws} is the Hamilton-Jacobi equation  \[W_t+e^{-\beta(t)}f\left(e^{\beta(t)}W_x\right)=0,\] 
adjoined with the following initial and boundary data:
\begin{align*}
W(x, 0)=\int_{0}^{x} u_0(y)dy,\,\,\, W_x (0, t)=e^{\beta(t)}u_b (t).
\end{align*}  
Now motivated by the works of Joseph and Gowda \cites{joseph, Joseph88}, for any Lipschitz continuous path $\gamma:[0, t]\to [0, \infty)$ with $\gamma(t)=x,$ we introduce the following functional,
\begin{align}\label{First functional}
    J(\gamma,x,t)=\int_{\{\gamma(\theta)\neq 0\}}f^*(\dot{\gamma}(\theta))e^{-\beta(\theta)}d\theta-\int_{\{\gamma(\theta)= 0\}}f(\bar{u_b}(\theta))e^{-\beta(\theta)}d\theta +W(\gamma(0),0).
 \end{align}
The above functional \eqref{First functional} will be minimized over all such Lipschitz curves, although the minimizing curve might not be unique. Moreover, the minimizing curves have specific geometric structures and will be known as the {\it characteristics}. 

In the case of pure initial value problem \cite{manish}, the characteristic curves turn out to be {\em h-curves}, and in general, they are not straight lines. Below, let us give the definition {\em h-curves} suitable adapted to the quarter plane.
\begin{definition}[\textbf{\textit{h}-curves}]\label{defn h curve}
Given any two points $(x_1, t_1)$ and $(x_2, t_2)$ in $[0, \infty)\times [0, \infty)$ with $0\leq t_2< t_1,$  the curve $X(t)$ with $X(t_1)=x_1$ and $X(t_2)=x_2$ is called a {\em h-curves} if there exists a unique $y_0\in \mathbb{R}$ such that following relation holds, 
\begin{align*}
    \dot{X}(t)=f^{\prime}\left(y_0 e^{\beta(t)}\right).
\end{align*}
%The curve joining from $(x_1,t_1)$ to $(x_2,t_2),$ where $0\leq t_2< t_1$ and satisfying the differential equation $X^{\prime}(\theta)=f^{\prime}(y_0e^{\beta(\theta)})$ for the unique $y_0$ is called a {\em h-curve.} 
Moreover, for $x_2=0,$ we denote the unique point $y_0$ as $h(x_1,t_1,t_2)$ and for $t_2=0,$ we denote the unique point $y_0$ as $h(x_1-x_2, t_1).$
\end{definition}
In the case of the initial boundary value problem, {\em h-curves} joining any two points of the quarter plane can cross the boundary and go beyond the prescribed domain (see \cref{example1}). Therefore such curves are not allowed to be a part of the minimization process and this complicates the structure of characteristics (see \cref{example2}). To be more precise, characteristics may oscillate along the boundary (which is $t$-axis) and meet the boundary at countably many points.  
%This complicates the  the characteristic curve is not simple as the pure initial value problem. 
%In general, for the boundary value problem, it is not possible to directly work with all the {\em h-curves} as we observed that for certain flux $f$ and $\alpha(t)$, the characteristic curve 
%can cross the boundary and go beyond of the prescribed domain (see example \eqref{example1}). In addition, as an effect of the source term the {\em h-curve} may oscillate along the boundary (which is $t$-axis) and meet the boundary countably many times (see example \eqref{example2}).
This phenomenon is completely different from the homogeneous case $\alpha(t)=0$ in which Joseph and Gowda \cite{joseph} showed that the characteristic curve is at most of three pieces of line segments. Unlike them, in this scenario, characteristic curves obtained via minimization may consist of countably many pieces of {\em h-curves}, and this leads to new difficulties in utilizing the variational approach.

As stated above, {\em h-curves} joining any two points can leave the quarter plane. Therefore, we introduce two sets $H_A (x,t)$ and $H_B (x,t)$ as follows:
\begin{align*}
&H_{A} (x,t)=\left\{y\,\,\Bigg|\,\, y+\int_{0}^\tau f^{\prime}\left(h(x-y,t)e^{\beta (\theta)}\right)d\theta \geq 0,\,\,\, \text{for fix}\,\, (x,t)\,\, \text{and}\,\, \,\,\, 0\leq \tau \leq t\right\},\\
&H_{B} (x,t)=\left\{\tau\,\,\Bigg| \int_{\tau}^s f^{\prime}\left(h(x,t, \tau)e^{\beta (\theta)}\right)d\theta \geq 0,\,\,\, \text{for fix}\,\, (x,t)\,\, \text{and}\,\, 0\leq \tau \leq s \leq t\right\}.
\end{align*}
Now following \cite{joseph}, we introduce the initial functional $A(y, x, t)$ as,
\begin{equation}\label{initial functional}
A(y,x, t)=\int_0^t e^{-\beta(\theta)} f^*\left(f^{\prime}\left(h(x-y,t)e^{\beta(\theta)}\right)\right)d\theta +\int_0^y u_0(\theta)d\theta,
\end{equation}
for $y\in H_A(x,t),$ and boundary functional $B(\tau, x, t)$ as,
\begin{equation}\label{boundary functional}
B(\tau, x, t)=\int_{\tau}^t e^{-\beta(\theta)} f^*\left(f^{\prime}\left(h(x,t, \tau)e^{\beta(\theta)}\right)\right)d\theta +W(0, \tau),
\end{equation}
for $\tau\in H_B(x, t).$
Here $f^*$ is the convex dual of $f$ defined as $f^*(u)=\displaystyle{\max_{v\in \mathbb{R}}}\left\{uv-f(v)\right\},$ and 
\begin{align*}
    W(0, \tau):=\inf_{\gamma}\int_{\gamma(\theta)\neq 0}f^*(\dot{\gamma}(\theta))e^{-\beta(\theta)}-\int_{\gamma(\theta)=0}f(\bar{u}_b(\theta))e^{-\beta(\theta)}d\theta+W(\gamma(0), 0).
\end{align*}
Minimization in the above functional is taken over all Lipschitz paths in the quarter plane joining $(0,\tau)$ to $(\gamma(0), 0)$ such that $\gamma(\tau)=0.$

Also, for a given $(x,t),$ $y_*(x,t)$ and $y^*(x,t)$ is the leftmost and rightmost points on the $x$-axis from the set $H_A(x,t)$ such that
\begin{align}\label{definition A}
    A(x,t)=\min_{y\in H_A(x,t)} A(y,x, t)=A(y_*(x,t), x,t)=A(y^*(x,t), x, t)
\end{align}
and $\tau_*(x,t)$ and $\tau^*(x,t)$ is the lowermost and uppermost points on the $t$-axis from the set $H_B(x,t)$ such that
\begin{align}\label{definition B}
    B(x,t)=\min_{ \tau\in H_B(x,t)} B(\tau, x, t)=B(\tau^*(x,t), x,t)=B(\tau_*(x,t), x,t).
\end{align}
Note that the superlinear growth of $f^*$ followed by an application of Jensen's inequality gives the existence of the above minimization, and the minimizer lies in a compact interval for any fixed $(x,t)$. Therefore $y^*(x,t), \tau_*(x,t)$ are well defined. This implies that $h\left(x-y_*(x,t), t\right)$ and $h\left(x,t, \tau_*(x,t)\right)$ are also well defined. Proof of these facts is available in \cref{min exists} and \cref{min h curve}, respectively.

Now, we are ready to state the main results. Before that, let us quickly recall the solution concept for balance laws and introduce the necessary definitions. It is well known that the solution of balance laws may develop discontinuity after a finite time even if the initial data is smooth. Therefore, the solution is always considered in the weak sense. The weak formulation to \eqref{equation 1} is the following:
\begin{definition}[\textbf{Weak solution}]\label{weak formulation}
    A function $t\mapsto u(x,t) \in C\left([0,T]; L^1_{loc}(0, \infty)\right)$ is said to be a weak solution to \eqref{equation 1} if the following integral identity
    \begin{align*}
        \int_0^{\infty}\int_0^{\infty} u \varphi_t +f(u)\varphi_x-\alpha(t)u \varphi \,dx dt + \int_0^{\infty}u_0(x)\varphi(x,0)\, dx=0
    \end{align*}
    holds for all $\varphi \in C^{\infty}_c ((0,\infty)\times [0,\infty))$
and $u(\cdot, t)\in BV (\mathbb{R}^+, \mathbb{R})$ for any fixed $t>0,$ and satisfies
\eqref{boundary condition}.
\end{definition}
\noindent Next, we define the controlled curves.
\begin{definition}[\textbf{Controlled curves}] Let $x\in [0, \infty)$ and   $0\leq s\leq t$, we then define the set of controlled curves to be
\begin{align*}
\Gamma(x,t,s) 	&:=\Big\{\gamma :[s,t]\to [0, \infty),\,\,\gamma(t)=x\,\,\, \text{and}\,\,\gamma\,\,\,\text{is a Lipschitz function}\Big\},\\
\Gamma(x,t)		&:=\Gamma(x,t,0).
\end{align*}
\end{definition}
Using the functional $J(\gamma, x,t)$ defined in \eqref{First functional}, for any $0\leq s\leq t$ we define
\begin{align*}
    J(\gamma, x,t, s):=\int_{\{\gamma(\theta)\neq 0\}}f^*\left(\dot{\gamma}(\theta)\right)e^{-\beta(\theta)}d\theta-\int_{\{\gamma(\theta)= 0\}}f\left(\bar{u_b}(\theta)\right)e^{-\beta(\theta)}d\theta +W(\gamma(s),s).
\end{align*}
Now the value functions are defined as
\begin{definition}[\textbf{Value function}]\label{defnvalue} Let  $x\in [0, \infty)$ and  $0\leq s < t.$ We define the value functions by 
\begin{align}
W(x,t) 	&=\inf_{\gamma\in\Gamma(x,t)}J\left(\gamma, x,t\right),\nonumber\\
W(x,t,s) 	&=\inf_{\gamma\in\Gamma(x,t,s)}J\left(\gamma, x, t, s\right)\label{HL1}.
\end{align}
\end{definition}
%The following is the main result of this article. 
%Before stating the main result, we assume that \cref{min h curve} to \cref{min exists}  holds true, which we proved in the later part of the paper. Note that \cref{min h curve} guarantees the existence of a unique $h$-curve in the class of Lipschitz continuous functions. Moreover, since the flux function $f$ is $C^2,$ by using the implicit function theorem one can further deduce that $h(x,t,\tau)$ is differentiable in all of its variables (see \cite{manish} also). \cref{bdf} says that the value function $W(x, t)$ can be represented in terms of $A(x,t)$ and $B(x,t),$ i.e., $W(x,t)=\min\left\{A(x,t), B(x,t)\right\}.$ Next, \cref{5.1} characterizes $W(x,t)$ at the boundary, whereas in \cref{LIP Value} we prove the Lipschitz continuity of the value function $W(x,t)$ in $\left\{(x,t)\,\, |\,\, x\geq 0, t\geq 0\right\}.$ This gives the almost everywhere differentiability of $W(x,t)$ in  $x$ and $t.$ Lastly, in \cref{min exists} we showed the existence of $A(x, t)$ and $B(x, t).$}
\noindent One of the main results of this article is stated below.
\begin{theorem}[\textbf{Explicit representation of solution}]\label{thm explicit formula} The value function $W(x,t)$ in \cref{defnvalue} is a locally Lipschitz continuous function in $(0, \infty)\times (0, \infty)$ and the explicit formula of its partial derivative $W_x=w$, defined almost everywhere is given by 
\begin{equation*}
\begin{aligned}
    w(x,t)=\begin{cases}
    \,h\left(x-y_*(x,t), t\right) \,\,\,\,\,\,\,\, &\textnormal{if}\,\,\,\, A(x,t)\leq B(x,t),\\
    \,h\left(x,t, \tau_*(x,t)\right)\,\,\,\,\,\,\,\,\,\,\, &\textnormal{if}\,\,\,\, A(x,t)> B(x,t),
    \end{cases}
    \end{aligned}
\end{equation*}
  where $\beta(t)=\int_0^t \alpha(\theta)d\theta$ and $A(x,t), B(x,t)$ are given by \eqref{definition A}-\eqref{definition B}.  Furthermore, $u(x,t)=e^{\beta(t)}w(x,t)$ satisfies \eqref{equation 1}-\eqref{IBD} in the sense of \cref{weak formulation}.
 \end{theorem}
One of the crucial tasks in the above theorem is to prove that the constructed solution satisfies the boundary condition \eqref{boundary condition}. Due to the structure of the characteristic curves, the minimizer $\tau(x,t)\not \to t$  always as $x\to 0+$, and therefore the proof of boundary condition is also not direct. We settle this issue by introducing a classification of boundary points depending upon the structure of characteristics and using the dynamic programming principle repeatedly at suitable time slices.

Next, we describe the solution by constructing the generalized characteristic. For that purpose, we need to introduce different types of characteristics triangles with an apex $(x,t), x\geq 0, t>0,$ depending on the different cases $A(x,t)<B(x,t), A(x,t)>B(x,t)$ or $A(x,t)=B(x,t).$ More precisely, we construct a curve $\mathtt{X}(t)$ such that $\dot{\mathtt{X}}(t)=f^{\prime}(u(\mathtt{X}(t),t)).$ In fact, we prove the following:
\begin{theorem}[\textbf{Generalized characteristics}]\label{GC thm}
    Let $t_1>0$ be fixed. Then for each $(x_1,t_1),$ there exists a unique Lipschitz continuous curve $x=\mathtt{X}(t)$ with $\mathtt{X}(t_1)=x_1$ such that the characteristic triangles associated with each point on the curve forms an increasing family of triangles and we have $u(x,t)= (f^{\prime})^{-1}\left(\dot{\mathtt{X}}(t)\right)$ almost everywhere in the quarter plane.
\end{theorem}

A more rigorous statement is presented in the \cref{thm4.1} on \cref{section 4}, which says that the solution would satisfy the Rankine-Hugoniot condition along the discontinuity curve. The construction of such generalized characteristics may be useful while studying the structure of the solution to the initial boundary value problem for \eqref{equation 1} and balance laws with discontinuous flux, which will be our next goal of research. 

Now we address the issue of the uniqueness of the entropy solution. Lax entropy admissibility of the explicit solution presented in the \cref{thm explicit formula} is shown in  \cref{Entropy condition}. Following the same proof as in the homogeneous case ($\alpha=0$), one can easily obtain that if $u(x,t)$ is a weak solution in the sense of \cref{weak formulation} satisfying Lax entropy condition and is piecewise smooth, then it satisfies the Kruzkov's entropy inequality. Hence $u$ is the unique entropy solution that is given by the vanishing viscosity method (see \cite[page 1028]{Bardos}).
 %For piecewise smooth solutions the Lax entropy is equivalent to Kruzkov's entropy criteria. This implies that the solution is unique in the class of piecewise smooth solutions that satisfy the Lax entropy condition and boundary condition \eqref{boundary condition} (see \cites{lefloch, Bardos}).
\begin{theorem}[\textbf{Uniqueness of entropy solution}]\label{UES}
Let $u, v \in L^{\infty}\cap BV_{loc}$ are two solutions to \eqref{equation 1}-\eqref{IBD} in the sense of \cref{weak formulation} with the initial and boundary datum $u_0 \in L^1([0, \infty))\cap L^{\infty}([0, \infty)),$ $u_b\in BV([0, \infty))$ respectively and satisfies the Lax entropy condition \eqref{entropy}. Moreover, assume that $u, v$ are $C^1$ except a discrete set of Lipschitz curves, then $u\equiv v.$
\end{theorem}
We finish this section with an example showing that minimization can not be taken over all {\em h-curves} joining any point $(x, t)$ and $(y,0)$ for balance laws, as the {\em h-curve} joining these points may leave the domain-the quarter plane. 
%The second example demonstrates that the minimization may be achieved by a curve that is of countably many pieces of {\em h-curves}.
\begin{example}\label{example1}
Consider the balance law 
\begin{align*}
\begin{cases}
&u_t + \Big(\frac{(u-60)^2}{2}\Big)_x=\Big(\frac{12t^2-60t+70}{4t^3-30t^2+70t+10}\Big) u,\\
&u_0(x)=10,\,\,\, \text{and}\,\,\, u_b \,\,\text{is any bounded measurable function.} 
\end{cases}
\end{align*}
The {\em h-curve} $X(t)$  joining $(24, 5)$ and $(0, 24)$ satisfies
\begin{align}\label{1stchar}
X^{\prime}(t)=10 e^{\beta (t)}- 60,\,\,\, X(0)=24, \,\, X(5)=24
\end{align} where \[\beta(t)=\int_0^t \frac{12\theta^2-60\theta+70}{4\theta^3-30\theta^2+70\theta+10} d\theta.\]
Then one can solve the above equation \eqref{1stchar} to get \[X(t)=(t-1)(t-2)(t-3)(t-4).\] This crosses the $t$-axis four times, and parts of the curve lie outside the quarter plane. It is easy to see that $\alpha(t)$ is bounded as $t\to 0$ and $t\to \infty.$ Moreover the roots of the denominator are either negative or complex. Hence $\alpha \in L^{\infty}([0, \infty),\mathbb{R}).$
\end{example}
%\subsection{Methodology and plan}
 %Firstly, we characterize the shape of the curve along which the functional achieves its minimum along h-curves provided the h-curves lies entirely in the quarter plane. The method we used is simple and completely different from the very lengthy and technical method used in \cite{manish}. Note that in \cite{manish}, a convex optimization technique followed by the Banach-Saks theorem is used to characterize the curve along which the functional achieves its minimum. Secondly, we characterise the curve along which the boundary functional achieves its minimum which is crucial in the sense that the h-curve joining two points in the quarter plane in principle may cross the quarter plane. Another difficulty is the absence of coercivity condition on the flux function. This issue is addressed  by  using a continuity type argument. %This method is in some sense unique and is not seen in the literature. 
 %In the last we construct generalized characteristics to show a compatibility with the usual characteristics available for classical solution of 1st order partial differential equations\cite{evans}.
\subsection{Organization of the paper} The rest of the article is organized as follows. In \cref{section 2}, we prove \cref{thm explicit formula} via a variational approach and discuss the boundary conditions. In \cref{section 4}, we construct generalized characteristics (\cref{GC thm}) and show that the solution for the initial-boundary value problem is entropy admissible in the sense of Lax. Finally, using this and the boundary condition \eqref{boundary condition}, we conclude \cref{UES}.

\section{The explicit formula}\label{section 2}
This section aims to prove \cref{thm explicit formula}. The proof of \cref{thm explicit formula} needs several preparatory lemmas. 
%Before proving the lemmas, let us introduce a few notations and definitions.

We denote the quarter plane as $Q,$ i.e., $Q=[0, \infty)\times [0, \infty).$ 
%We start with the dynamic programming principle for the initial-boundary value problem. For $\gamma \in \Gamma(x,t, s)$,  let us denote
%Here onward the symbol $R^{+}$ will be used for the closed interval $[0, \infty]$. 
%Let $\gamma:[0,t]\to \mathbb{R}^{+}\times\mathbb{R}^{+}$ is a Lipschitz continuous curve with $\gamma(t)=x.$ Let's First introduce the boundary functional as follows:
%\begin{equation*}
  %  \begin{aligned}
  %  J(\gamma,x,t, s)=\int_{\{\gamma(\theta)\neq 0\}}f^*(\dot{\gamma}(\theta))e^{-\beta(\theta)}d\theta-\int_{\{\gamma(\theta)= 0\}}f(\bar{u_b}(\theta))e^{-\beta(\theta)}d\theta +W(\gamma(s),s)
  %  \end{aligned}
%\end{equation*}
%and $J(\gamma, x,t,0)=J(\gamma,x,t).$
%Define $w(x,t)= \displaystyle{\min_{\gamma} J(\gamma, x, t)}$
Below, we prove the dynamic programming principle which is the first result towards proving \cref{thm explicit formula}.
\begin{lemma}[\textbf{Dynamic programming principle}]\label{DPP}
Let $(x,t)\in Q.$ Then the value function $W(x,t)$ satisfies the dynamic programming principle, i.e., for any $0\leq s \leq t,$ we have
$$W(x, t)= W(x,t,s).$$ 
\end{lemma} 
\begin{proof}
First, we prove that $W(x,t)\geq W(x,t,s).$ Let $\gamma:[0, t]\to [0, \infty)$ be any Lipschitz continuous curve satisfying $\gamma(t)=x$.
Then we have
\begin{align*}
    J(\gamma, x,  t)= &\int_{\{\gamma(\theta)\neq 0\}\cap [s, t]}f^*(\dot{\gamma}(\theta))e^{-\beta(\theta)}d\theta-\int_{\{\gamma(\theta)=0\}\cap [s, t]}f(\bar{u_b}(\theta))e^{-\beta(\theta)}d\theta \\
   & +\int_{\{ \gamma(\theta)\neq 0\} \cap [0, s]}f^*(\dot{\gamma}(\theta))e^{-\beta(\theta)}d\theta
   -\int_{\{\gamma(\theta)= 0\}\cap [0, s]}f(\bar{u_b}(\theta))e^{-\beta(\theta)}d\theta +W_0(\gamma(0))\\
    &\geq \int_{\{\gamma(\theta)\neq 0\}\cap [s, t]}f^*(\dot{\gamma}(\theta))e^{-\beta(\theta)}d\theta-\int_{\{\gamma(\theta)=0\}\cap [s, t]}f(\bar{u_b}(\theta))e^{-\beta(\theta)}d\theta
    +W(\gamma(s), s).
\end{align*}
Therefore 
\begin{align*}
    J(\gamma, x, t)\geq \inf _{\gamma \in \Gamma(x,t,s)}\int_{\{\gamma(\theta)\neq 0\}}f^*(\dot{\gamma}(\theta))e^{-\beta(\theta)}d\theta-\int_{\{\gamma(\theta)=0\}}f(\bar{u_b}(\theta))e^{-\beta(\theta)}d\theta +W(\gamma(s), s).
\end{align*}
Taking infimum over all $\gamma \in \Gamma(x,t),$ we obtain
\begin{equation}\label{dyn1}
    W(x,t)\geq W(x,t,s).
\end{equation}
To establish the reverse inequality, let us recall from \eqref{HL1}:
\begin{align*}
    W(x,t, s)= \inf _{\gamma \in \Gamma(x,t,s)}\int_{\{\gamma(\theta)\neq 0\}}f^*(\dot{\gamma}(\theta))e^{-\beta(\theta)}d\theta-\int_{\{\gamma(\theta)= 0\}}f(\bar{u_b}(\theta))e^{-\beta(\theta)}d\theta +W(\gamma(s), s).
\end{align*}
Now for any $\varepsilon >0 $ there exists a $\gamma_{\varepsilon} \in \Gamma(x,t,s)$ such that 
\begin{equation}\label{eq1}
\begin{aligned}
    &\int_{\{\gamma_{\varepsilon}(\theta)\neq 0\}}f^*(\dot{\gamma_{\varepsilon}}(\theta))e^{-\beta(\theta)}d\theta-\int_{\{\gamma_{\varepsilon}(\theta)=0\}}f(\bar{u_b}(\theta))e^{-\beta(\theta)}d\theta +W(\gamma_{\varepsilon}(s), s)
   \\ &\leq W(x,t, s)+\varepsilon.
    \end{aligned}
\end{equation}
Again for $\varepsilon^{\prime}>0$ there exists a $\gamma_{{\varepsilon}^{\prime}} \in \Gamma(\gamma_{\varepsilon}(s),s) $ such that
\begin{equation}\label{eq2}
\begin{aligned}
    &\int_{\{\gamma_{\varepsilon^{\prime}}(\theta)\neq 0\}}f^*(\dot{\gamma_{\varepsilon^{\prime}}}(\theta))e^{-\beta(\theta)}d\theta-\int_{\{\gamma_{\varepsilon^{\prime}}(\theta)= 0\}}f(\bar{u_b}(\theta))e^{-\beta(\theta)}d\theta +W(\gamma_{\varepsilon^{\prime}}(0), 0)\\
    &\leq W(\gamma_{\varepsilon}(s), s)+\varepsilon^{\prime}.\\
    \end{aligned}
\end{equation}
Let us define the curve
\begin{equation*}
    \begin{aligned}
    \bar{\gamma}(\theta)=\begin{cases}
    \gamma_{\varepsilon^{\prime}}(\theta)\,\,\, &if\,\,\, \theta \in [0,s],\\
    \gamma_{\varepsilon}(\theta)\,\,\, &if\,\,\,\theta \in [s,t].
    \end{cases}
    \end{aligned}
\end{equation*}
Combining the above equations \eqref{eq1} and \eqref{eq2}, we get
\[J(\bar{\gamma}, x, t)\leq W(x,t,s)+\varepsilon +\varepsilon^{\prime}.\] Since $\varepsilon $ and $\varepsilon^{\prime}$ are arbitrary positive numbers, we get
\begin{equation}\label{dyn2}
  W(x,t)\leq W(x,t,s).  
\end{equation}
 The inequalities \eqref{dyn1} and \eqref{dyn2} complete the proof of the dynamic programming principle.
\end{proof}
Despite the fact that the proof is simple, the following lemma is crucial for characterising the minimizers of $J(\gamma,x,t).$ 
\begin{lemma}[\textbf{Minimization via \textit{h}-curve}] \label{min h curve}
Let $S_L$ be the set of all Lipschitz continuous curves in $Q$ joining $(x_1, t_1)$ to $(x_2,t_2)$ with $t_1<t_2,$ i.e.,
\begin{align*}
    S_L:=\left\{\gamma\,\, \Big| \,\,\gamma:[t_1,t_2] \to [0, \infty)\,\, \text{is a Lipschitz function with}\,\, \gamma(t_1)=x_1 \,\, \text{and}\,\, \gamma(t_2)=x_2\right\},
\end{align*}
then we have a unique $X\in S_L$ such that
\begin{align}\label{infh}
\inf_{\gamma \in S_L} \int_{t_1}^{t_2} f^{*}\left(\dot{\gamma} (\theta)\right)e^{-\beta (\theta)} d \theta
= \int_{t_1}^{t_2} f^{*}\left(\dot{X}(\theta)\right)e^{-\beta (\theta)} d \theta,
\end{align}
where $X(t)$ satisfies
\begin{align*}
    \dot{X}(t)=f^{\prime}\left(y_0 e^{\beta(t)}\right)
\end{align*}
 for a unique $y_0\in \RR.$ In other words, $X(t)$ is a { h-curve}. 
\end{lemma}
\begin{proof}
To establish the existence of such curve $X,$ we consider the function $\Phi: \mathbb{R} \to \mathbb{R}$ defined as
\begin{align*}
    \Phi(z):=x_1+\int_{t_1}^{t_2} f^{\prime}\left(z e^{\beta(\theta)}\right)d\theta.
\end{align*}
Since $f$ has superlinear growth, as $z\to \pm \infty,$ we have $\Phi(z)\to \pm \infty.$ On the other hand, since $f^{\prime}$ is increasing, using the intermediate value theorem, we get a unique $y_0\in \mathbb{R}$ such that $\Phi(y_0)=x_2.$ Now setting
\begin{align*}
    X(s)=x_1+\int_{t_1}^{s} f^{\prime}\left(y_0 e^{\beta(\theta)}\right)d\theta \,\,\, \text{for all}\,\,\, s\in [t_1, t_2],
\end{align*}
we obtain the desired curve.

Next to show \eqref{infh}, we use the convexity of $f^*,$ i.e.,
\begin{align}\label{convexity}
    f^*(x) \geq f^*(y) + f^{*\prime}(y)(x-y).
\end{align}
Let $\gamma \in S_L.$ Setting $x=\dot{\gamma}(\theta)$ and $y=\dot{X}(\theta)$ in \eqref{convexity} and integrating over the interval $[t_1, t_2],$ we get
\begin{align} \label{convex inequality}
    \int_{t_1}^{t_2}e^{-\beta(\theta)}f^*\left(\dot\gamma(\theta)\right)d\theta \geq \int_{t_1}^{t_2}e^{-\beta(\theta)}f^*\left(f^{\prime}\left(y_0 e^{\beta(\theta)}\right)\right)d\theta+\int_{t_1}^{t_2}y_0\left(\dot\gamma(\theta)-\dot X (\theta)\right)d\theta,
\end{align}
where we used the identity $f^{* \prime}(\cdot)=(f^{\prime})^{-1}(\cdot).$ Now \eqref{infh} follows from the fact that 
\begin{align*}
\int_{t_1}^{t_2}y_0\left(\dot{\gamma}(\theta)-\dot{X}(\theta)\right)d\theta=0.
\end{align*}
To prove the uniqueness of the minimizer, we observe that if any curve $\gamma$ is different from $X$, then 
$\mu \left(\left\{\theta\,\, |\,\,\dot{\gamma}(\theta)\neq \dot{X}(\theta)\right\}\right)>0$ where $\mu$ denotes the Lebesgue measure. 
This implies the inequality \eqref{convexity} is strict in a set of positive measures, thereby making the inequality \eqref{convex inequality} strict. Therefore the minimizer $X$ is unique.
\end{proof}
%\begin{remark}\label{rem 1}
%The inequality in \eqref{convex inequality} becomes strict when $\mu \left(\left\{\theta\,\, |\,\,\dot{\gamma}(\theta)\neq \dot{X}(\theta)\right\}\right)>0$ where $\mu$ denotes the Lebesgue measure.
%\end{remark}
The next lemma associates the value function $W(x,t)$ with $A(x,t)$ and $B(x,t)$ given in \eqref{definition A}-\eqref{definition B}.
\begin{lemma}\label{bdf}
Let $x>0,t>0$ be fixed. Then the value function $W(x,t)$ can be represented as 
\begin{align*}
    W(x,t)=\min\left\{A(x,t), B(x,t)\right\}
\end{align*}
where $A(x,t)$ and $B(x,t)$ are given by \eqref{definition A} and \eqref{definition B} respectively.
\end{lemma}
\begin{proof}
 Let $\gamma:[0, t]\to [0,\infty)$ is a Lipschitz continuous curve joining $(x,t)$ to $(y, 0)$ with Lipschitz constant $M_0.$ We consider the following two cases. \\
 {\bf Case 1.} Let $\gamma(\tau_0)=0$ for some $\tau_0\in [0, t)$ and $\gamma(\theta)>0$ for all $\theta \in (\tau_0, t].$ In this case we will show that $B(x, t)\leq J(\gamma, x, t)$.  This will be proved in the following three steps.\\
\noindent{\bf Step 1:} We show that the {\em h-curve} joining the nearby points $(x_1,t_1)$ and $(x_2, t_2)$ (with $x_i, t_i>0$ for $i=1,2$) lying on $\gamma$ completely stays in the interior of $Q.$ Let 
\begin{align*}
    X(s)=x_1+\int_{t_1}^{s} f^{\prime}\left(y_0 e^{\beta(\theta)}\right)d\theta\,\,\, \text{for all}\,\, s\in[t_1, t_2]
\end{align*}
is a {\em h-curve} joining $(x_1, t_1)$ and $(x_2, t_2).$ Then for $y_0 \geq 0 \,\,(resp.\,\,\, y_0 \leq 0),$ we have
\begin{align*}
    \frac{\left(f^{\prime}\right)^{-1}\left(\frac{x_2-x_1}{t_2-t_1}\right)}{e^{\displaystyle{\max_{\theta\in [t_1, t_2]}} \beta(\theta)}} \,\, (resp.\,\,\, \geq) \leq y_0 \leq\,\,(resp. \,\,\, \geq) \frac{\left(f^{\prime}\right)^{-1}\left(\frac{x_2-x_1}{t_2-t_1}\right)}{e^{\displaystyle{\min_{\theta\in [t_1, t_2]}} \beta(\theta)}}.
\end{align*}
This implies 
\begin{align}\label{y_0 bounded}
    \left|y_0\right| \leq \max \left\{\frac{\left(f^{\prime}\right)^{-1}(M_0)}{e^{\displaystyle{\max_{\theta\in [t_1, t_2]}} \beta(\theta)}}, \frac{\left(f^{\prime}\right)^{-1}(M_0)}{e^{\displaystyle{\min_{\theta\in [t_1, t_2]}} \beta(\theta)}}\right\},
\end{align}
where $\left|\frac{x_2-x_1}{t_2-t_1}\right| \leq M_0$ and we conclude that 
\begin{align*}
x_2-x_1=\int_{t_1}^{t_2} f^{\prime}\left(y_0 e^{\beta(\theta)}\right)d\theta \to 0\,\,\, \text{as}\,\,\, t_1\to t_2.
\end{align*}
Therefore if $(x_1,t_1)$ and $(x_2, t_2)$ are nearby points on $\gamma$, the $h$-curve joining these two points completely lies inside the quarter plane.

\noindent{\bf Step 2:} Let $X_{\tau}$ be the {\em h-curve} joining the points $(x,t)$ to $(\gamma(\tau), \tau),$ i.e.,
\begin{align*}
X_{\tau}(s)=\gamma(\tau)+\int_{\tau}^s f^{\prime}\left(y_0(\tau)e^{\beta(\theta)}\right)d\theta,\,\,\, \tau\leq s \leq t.
\end{align*}
We claim that $y_0$ is a continuous function. Indeed, from the step 1, we have that $y_0(\tau)$ is bounded and also, we note that
\begin{align}\label{eqq 2.10}
    x-\gamma(\tau)=\int_{\tau}^t f^{\prime}\left(y_0(\tau)e^{\beta(\theta)}\right)d\theta.
\end{align}
Let us assume that for any sequence $\tau_n\to \bar{\tau},$ there exists a subsequence $\tau_{n_k}$ such that $y_0(\tau_{n_k}) \to \alpha.$ Now passing to the limit in \eqref{eqq 2.10}, we obtain
\begin{align*}
x-\gamma(\bar{\tau})=\int_{\bar{\tau}}^t f^{\prime}\left(\alpha e^{\beta(\theta)}\right)d\theta.
\end{align*}

By the uniqueness of $y_0$, we have $y_0(\bar{\tau})=\alpha.$ This proves the continuity of $y_0.$\\
\noindent{\bf Step 3:} First, let us define a set
\begin{align*}
  H_{\tau_0}:= \left\{\tau \in [\tau_0, t)\,\,\Big|\,\,\textnormal{there exists a $h$-curve in $Q$ joining $(x, t)$ to $(\gamma(\tau), \tau)$}\right\}.
\end{align*}
Clearly $H_{\tau_0}\neq \phi$ by Step 1. Moreover, let $\tau^{\prime}=\displaystyle{\inf_{\tau}} \,H_{\tau_0}.$ If $\tau^{\prime}>\tau_0,$ using the continuity of $y_0$ we conclude that there exists a point $(0, \tau_1)$ on the boundary such that the {\em h-curve} joining $(x,t)$ to $(\gamma(\tau^{\prime}), \tau^{\prime})$ will touch the boundary at $(0, \tau_1).$ Then by \cref{min h curve}, we obtain
\begin{align*}
    \int_{\tau_1}^t e^{-\beta(\theta)} f^*\left(f^{\prime}\left(h(x,t, \tau_1)e^{\beta(\theta)}\right)\right)d\theta + W(0, \tau_1)\leq J(\gamma, x, t),
\end{align*}
which implies $B(\tau_1, x, t)\leq J(\gamma, x, t)$ and hence
\begin{align*}
    B(x, t)=\inf_{\tau_1 \in H_B}B(\tau_1, x, t)\leq J(\gamma, x, t).
\end{align*}
If $\tau^{\prime}=\tau_0,$ in that case $\tau_0=\tau$ and the {\em h-curve} joining $(x, t)$ to $(0, \tau)$ completely lies in $Q.$
Then by the \cref{min h curve}
$$\int_{\tau}^t e^{-\beta(\theta)} f^*\left(f^{\prime}\left(h(x,t, \tau)e^{\beta(\theta)}\right)\right)d\theta +W(0, \tau)\leq J\left(\gamma, x, t\right)$$
and again we conclude that $B(x,t) \leq J(\gamma, x,t).$ Therefore from the \cref{defnvalue}, we have 
\begin{align*}
    B(x,t)=\inf_{\gamma \in \Gamma(x,t)} J(\gamma, x,t)=W(x,t).
\end{align*}

\noindent{\bf Case 2.} Let $\gamma(\theta)>0$ for all $0<\theta\leq t,$ then by the previous analysis one can infer that either there exists
a {\em h-curve} joining $(x, t)$ to $(y, 0)$ that touches $t$-axis at $(0, \tau_1)$
or the {\em h-curve} joining $(x,t)$ to $(y, 0)$ completely lies in $Q$. Therefore, either
\begin{align*}
\int_{\tau_1}^t e^{-\beta(\theta)} f^*\left(f^{\prime}\left(h(x,t, \tau_1)e^{\beta(\theta)}\right)\right)d\theta +W\left(0, \tau_1\right)\leq J\left(\gamma, x, t\right)
\end{align*}
or
\begin{align*}
\int_{0}^t f^{\prime}\left(h(x-y,t)e^{\beta (\theta)}\right)d\theta + U_0 (y)\leq J\left(\gamma, x, t\right).
\end{align*}
This implies either  $B(x, t, \tau_1)\leq J(\gamma, x, t)$ or $A(y, x, t)\leq J(\gamma, x, t).$ 
Therefore we conclude 
\begin{align*}
W(x, t)=\min\left\{A(x, t), B(x,t)\right\}.
\end{align*}
This completes the proof.
\end{proof}
Next, we characterize the value function $W(x,t)$ at the boundary.
\begin{lemma}[\textbf{Characterization of $W(0, t)$}]\label{5.1}
The value function $W(0, t)=\min_{\upgamma}J(\upgamma, 0, t)$ is achieved and if $W(0, t)=J(\gamma, 0, t),$ then the minimizing curve $\gamma$ has the following property: If $(t_1, t_2)\subset\{\theta: \gamma(\theta)\neq 0\},$ then $\gamma:(t_1, t_2)\to (0, \infty)$ is  a $h$-curve.
\end{lemma}
 \begin{proof} First, we show that for any Lipschitz curve $\gamma_0$ joining $(0,t)$ to $(y,0)$, there exists a curve $\eta_0$ with the same endpoints such that $|\eta_0^{\prime}(\theta)|\leq C_0,$ where $C_0$ is independent of the Lipschitz constant of $\gamma_0$ and 
 $J(\eta_0, x, t))\leq J(\gamma_0, x, t).$ An explicit construction of the curve $\eta_0$ is as follows: for each Lipschitz curve $\eta: [0,t]\to [0, \infty),$ we define the set
 \[D_{\eta}= \left\{ \theta \in [0, t] \,\, \Big|\,\, \eta(\theta) \neq \gamma_0(\theta)\right\}.\]
 Now consider the set of all Lipschitz curve $\eta:[0,t]\to [0, \infty)$ such that $\eta(t)=0$ and $\eta(0)=y,$ and define
 \begin{align*}
 S= \Big\{\eta\,\,\Big|\,\, |\eta^{\prime} (t_0)|\leq M \,\,\textnormal{for any}\, t_0 \in D_{\eta} \,\,  \textnormal{and $t_0$ be a Lebesgue point of $\eta^{\prime} (\theta)$} \Big\},
 \end{align*}
 where 
 \begin{align*}
 M= \max\left(\max_{\xi_1, \xi_2 \in [0,t]} f^{\prime}\left( e^{\beta(\xi_1)-\beta(\xi_2)}\lambda_f\right), \,\, \text{Lipschitz constant of}\,\, \gamma_0\right).
 \end{align*}

Let $\eta_{i} \in S$ be a sequence such that $$\lim_{i\to \infty} J(\eta_i, x,t)=\inf_{\eta \in S}J(\eta, x, t).$$ 
 Since $\eta_{i} \in S,$ $|\eta_i ^{\prime}(\theta)|\leq M$ \emph{a.e.} and by Arzela-Ascoli, we have a subsequence (still denoted by $\eta_i$) such that $\eta_i \to \eta_0$ uniformly. The proof of the facts that $\eta_0$ satisfies $|\eta_0^{\prime}(\theta)|\leq C_0,$ and $J(\eta_0, x, t))\leq J(\gamma_0, x, t)$ are divided into three steps.
 
\noindent{\bf Step 1:}  We first prove that $\eta_0 \in S.$ To show this, consider  $(a, b) \subset \{\theta \in [0, t]:\eta_0 (\theta) \neq \gamma_0(\theta)\}.$  Then there exists a $i_0 \in \mathbb{N}$ such that $\eta_i(\theta)\neq\gamma_0(\theta),$ for $\theta \in (a, b)$ and  $i\geq i_0.$ Therefore
$|\eta_i ^{\prime} (\theta)| \leq M$ for almost every $\theta \in (a, b).$
Now for any test function $\varphi,$ we have
\begin{equation*}
\begin{aligned}
&\Big|\int_{a}^{b} \eta_i^{\prime}(\theta) \varphi (\theta) \Big| \leq M \int_{a}^{b} |\varphi(\theta)| d\theta
\implies& \Big|\int_{a}^{b} \eta_i(\theta) \varphi ^{\prime} (\theta) \Big| \leq M \int_{a}^{b} |\varphi(\theta)|d\theta\\
\implies& \Big|\int_{a}^{b} \eta_0(\theta) \varphi ^{\prime} (\theta) \Big| \leq M \int_{a}^{b} |\varphi(\theta)|d\theta
\implies& \Big|\int_{a}^{b} \eta_0 ^{\prime}(\theta) \varphi  (\theta) \Big| \leq M \int_{a}^{b} |\varphi(\theta)|d\theta.
\end{aligned}
\end{equation*}
Let $t_0 \in (a, b)$ is a Lebesgue point of $\eta_0$. Now choose $\varphi$ as a suitable \emph{Friedrichs mollifier} converging to $\delta_{t_0}$ to arrive at the result that  $|\eta_0 ^{\prime} (t_0)| \leq M.$ This proves the claim.\\
\noindent
{\bf Step 2:} Next we show that $J(\eta_0, x, t)=\inf_{\eta \in S}J(\eta, x, t).$ Note that $f^{*}$ satisfies 
$$f^{*} (x)\geq f^{*} (y)+\big(f^{*}\big)^{\prime} (y) (x-y).$$
Taking $x=\eta^{\prime}_i (\theta)$ and $y=\eta^{\prime}_0 (\theta),$ multiplying with $e^{- \beta(\theta)}$ and integrating over ${\eta_i \neq 0}$ we get 
\begin{equation*}
    \begin{aligned}
    \int_{\eta_i \neq 0}f^{*} (\eta^{\prime}_i (\theta))e^{-\beta(\theta)} d\theta &\geq \int_{\eta_i \neq 0} f^{*} (\eta^{\prime}_0 (\theta))e^{-\beta(\theta)} d\theta+ \int_{\eta_i \neq 0}\big(f^{*}\big)^{\prime} (\eta^{\prime}_0 (\theta)) (\eta^{\prime}_i (\theta)-\eta^{\prime}_0 (\theta))e^{-\beta(\theta)} d\theta.
    \end{aligned}
\end{equation*}
Again adding the $-\int_{\eta_i= 0}f(\bar{u_b}(\theta))e^{-\beta(\theta)} d\theta +W_0(\eta_i(0)) $ on both sides we get
\begin{align}\label{mimi3.13}
   & \int_{\eta_i \neq 0}f^{*} (\eta^{\prime}_i (\theta))e^{-\beta(\theta)} d\theta -\int_{\eta_i= 0}f(\bar{u_b}(\theta))e^{-\beta(\theta)} d\theta + W_0(\eta_i(0))\nonumber \\
    &\geq\int_{\eta_i \neq 0} f^{*} (\eta^{\prime}_0 (\theta))e^{-\beta(\theta)} d\theta+ \int_{\eta_i \neq 0}f^{* \prime} (\eta^{\prime}_0 (\theta)) (\eta^{\prime}_i (\theta)-\eta^{\prime}_0 (\theta))e^{-\beta(\theta)}d\theta\nonumber\\
    &-\int_{\eta_i= 0}f(\bar{u_b}(\theta))e^{-\beta(\theta)} d\theta + W_0(\eta_i(0)).
    \end{align}
As $\eta^{\prime}_{i}$ convergence weakly to $\eta^{\prime}_{0}$, applying the $\liminf$ as $i \to \infty$ in the above inequality \eqref{mimi3.13} yields

$$\inf_{\eta \in S}J(\eta, x, t)\geq \int_{\eta_0 \neq 0}f^{*} (\eta^{\prime}_0 (\theta))e^{-\beta(\theta)}d\theta -\int_{\eta_0= 0}f(\bar{u_b}(\theta))e^{-\beta(\theta)} d\theta + W_0(\eta_0(0))=J(\eta_0, x,t).$$ 

\noindent{\bf Step 3:} In this final step, we will show that $|\eta_0 ^{\prime}(\theta)| \leq C_0$ \emph{a.e.} on $[0, t],$ where $C_0$ is independent of the Lipschitz constant of $\gamma_0,$ more precisely, 
\begin{align*}
C_0=  \max\left\{\max_{\xi_1, \xi_2 \in [0,t]} f^{\prime}\left( e^{\beta(\xi_1)-\beta(\xi_2)}\lambda_f\right), \,\, \max \left(\left|f^{\prime}\left(y_0 e^{\displaystyle{\max_{\theta \in [0, t]}}\beta(\theta)}\right)\right|, \left|f^{\prime}\left(y_0 e^{\displaystyle{\min_{\theta \in [0, t]}}\beta(\theta)}\right)\right|\right)\right\},
\end{align*}
where
\begin{align*}
    \left|y_0\right| \leq \max\left\{\frac{\left(f^{\prime}\right)^{-1}\left(y/t\right)}{e^{\displaystyle{\max_{\theta\in [0, t]}} \beta(\theta)}}, \frac{\left(f^{\prime}\right)^{-1}\left(y/t\right)}{e^{\displaystyle{\min_{\theta\in [0, t]}} \beta(\theta)}}\right\}.
\end{align*}
Suppose this is not the case. Then there exists an interval $[a, b]\subset [0, t]$ where $\eta_0(\theta)\neq 0$ and a small $\delta>0$ such that 
\begin{align*}
    \left|\left\{\theta \in (a, a+\delta)\,\, \Big|\,\, \left|\eta_0^{\prime}(\theta)\right|>C_0\right\}\right|>0.
\end{align*}
Then for some $\theta_1 \in (a, b),$ there exists a {\em h-curve} $h_1$ joining $(\eta_0(a), a)$ and $(\eta_0(\theta_1), \theta_1)$ that either lies completely in $Q$ or touches the $t$-axis at some point $(0, \theta_2).$
If $h_1$ touches the boundary, noting that $h_1^{\prime}(\theta)= f^{\prime}\left(y_0 e^{\beta(\theta)}\right)$ for some $y_0$ and $f^{\prime}(y_0 e^{\beta(\theta_2)})=0,$ we find $y_0= \lambda_f e^{-\beta(\theta_2)}$ which gives $|h_1^{\prime}(\theta)|\leq C_0.$ Now the new curve $\eta_1$ defined as
\begin{align*}
    \eta_1(\theta)= \begin{cases}
    h_1(\theta) ,\,\, a \leq \theta \leq  \theta_1\\
    \eta_0 (\theta),\,\, \theta_1 \leq \theta \leq  b
    \end{cases}
\end{align*}
is an element of $S$ and satisfies $J(\eta_1, x,t) <J(\eta_0, x,t)$ which is a contradiction to the previous step that $\eta_0$ is the minimizer.

If $h_1$ completely lies in $Q,$ then we have a {\em h-curve} joining  $(0, t)$ to $(\eta_0 (a),a)$ that does not touch the boundary at any point $\theta\in [a, t]$ and completely lies in $Q.$

Now if there exists a point $0\leq \theta_3< a,$ and a {\em h-curve} $h_2$ such that it joins $(0, t)$ to $( \eta_0 (\theta_3),\theta_3)$ and touches $t$-axis at $\theta_4,$ then using the previous argument we get a contradiction. If such a {\em h-curve} $h_2$ does not exist, we can further find a point below $\theta_4$ and a {\em h-curve} that joins $(0, t)$ to the corresponding point on the curve $\eta_0$ and touches the boundary. If the above process does not give any $\theta_i$ for $i=2, 4,$ then we have 
a {\em h-curve}, let say, $h_3$ joining $(0,t)$ and $(y, 0).$   In that case $h_3\in S,$  and $|h_3^{\prime} (\theta)|\leq C_0,$ and furthermore  $J(h_3, 0,t) <J(\eta_0, 0,t)$ leading to the contradiction again.

Therefore, considering the set of Lipschitz curves,
\begin{align*}
D_{C_0}=\Big\{\eta :[0,t] \to [0, \infty)\,\,\, \textnormal{Lipschitz curve}\,\,\Big| \,\,|\eta^{\prime}(\theta)|\leq C_0\,\,\, \textnormal {for \emph{a.e}}\,\,\, \theta \in [0,t] \Big\},
\end{align*} we obtain 
$$\inf_{\eta }J(\eta, 0, t)=\inf_{D_{C_0}} J(\eta, 0, t).$$ 
Then following Step 1 to Step 2 we get the existence of a minimizer $\gamma$. Now applying the \cref{min h curve}, we conclude the proof of the theorem.
\end{proof}
\begin{remark}\label{rem 1.1}
If $\bar{u}_b$ satisfies  $f(\lambda_f e^{\beta(t)}) \leq f(\bar{u}_b(t))$ for almost every $t>0,$ then $W(x,t)$ has the following form:
\begin{equation}\label{r1.2}
\begin{aligned}
W(x, t)&= \min_{t\geq t_2\geq t_1\geq 0, y} \int_{t_2}^t e^{-\beta(\theta)} f^*(f^{\prime}(h(x,t, t_2)e^{\beta(\theta)}))d\theta-\int_{t_1}^{t_2}f(\bar{u_b}(\theta))e^{-\beta(\theta)}d\theta\\ &+\int_0^{t_1} e^{-\beta(\theta)}f^*(f^{\prime}(h(-y,t_1)e^{\beta(\theta)}))d\theta +\int_0^y u_0(\theta)d\theta.
\end{aligned}
\end{equation}
In particular if $f^{\prime}(0)=0,$ $W(x,t)$ is of the form \eqref{r1.2}.
\end{remark}
%Proof of the remark\eqref{rem 1.1}:
\begin{proof}[Proof of Remark $1.2$]
It is enough to verify the following inequality for any curve $\gamma$ which satisfies $\gamma(t_1)=\gamma(t_2)=0,$ the following inequality holds:
\begin{equation}
    \int_{t_1}^{t_2} f^*(\dot{\gamma}(\theta))e^{-\beta(\theta)}d\theta \geq -\int_{t_1}^{t_2}f(\bar{u_b}(\theta))e^{-\beta(\theta)}d\theta.
    \label{ineqrem-1}
\end{equation}
This forces the minimum to be held along the boundary. The above inequality \eqref{ineqrem-1} follows from the following inequality:
\begin{equation*}
\begin{aligned}
    \int_{t_1}^{t_2} f^*(\dot{\gamma}(\theta))e^{-\beta(\theta)}d\theta
    = \int_{t_1}^{t_2} \sup_{v\in \mathbb{R}}[v\dot{\gamma}(\theta)-f(v)]e^{-\beta(\theta)}d\theta
    \geq -\int_{t_1}^{t_2}f(\lambda_f e^{\beta(\theta)})e^{-\beta(\theta)}d\theta.
\end{aligned}
\end{equation*}
In the above, we have chosen $v=\lambda_f e^{\beta(\theta)}$ in the second term to arrive at the last term. This proves the \cref{rem 1.1}

In particular if $f^{\prime}(0)=0,$ the inequality $f(\lambda_f e^{\beta(t)}) \leq f(\bar{u}_b(t))$ reduces to $f(\bar{u_b}(\theta)) \geq f(0)$ and the latter holds true as $f^{\prime}$ changes sign at $0$.
\end{proof}
The following example demonstrates that the minimization can be achieved by a curve consisting of many pieces of {\em h-curves}. We consider the boundary data $u_b$ to be a step function. Note that $u_b$ is a bounded measurable function but not a function of bounded variation in any interval containing zero. Such pathological examples can also be constructed for the BV boundary datum.
\begin{example}\label{example2}
Consider the balance law
\begin{align*}
\begin{cases}
&u_t + \Big(\frac{(u-60)^2}{2}\Big)_x=\alpha(t)u,\\
&u(x,0)=0,\,\,\,u(0, t)= u_b(t).
\end{cases}
\end{align*}
We set $\alpha(t)$ and $u_b(t)$ as:
\begin{align*}
\alpha(t)=\begin{cases}
     \frac{\dot{\upgamma}(t)+60}{\Ddot{\upgamma}(t)} \, &\textnormal{if}\,\,t\in\left(0, \frac{1}{\pi}\right),\\
    0 &\textnormal{if}\,\, t\in\left(\frac{1}{\pi}, \infty \right).
\end{cases}
\end{align*}
 \begin{equation*}
     \begin{aligned}
    u_b(t)=\begin{cases}
    60\,\,\, &\textnormal{if}\,\,t\in \cup_{n=1}^{\infty}\left(\frac{1}{2n\pi}, \frac{1}{(2n-1)\pi}\right),\\
     4080 &\textnormal{if}\,\, t\in \cup_{n=1}^{\infty} \left(\frac{1}{(2n+1)\pi}, \frac{1}{2n\pi}\right),\\
    60 &\textnormal{if}\,\, t\in \left(\frac{1}{\pi}, \infty \right),
    \end{cases}
    \end{aligned}
    \end{equation*}
    where $n\in \mathbb{N}$ and  $\upgamma(\theta)=\left(\theta\sin \frac{1}{\theta}\right)^4.$
    We get
    \begin{equation*}
    \begin{aligned}
W \Big(0,\frac{1}{\pi}\Big) 	&=\inf_{\gamma\in \Gamma(0, \frac{1}{\pi})}\left\{\int_{\{\gamma(\theta)\neq 0\}}f^*(\dot{\gamma}(\theta))e^{-\beta(\theta)}d\theta-\int_{\{\gamma(\theta)= 0\}}f(\bar{u_b}(\theta))e^{-\beta(\theta)}d\theta\right\}\\
&= \sum_{n\in \mathbb{N}}\int^{\frac{1}{(2n-1)\pi}}_{\frac{1}{2n\pi}}f^*(\dot{\upgamma}(\theta))e^{-\beta(\theta)}d\theta -
\sum_{n\in \mathbb{N}} \int^{\frac{1}{2n\pi}}_{\frac{1}{(2n+1)\pi}}f(u_b (\theta)) d\theta.
\end{aligned}
\end{equation*}
    Here $f(u)= \frac{(u-60)^2}{2},$ $f^ *(u)= \frac{u^2}{2}+ 60 u.$ Indeed, for any curve $\gamma$ that joins $(0,t)$ to $(y, 0)$, we get
\begin{align}\label{exampleverification1}
&\int_{\{\gamma(\theta)\neq 0\}}f^*(\dot{\gamma}(\theta))e^{-\beta(\theta)}d\theta-\int_{\{\gamma(\theta)= 0\}}f(\bar{u_b}(\theta))e^{-\beta(\theta)}d\theta\\
=&\sum_{n\in \mathbb{N}}\int_{\{\gamma(\theta)\neq 0\}\cap\big[\frac{1}{2n\pi}, \frac{1}{(2n-1)\pi }\big]}f^*(\dot{\gamma}(\theta))e^{-\beta(\theta)}d\theta+ \sum_{n\in \mathbb{N}}
\int_{\{\gamma(\theta)\neq 0\}\cap\big[\frac{1}{(2n+1)\pi}, \frac{1}{2n\pi }\big]}f^*(\dot{\gamma}(\theta))e^{-\beta(\theta)}d\theta\nonumber\\
&-\sum_{n\in \mathbb{N}}\int_{\{\gamma(\theta)=0\}\cap\big[\frac{1}{2n\pi}, \frac{1}{(2n-1)\pi }\big]}f(\bar{u_b}(\theta))e^{-\beta(\theta)}d\theta-\sum_{n\in \mathbb{N}}
\int_{\{\gamma(\theta)= 0\}\cap\big[\frac{1}{(2n+1)\pi}, \frac{1}{2n \pi }\big]}f(\bar{u_b}(\theta))e^{-\beta(\theta)}d\theta \nonumber\\
= &\sum_{n\in \mathbb{N}}\int_{\{\gamma(\theta)\neq 0\}\cap\big[\frac{1}{2n\pi}, \frac{1}{(2n-1)\pi }\big]}f^*(\dot{\gamma}(\theta))e^{-\beta(\theta)}d\theta +\sum_{n\in \mathbb{N}}\int_{\{\gamma(\theta)=0\}\cap\big[\frac{1}{2n\pi}, \frac{1}{(2n-1)\pi }\big]}f^*(\dot{\gamma}(\theta))e^{-\beta(\theta)}d\theta \nonumber\\
&+ \sum_{n\in \mathbb{N}}
\int_{\{\gamma(\theta)\neq 0\}\cap\big[\frac{1}{(2n+1)\pi}, \frac{1}{2n\pi }\big]}f^*(\dot{\gamma}(\theta))e^{-\beta(\theta)}d\theta-\sum_{n\in \mathbb{N}}
\int_{\{\gamma(\theta)= 0\}\cap\big[\frac{1}{(2n+1)\pi}, \frac{1}{2n \pi }\big]}f(\bar{u_b}(\theta))e^{-\beta(\theta)}d\theta \nonumber
\end{align}
where we used the fact that $f^*(\dot{\gamma}(\theta))=f(\overline{u}_b(\theta))=0$ in $(1/2n\pi, 1/(2n-1)\pi).$
Now invoking \cref{min h curve}, we obtain
\begin{align}\label{exampleverification2}
    &\sum_{n\in \mathbb{N}}\int_{\{\gamma(\theta)\neq 0\}\cap\big[\frac{1}{2n\pi}, \frac{1}{(2n-1)\pi }\big]}f^*(\dot{\gamma}(\theta))e^{-\beta(\theta)}d\theta +\sum_{n\in \mathbb{N}}\int_{\{\gamma(\theta)=0\}\cap\big[\frac{1}{2n\pi}, \frac{1}{(2n-1)\pi }\big]}f^*(\dot{\gamma}(\theta))e^{-\beta(\theta)}d\theta\nonumber\\ 
    =&\sum_{n\in \mathbb{N}}\int^{\frac{1}{(2n-1)\pi}}_{\frac{1}{2n\pi}}f^*(\dot{\gamma}(\theta))e^{-\beta(\theta)}d\theta
    \geq \sum_{n\in \mathbb{N}}\int^{\frac{1}{(2n-1)\pi}}_{\frac{1}{2n\pi}}f^*(\dot{\upgamma}(\theta))e^{-\beta(\theta)}d\theta.
\end{align}
Moreover, using \cref{rem 1.1}, we deduce
\begin{equation}\label{exampleverification3}
\begin{aligned}
& \sum_{n\in \mathbb{N}}
\int_{\{\gamma(\theta)\neq 0\}\cap\big[\frac{1}{(2n+1)\pi}, \frac{1}{2n\pi }\big]}f^*(\dot{\gamma}(\theta))e^{-\beta(\theta)}d\theta-\sum_{n\in \mathbb{N}}
\int_{\{\gamma(\theta)= 0\}\cap\big[\frac{1}{(2n+1)\pi}, \frac{1}{2n \pi }\big]}f(\bar{u_b}(\theta))e^{-\beta(\theta)}d\theta\\
&\geq -\sum_{n\in \mathbb{N}}\int_{\{\gamma(\theta)\neq 0\}\cap\big[\frac{1}{(2n+1)\pi}, \frac{1}{2n\pi }\big]}f(\bar{u_b}(\theta))e^{-\beta(\theta)}d\theta -\sum_{n\in \mathbb{N}}
\int_{\{\gamma(\theta)= 0\}\cap\big[\frac{1}{(2n+1)\pi}, \frac{1}{2n \pi }\big]}f(\bar{u_b}(\theta))e^{-\beta(\theta)}d\theta\\
&= -\sum_{n\in \mathbb{N}} \int^{\frac{1}{2n\pi}}_{\frac{1}{(2n+1)\pi}}f(u_b (\theta)) d\theta.
\end{aligned}
\end{equation}
Finally, inserting \eqref{exampleverification2} and \eqref{exampleverification3} in \eqref{exampleverification1}, we get 
\begin{equation*}
\begin{aligned}
&\int_{\{\gamma(\theta)\neq 0\}}f^*(\dot{\gamma}(\theta))e^{-\beta(\theta)}d\theta-\int_{\{\gamma(\theta)= 0\}}f(\bar{u_b}(\theta))e^{-\beta(\theta)}d\theta\\ &\geq \int_{\{\upgamma(\theta)\neq 0\}}f^*(\dot{\upgamma}(\theta))e^{-\beta(\theta)}d\theta-\int_{\{\upgamma(\theta)= 0\}}f(\bar{u_b}(\theta))e^{-\beta(\theta)}d\theta.
\end{aligned}
\end{equation*}
\end{example}

\begin{lemma}[\textbf{Lipschitz continuity of value function}]\label{LIP Value}
If the boundary data $u_b$ is a function of bounded variation in $[0, \infty)$, then the value function $W(x,t)$ is locally Lipschitz continuous for all $x> 0$ and $t> 0$ and continuous up to the boundary.
\end{lemma}\begin{proof}
First, we show that $W(0, t)$ is a locally Lipschitz continuous for $t>0$. For a fixed $t_0 \in (0, \infty),$ choose $t_1, t_2 \in (t_0-\delta, t_0+\delta)$  with $0<t_1<t_2$ for a small $\delta>0.$ Note that  $W(0, t_2)\leq W(0, t_1)- \int_{t_1}^{t_2} f(\bar{u}_b (\theta))e^{-\beta(\theta)} d \theta.$
This implies
\begin{align}\label{lipw1}
    W(0, t_2)-W(0, t_1) \leq -\int_{t_1}^{t_2} f(\bar{u}_b (\theta)) e^{-\beta(\theta)}d \theta.    \end{align}
For $s\in (t_0-\delta, t_0+\delta),$ using the above characterization in \cref{5.1}, we deduce that $W(0,s)$ is achieved by a Lipschitz curve $\gamma_s$ such that $\left|\dot \gamma_s (\theta)\right|< M$ for some $M>0.$ Then
\begin{align*}
W(0, t_2)= \int_{\gamma_{t_2}(\theta) \neq 0}f^{*} \left(\dot{\gamma}_{t_2}(\theta)\right) e^{-\beta(\theta)}d\theta -\int_{\gamma_{t_2}(\theta)=0}f(\bar{u_b}(\theta)) e^{-\beta(\theta)}d\theta + W_0(\gamma_{t_2}(0)).
\end{align*}
Define a new curve $\tilde{\gamma}:[0, t_1]\to [0, \infty)$ such that $\tilde{\gamma}(\theta)=\gamma_{t_2}(\theta+t_2- t_1).$
Now we estimate
\begin{align}\label{lipw2}
W(0, t_1)-W(0, t_2) &\leq J(\tilde {\gamma}, 0,t_1)-J(\gamma_{t_2}, 0, t_2)\\
&\leq \int_{\tilde{\gamma} (\theta)\neq 0}f^{*} (\dot{\tilde{\gamma}} (\theta)) e^{-\beta(\theta)} -\int_{\tilde{\gamma} (\theta)= 0}f(\bar{u_b}(\theta)) e^{-\beta(\theta)}d\theta + W_0 (\tilde{\gamma}(0))\nonumber\\
&- \int_{\gamma_{t_2} (\theta)\neq 0}f^{*} (\dot{\tilde{\gamma}} (\theta)) e^{-\beta(\theta)} +\int_{\gamma_{t_2} (\theta)= 0}f(\bar{u_b}(\theta)) e^{-\beta(\theta)}d\theta - W_0 (\gamma_{t_2}(0))\nonumber\\
&=\int_{\gamma_{t_2}(\theta)\neq 0}f^*(\dot{\gamma}_{t_2}(\theta))\left[e^{-\beta(\theta+t_1-t_2)}-e^{-\beta(\theta)}\right]d\theta \nonumber\\
&-\int_{\gamma_{t_2}(\theta)=0}\left[f(\bar{u}_b(\theta+t_1-t_2))e^{-\beta(\theta+t_1-t_2)}-f(\bar{u}_b(\theta))e^{-\beta(\theta)}\right]d\theta \nonumber\\
&+W_0(\gamma_{t_2}(t_1-t_2))-W_0(\gamma_{t_2}(0)):=I_1+I_2+I_3\nonumber
\end{align}
Since $\dot\gamma_{t_2}$ is bounded, $I_1 \leq C_1\left|t_1-t_2\right|.$ 
Now to estimate $I_2,$ denote $g(\theta)=f(\bar{u}_b(\theta))e^{-\beta(\theta)}.$ Note that $g$ is a function of bounded variation. Therefore $g$ can be written as $g=g_1 -g_2$ where $g_1$ and $g_2$ are bounded monotone increasing functions. We estimate
\begin{align}
|I_2|\leq &\left|\int_{\gamma_{t_2}(\theta)=0}\left[f(\bar{u}_b(\theta+t_1-t_2))e^{-\beta(\theta+t_1-t_2)}-f(\bar{u}_b(\theta))e^{-\beta(\theta)}\right]d\theta\right |\nonumber\\
&\leq \int_{t_2-t_1}^{t_2}\left|f(\bar{u}_b(\theta+t_1-t_2))e^{-\beta(\theta+t_1-t_2)}-f(\bar{u}_b(\theta))e^{-\beta(\theta)}\right| d\theta \nonumber\\
&=\int_{t_2-t_1}^{t_2}\left|g(\theta+t_1-t_2)-g(\theta)\right|d\theta \nonumber\\
&=\int_{t_2-t_1}^{t_2}\left|g_1(\theta+t_1-t_2)-g_2(\theta+t_1-t_2)-g_1(\theta)+g_2(\theta)\right|d\theta \nonumber\\
&\leq\int_{t_2-t_1}^{t_2}\left|g_1(\theta+t_1-t_2)-g_1(\theta)\right|d\theta+\int^{t_2}_{t_2-t_1}\left|g_2(\theta+t_1-t_2)-g_2(\theta)\right|d\theta \nonumber\\
&=\int_{t_2-t_1}^{t_2}\left(g_1(\theta)-g_1(\theta+t_1-t_2)\right)d\theta+\int^{t_2}_{t_2-t_1}\left(g_2(\theta)-g_2(\theta+t_1-t_2)\right)d\theta\nonumber \\ 
&=\int_{t_2-t_1}^{t_2}g_1(\theta)d\theta-\int_{0}^{t_1}g_1(\theta)d\theta+\int_{t_2-t_1}^{t_2}g_2(\theta)d\theta-\int_{0}^{t_1}g_2(\theta)d\theta\nonumber\\
%&=\int_{t_2-t_1}^{t_2}g_1(\theta)d\theta-\int_0^{t_2-t_1}g_1(\theta)d\theta-\int_{t_2-t_1}^{t_1}g_1(\theta)d\theta\nonumber\\
%&+\int_{t_2-t_1}^{t_2}g_2(\theta)d\theta-\int_0^{t_2-t_1}g_2(\theta)d\theta-\int_{t_2-t_1}^{t_1}g_2(\theta)d\theta \nonumber\\
&=\int_{t_1}^ {t_2} \left(g_1 (\theta) +g_2 (\theta)\right) d\theta - \int_{0}^{t_2 -t_1} \left(g_1 (\theta) +g_2 (\theta)\right) d\theta\nonumber
\end{align}
Since $g_1$ and $g_2$ are bounded, there exists $C_2$ such that $|I_2|\leq C_2|t_1-t_2|$. 
Again the Lipschitz continuity of $W_0,$ we obtain $|I_3| \leq C_3|t_1-t_2|$. Therefore combining the estimates \eqref{lipw1}-\eqref{lipw2}, we have
\begin{align*}
   |W(0, t_1)-W(0, t_2)|\leq C |t_1-t_2|. 
\end{align*}
When $W(x,t)=A(x,t),$ the proof of Lipschitz continuity is available in \cite{manish}. We only provide the proof when $W(x,t)=B(x,t).$
First, we prove that $B(x,t)$ is Lipschitz continuous in $t>0.$ Similar to the above, we fix a point $t_0\in (0, \infty)$ and choose $t_1, t_2 \in (t_0-\delta, t_0+\delta)$ with $0<t_1<t_2,$
and also for  $s\in (t_0-\delta, t_0+\delta)$, there exists a {\em h-curve} $\gamma_s$ along which $B(x, t)$ achieves its minimum. As $\dot \gamma_s$ is uniformly bounded, there exists $M>0$ such that $|f^*\left(\dot\gamma_s\right)|<M$. Let $\gamma_{t_1}:[\tau, t_1]\to [0, \infty)$ is a Lipschitz continuous curve such that $\gamma_{t_1}(t_1)=x,\, \gamma_{t_1}(\tau)=0$ and
\begin{align*}
    B(x,t_1):=J(\gamma_{t_1}, x, t)=\int_{\tau}^{t_1}f^*\left(\dot{\gamma}_{t_1}(\theta)\right)e^{-\beta(\theta)}d\theta+W(0, \tau).
\end{align*}
Moreover, we define $\tilde{\gamma}:[\tau+(t_2-t_1), t_2]\to [0, \infty) $ as follows:
\begin{align*}
\tilde{\gamma}(\theta)=\gamma_{t_1}(\theta+t_1-t_2)
\end{align*}
such that $\tilde{\gamma}(t_2)=x.$ Then we have the following inequality 
\begin{align*}
    B(x,t_2)-B(x,t_1)\leq J(\tilde{\gamma}, x, t_2)-J(\gamma_{t_1}, x, t_1)
\end{align*}
which implies
\begin{align*}
    B(x,t_2)-B(x,t_1)\leq \int_{\tau}^{t_1} f^*(\dot{\gamma}_{t_1}(\theta))\left[e^{-\beta(\theta+t_2-t_1)}-e^{-\beta(\theta)}\right]d\theta + W(0, \tau+t_2-t_1)-W(0, \tau).
\end{align*}
Therefore using the Lipschitz continuity of $W(0, \tau),$ we have
\begin{align*}
    B(x,t_2)-B(x,t_1)\leq C_1 |t_1-t_2|.
\end{align*}
Similarly, to estimate $B(x,t_1)-B(x,t_2),$ we assume $B(x,t_2)=J(\gamma, x, t_2)$ and obtain
\begin{align*}
    B(x,t_1)-B(x,t_2)\leq C_2 |t_1-t_2|.
\end{align*}
Combining the above two inequalities, we complete the proof.

Next we shall show that $B(x,t)$ is locally Lipschitz for $x>0.$ Fix $x_0\in (0, \infty)$ and let $x_1, x_2 \in (x_0-\delta, x_0+\delta)$ with $x_1<x_2.$ Also, for any $s\in (x_0-\delta, x_0+\delta)$ there exists a {\em h-curve} $\eta_s$ along which $W(x,t)$ achieved its minimum.
Since $\dot \eta_s$ is uniformly bounded, there exists $M>0$ such that for $r\in(0,1),$ $|(f^*)^{\prime}(r \dot{\eta}_s)|<M.$ Let $\gamma_{x_2}:[\tau,t]\to [0, \infty)$ such that $\gamma_{x_2}(t)=x_2$  and 
\begin{align*}
B(x_2, t)=\int_{\tau}^t f^*(\dot{\gamma}_{x_2}(\theta))e^{-\beta(\theta)}d\theta +W(0, \tau).
\end{align*}
Define the new curve $\tilde{\gamma}:[\tau,t]\to [0, \infty) $ as: $\tilde{\gamma}(\theta)=x_1\gamma_{x_2}(\theta)/x_2.$
Then $\tilde{\gamma}(\tau)=0$ and $\tilde{\gamma}(t)=x_1.$ Now we estimate
\begin{align*}
B(x_1,t)-B(x_2,t)&\leq \int_{\tau}^t \left(f^*(\dot{\tilde{\gamma}}(\theta))-f^*(\dot{{\gamma}_{x_2}}(\theta))\right) e^{-\beta(\theta)}d\theta
=\int_{\tau}^t f^{*\prime}(\xi_{\theta})\left(x_1/x_2-1\right)\dot{\gamma}_{x_2}(\theta) e^{-\beta(\theta)}d\theta
\end{align*}
and since $\dot{\gamma}_{x_2}$ is bounded, we obtain
\begin{align}\label{Lip B_x}
B(x_1,t)-B(x_2,t) \leq C e^{\max_{\theta\in [\tau, t]}\left\{-\beta(\theta)\right\}}\left(x_1-x_2\right)\leq C_1|x_1-x_2|.
\end{align}
To estimate $B(x_2,t)-B(x_1,t),$ similarly we assume that $\gamma_{x_1}:[\tau, t]\to [0,\infty)$ is a {\em h-curve} such that $\gamma_{x_1}(t)=x_1$ and 
\begin{align*}
    B(x_1,t)=\int_{\tau}^t  e^{-\beta(\theta)}f^*(\dot{\gamma}_{x_1}(\theta))d\theta+W(0,\tau).
\end{align*}
Now define a curve
\begin{align*}
    \tilde{\gamma}(\theta)=\begin{cases}
    \gamma_{x_1}(\theta)+x_2-x_1 \,\,\,\, &\text{if}\,\,\, \theta \in [\tau, t],\\
    \theta-\tau+x_2-x_1    \,\,\,\, &\text{if}\,\,\, \theta \in [\tau-(x_2-x_1), \tau],
    \end{cases}
\end{align*}
and the following estimate holds:
\begin{align}\label{lip B_x}
    B(x_2,t)-B(x_1,t)\leq &\int_{\tau-(x_2-x_1)}^t e^{-\beta(\theta)} f^*(\dot{\tilde{\gamma}}(\theta))d\theta+ W(0,\tau-(x_2-x_1))\nonumber\\
    &-\int_{\tau}^t e^{-\beta(\theta)}f^*(\dot{\gamma}_{x_1}(\theta))d\theta-W(0,\tau)\nonumber\\
    &\leq \Big|\int_{\tau-(x_2-x_1)}^\tau e^{-\beta(\theta)}f^*(1) d\theta \Big|+|W(0,\tau-(x_2-x_1))-W(0,\tau)|\nonumber\\
    &\leq C_2|x_2-x_1|
\end{align}
where in the last inequality we used the Lipschitz continuity of $W(0,t).$
Combining the inequalities \eqref{Lip B_x}-\eqref{lip B_x} we get $B(\cdot, t)$ is Lipschitz continuous. 

Finally, one can follow a variant of the arguments as in Step $1$ to Step $3$ of \cref{5.1} to show that the value function $W(x,t)$ is continuous at the boundary $x=0.$
This completes the proof.
\end{proof}
\begin{lemma}[\textbf{Existence of minimizer}]\label{min exists} Let $A(y, x, t)$ and $B(\tau, x, t)$ are the initial and boundary functional, respectively given by \eqref{initial functional}-\eqref{boundary functional}. Then there exists $\bar{y} \in H_A(x,t)$ and $\bar{\tau}\in H_B(x,t)$ such that $A(x,t)=A(\bar{y}, x, t)$ and $B(x,t)=B(\bar{\tau}, x, t).$
\end{lemma}
\begin{proof}
Consider the sequence $y_n \in H_{A}(x,t)$ such that 
\begin{align*}
\int_0^t e^{-\beta(\theta)} f^*(f^{\prime}(h(x-{y_n},t)e^{\beta(\theta)}))d\theta \to A(x,t)\,\,\,\, \text{as}\,\,\, n \to \infty.
\end{align*}
From \eqref{y_0 bounded}, we note that $\{y_n\}$ is a bounded sequence. Hence there exists convergent subsequence $y_{n_k}$  of $y_n$ such that $y_{n_k}\to \bar{y}.$ By continuity of $h(\xi, x, t)$, we obtain $h(x-y_{n_k}, t) \to h(x-\bar{y}, t).$ Thus by dominated convergence, we have
\begin{align*}
  \int_0^t e^{-\beta(\theta)} f^*(f^{\prime}(h(x-y_{n_k},t)e^{\beta(\theta)}))d\theta +U_0(y_{n_k})
  \to
  \int_0^t e^{-\beta(\theta)} f^*(f^{\prime}(h(x-\bar{y},t)e^{\beta(\theta)}))d\theta +U_0(\bar{y})
\end{align*}
Since $W(0, \bar{\tau})$ is Lipschitz continuous, similar arguments can be used for the functional $B(\tau, x,t).$
\end{proof}
Next lemma shows the non-intersecting property for $W(x,t)=\min\{A(x,t), B(x,t)\}.$ When $W(x,t)=A(x,t),$ the non-intersecting property is already proved in \cite{manish}. We only consider the case when $W(x,t)=B(x,t).$
\begin{lemma}[\textbf{Non-intersecting property}]\label{NIP}
Given a fixed point $(x,t)$ with $x\geq0, t>0,$ let $W(x,t)=B(x,t)=B(\tau, x,t)$ for some point $\tau.$ Let 
$$B(\tau, x, t)=\int_{\tau}^t f^*(\dot{\gamma}(\theta))e^{-\beta(\theta)}d\theta +W(0,\tau),$$
where $\gamma$ is a $h$-curve that joins $(x, t)$ to $(0, \tau).$ Then for any point $(\bar{x}, \bar{t})$ on $\gamma,$ we have
\begin{align*}
    B(\bar{x},\bar{t})< J(\tilde{\gamma}, \bar{x}, \bar{t})
\end{align*}
where $\tilde{\gamma}\in Q$ is different from $\gamma$ in the interval $[\tau, \bar{t}]$. Consequently,  the minimizers corresponding to the points $(x_1, t)$ and $(x_2, t)$ where $x_1\neq x_2, $ will not intersect.
%Let for a fixed the point $(x, t),$  $x\geq0, t>0$, $W(x,t)=B(x,t)$ given by
%\[B(x,t)=\int_{\tau}^t f^*(\dot{\gamma}(\theta))e^{\beta(\theta)}))d\theta +W(0, {\tau}),\] 
%where $\gamma$ is the h-curve joining $(x, t)$ to $(0, \tau)$ with $\gamma(\theta)>0$ for $\tau< \theta\leq t.$  Suppose $(\bar{x}, \bar{t})$ is a point on $\gamma$, then
%\[\int_{\tau}^{\bar{t}} f^*(\dot{\gamma}(\theta))e^{\beta(\theta)}))d\theta +w(0, {\tau})< J({\tilde{\gamma}, \bar{x}, \bar{t}})\]
%for any $\tilde{\gamma}$ in $Q$ which is different from $\gamma$ in the interval $[\tau, \bar{t}]$. Consequently,  the minimisers corresponding to the points $(x_1, t)$ and $(x_2, t)$ where $x_1\neq x_2, $ will not intersect.
\end{lemma}
\begin{proof}
From the definitions of $W(x,t)$ and the boundary functional $B(\tau, x,t),$ we have 
\begin{align*}
    W(x,t)=\int_{\tau}^t f^*(\dot{\gamma}(\theta))e^{-\beta(\theta)}d\theta +W(0,\tau)
\end{align*}
where $\gamma$ is a {\em h-curve} and $\tau=\max\left\{\theta\,\Big|\,{\gamma}(\theta)=0 \right\}.$
%From the previous lemma, we have
%\[ W(x,t)=\int_{\tau}^t f^*(\dot{\gamma}(\theta))e^{-\beta(\theta)}d\theta +W(0,\tau)\] where $\gamma$ is a $h$-curve and $\tau=\max\{\theta|{\gamma}(\theta)=0 \}.$
Suppose at the point $(\bar{x}, \bar{t})$ on $\gamma,$ there exists a {\em h-curve}  $\tilde{\gamma}$ different from $\gamma$ and $\tau^{\prime}=\max\left\{\theta\,\Big|\,\tilde{\gamma}(\theta)=0\right\}$ such that 
\begin{align*}
W(\bar{x},\bar{t})=\int_{\tau^{\prime}}^{\bar{t}} f^*(\dot{\tilde{\gamma}}(\theta))e^{-\beta(\theta)}d\theta +W(0,\tau^{\prime}).
\end{align*}
Since $\gamma$ is the minimizer at $(x,t),$ we have
\begin{align*}
    W(x,t)=\int_{\tau}^t f^*(\dot{\gamma}(\theta))e^{-\beta(\theta)}d\theta +W(0,\tau) &=\int_{\tau}^{\bar{t}} f^*(\dot{\gamma}(\theta))e^{-\beta(\theta)}d\theta +W(0,\tau) +\int_{\bar{t}}^t f^*(\dot{\gamma}(\theta))e^{-\beta(\theta)}d\theta\\
   &\geq W(\bar{x}, \bar{t})+\int_{\bar{t}}^t f^*(\dot{\gamma}(\theta))e^{-\beta(\theta)}d\theta\\
   &=\int_{\tau^{\prime}}^{\bar{t}} f^*(\dot{\tilde{\gamma}}(\theta))e^{-\beta(\theta)}d\theta +W(0,\tau^{\prime})+\int_{\bar{t}}^t f^*(\dot{\gamma}(\theta))e^{-\beta(\theta)}d\theta.
\end{align*}
Then defining the following curve,
\begin{equation*}
\begin{aligned}
    \Gamma(\theta)=\begin{cases}
    \gamma(\theta) \,\,\,\,&\textnormal{if}\,\,\, \theta \in [\bar{t},\,\, t],\\
    \tilde{\gamma}(\theta)\,\,\,\,&\textnormal{if}\,\,\, \theta \in [\tau^{\prime}, \bar{t}],
    \end{cases}
    \end{aligned}
\end{equation*}
we get
\begin{align*}
W(x,t)\geq \int_{\tau^{\prime}}^{t} f^*(\dot{\Gamma}(\theta))e^{-\beta(\theta)}d\theta +W(0,\tau^{\prime}).
\end{align*}
%From the construction of Lemma \eqref{bdf}, we can construct a {\em h-curve} $\Gamma_h$ that touches the boundary and joins $(x,t)$ to $(\Gamma(\theta), \theta).$ Now take the supremum of all such points.

Now we define $\tau^{\prime\prime}$ as follows:
\begin{align*}
    \tau^{\prime\prime}=\inf_{\tau^{\prime}\leq \theta \leq t}\left\{\theta\,\,\, \big|\,\,\, \textnormal{there exists a {\em h-curve }$\Gamma_\theta $ in $Q$ joining}\,\, (x,t)\,\,\textnormal{to}\,\, (\Gamma(\theta), \theta) \right\}.
\end{align*}
From the construction as in \cref{bdf}, we have that $\tau^{\prime\prime}< \bar{t}$  and the {\em h-curve}  $\Gamma_ {\tau^{\prime\prime}}$ joining $(x, t)$ to $(\Gamma(\tau^{\prime\prime}), \tau^{\prime\prime})$ touches the boundary, i.e., the $t$-axis. 
Next we define a curve $\Tilde{\Gamma}$ by
\begin{align*}
\Tilde{\Gamma}(\theta)=\begin{cases}
\Gamma_{\tau^{\prime \prime}}(\theta) \,\,\,\,&\textnormal{if}\,\,\, \theta \in [\tau^{\prime\prime},\,\, t]\\
\Gamma(\theta)\,\,\,\,&\textnormal{if}\,\,\, \theta \in [\tau^{\prime}, \tau^{\prime\prime}].
\end{cases}
\end{align*}
Since $\tau^{\prime\prime}< \bar{t},$ the \cref{min h curve} gives
 \[W(x,t)=\int_{\tau^{\prime}}^{t} f^*(\dot{\Gamma}(\theta))e^{-\beta(\theta)}d\theta +W(0,\tau^{\prime})>\int_{\tau^{\prime}}^{t} f^*(\dot{\Tilde{\Gamma}}(\theta))e^{-\beta(\theta)}d\theta +W(0,\tau^{\prime}),\]
 which is a contradiction. This completes the proof of non-intersecting property.
 \end{proof}
Next, we collect some properties of $y_*, y^*, \tau_*, \tau^*$ in the following lemma.
\begin{lemma}\label{lnew}
With the definitions \eqref{definition A}-\eqref{definition B}, we have the following:
\begin{enumerate}
\item $\tau_{*}(x,t)$ and $\tau^{*}(x,t)$ are, for fixed $x$, monotonically increasing in $t$ and for fixed $t$ monotonically decreasing in $x$. Moreover we have for $t_1<t_2$ that $\tau^{*}(x,t_1)\leq\tau_{*}(x,t_2)$ and for $x_1<x_2$ that $\tau_{*}(x_1,t)\geq\tau^{*}(x_2,t)$.
\item $y_{*}(x,t)$ and $y^{*}(x,t)$ are, for fixed $t$ monotonically increasing in $x$ and for $x_1<x_2$ we have $y^{*}(x_1,t)\leq y_{*}(x_2,t)$.
\item $y_{*}(x,t)$ is lower semicontinuous and $y^{*}(x,t)$ is upper semicontinuous.
\item $\tau_{*}(x,t)$ is lower semicontinuous and $\tau^{*}(x,t)$ is upper semicontinuous.
\end{enumerate}
\end{lemma}
\begin{proof}
Proof of these properties follows from the non-intersecting property.
\end{proof}
\subsection{Proof of \cref{thm explicit formula}} The proof of \cref{thm explicit formula} is divided into two parts. In the first part, we derive the explicit representation of the solution and in the second part, we show that the explicit formula satisfies the boundary condition \eqref{boundary condition}.
 \subsubsection{Derivation of the explicit formula}\label{subsection 1}
 First of all, we observe that in the case $B(x,t)\leq A(x,t),$ i.e., when the minimum is achieved by the boundary functional, by virtue of the \cref{DPP} (dynamic programming principle) one can find a time level $t=\tau$ such that the minimum is achieved by the initial functional $A(y(x,t, \tau), x,t).$ Keeping this idea in mind, we will present a proof that essentially works for both of the cases whether $A(x,t) \leq B(x,t)$ or $B(x,t) \leq A(x,t)$ for any fixed $(x,t)\in Q.$
 
\noindent\textbf{Step 1:} %Suppose $$W(x,t)=A(x, t)= A(x, y(x,t), t).$$ 
Let $(x, t)$ be a point of continuity for $y_*, \tau_*, y^*$ and $\tau^*.$
Using the dynamics programming principle (\cref{DPP}), there exist a level $\tau>0$ and a neighbourhood $N_A$ of $x$ and $N_B$ of $t$ such that for $(y,s)\in N_A\times N_B,$ $W(y,s, \tau)$ can be expressed as follows:
$$W(y, s, \tau)= \min_ {z\in H_A(y,s, \tau) }\Big\{ W(z, \tau)+ \int_{\tau}^s f^*\left(f^{\prime}\left(h\left(y-z,s,\tau\right)e^{\beta(\theta)}\right)\right)e^{-\beta(\theta)}d\theta\Big\}.$$
Define $$ W(z, y, s, \tau)=  W(z, \tau)+ \int_{\tau}^s f^*\left(f^{\prime}\left(h\left(y-z,s,\tau\right)e^{\beta(\theta)}\right)\right)e^{-\beta(\theta)}d\theta.$$
Let $x_1<x<x_2$ and $x_1, x_2 \in N_A,$ now we calculate
\begin{align*}
& W(x_1,t, \tau)-W(x_2,t, \tau)\\
&= W(z(x_1, t, \tau), x_1, t, \tau)-  W(z(x_2,t,\tau), x_2, t, \tau)\\
&= W(z(x_1, t, \tau), x_1, t, \tau)- W(z(x_1, t, \tau), x_2, t, \tau) + W( z(x_1, t, \tau), x_2, t, \tau)- W(z(x_2,t,\tau), x_2, t, \tau).
\end{align*}
As the second term is non-negative, we obtain
\begin{align} \label{ww4.11}
&W(x_1,t, \tau)-W(x_2,t, \tau)\nonumber\\
&\geq W(z(x_1, t, \tau), x_1, t, \tau)- W(z(x_1, t, \tau), x_2, t, \tau)\nonumber\\
&\geq \int_{\tau}^t \left[f^*\left(f^{\prime}\left(h(x_1-z(x_1, t, \tau),t,\tau)e^{\beta(\theta)}\right)\right)-f^*\left(f^{\prime}\left(h(x_2-z(x_1, t, \tau),t,\tau)e^{\beta(\theta)}\right)\right)\right]e^{-\beta(\theta)}d\theta.
\end{align}
Similarly, we have  
\begin{align*} 
&W(x_1,t, \tau)-W(x_2,t, \tau)\\
&= W(z(x_1, t, \tau), x_1, t, \tau)-  W(z(x_2,t,\tau), x_2, t, \tau)\\
&= W(z(x_1, t, \tau), x_1, t, \tau)- W(z(x_2, t, \tau), x_1, t, \tau) + W( z(x_2, t, \tau), x_1, t, \tau)- W(z(x_2,t,\tau), x_2,t, \tau).
\end{align*}
As the first term is non-positive, we get
\begin{align}\label{4.13}
&W(x_1,t, \tau)-W(x_2,t, \tau)\nonumber\\
&\leq W( z(x_2, t, \tau), x_1, t, \tau)- W(z(x_2, t, \tau), x_2, t, \tau)\nonumber\\
&\leq \int_{\tau}^t \left[f^*\left(f^{\prime}\left(h(x_1-z(x_2, t, \tau),t,\tau)\right)e^{\beta(\theta)}\right)-f^*\left(f^{\prime}\left(h(x_2-z(x_2, t, \tau),t,\tau)e^{\beta(\theta)}\right)\right)\right]e^{-\beta(\theta)}d\theta.
\end{align}
\noindent\textbf{Step 2:} In the second step we simplify both the expressions \eqref{ww4.11} and\eqref{4.13}. Let us calculate
\begin{align*}
&\frac{\partial}{\partial \xi}\left[e^{-\beta(\theta)}f^*\left(f^{\prime}\left(h(x-\xi,t, \tau)e^{\beta(\theta)}\right)\right)\right]\\
&=h(x-\xi, t, \tau)\frac{\partial}{\partial \xi} f^{\prime}\big(h(x-\xi,t, \tau)e^{\beta(\theta)}.
\end{align*}
Integrating the above expression from $x$ to $z$ with respect to the variable $\xi$,
\begin{align*}
e^{-\beta(\theta)}f^*\left(f^{\prime}\left(h(x-\xi,t, \tau)e^{\beta(\theta)}\right)\right)-e^{-\beta(\theta)}f^*\left(f^{\prime}\left(h(0,t, \tau)e^{\beta(\theta)}\right)\right)\\
=\int_{x}^{z} h(x-\xi, t, \tau)\frac{\partial}{\partial \xi} f^{\prime}\left(h(x-\xi,t, \tau)e^{\beta(\theta)}\right)d\xi.
\end{align*}
 Again integrating from $\tau$ to $t$ with respect to the variable $\theta$, we find
\begin{align}\label{equation2.4}
&\int_{\tau}^t e^{-\beta(\theta)}f^*\left(f^{\prime}\left(h(x-\xi,t, \tau)e^{\beta(\theta)}\right)\right)d\theta\nonumber\\
&=\int_{\tau}^t \int_{x}^z h(x-\xi, t, \tau)\frac{\partial}{\partial \xi} f^{\prime}\left(h(x-\xi,t, \tau)e^{\beta(\theta)}\right)d\xi d\theta+\int_{\tau}^t e^{-\beta(\theta)}f^*\left(f^{\prime}\left(h(0,t,\tau)e^{\beta(\theta)}\right)\right)d\theta\nonumber\\
&=\int_x^z \int_{\tau}^t  h(x-\xi, t, \tau) \frac{\partial}{\partial \xi} f^{\prime}\left(h(x-\xi,t, \tau)e^{\beta(\theta)}\right)d\theta d\xi+\int_{\tau}^t e^{-\beta(\theta)}f^*\left(f^{\prime}\left(h(0,t, \tau)e^{\beta(\theta)}\right)\right)d\theta\nonumber\\
&=-\int_x^z h(x-\xi,t, \tau)d\xi -\int_{\tau}^t e^{-\beta(\theta)}f\left(h(0, t, \tau)e^{\beta(\theta)}\right)d\theta,
\end{align} 
where we used the identity
$$\int_{\tau}^t e^{-\beta(\theta)}f^*\left(f^{\prime}\left(h(0,t, \tau)e^{\beta(\theta)}\right)\right)d\theta= -\int_{\tau}^t e^{-\beta(\theta)}f\left(h(0, t, \tau)e^{\beta(\theta)}\right)d\theta,$$
which follows by using the identity $f^*(f^{\prime}(p)= p f^{\prime}(p)-f(p)$ and the \cref{defn h curve}. Dividing by $x_1-x_2$ and simplifying the inequalities \eqref{4.13} and \eqref{ww4.11} using \eqref{equation2.4} , we obtain
\begin{align*}
    \frac{1}{x_1-x_2}\int_{x_2-z(x_1,t, \tau)}^{x_1-z(x_1, t, \tau)}h(p, t, \tau)dp \leq \frac{W(x_1,t, \tau)-W(x_2,t,\tau)}{x_1-x_2}\leq \frac{1}{x_1-x_2}\int_{x_2-z(x_2,t, \tau)}^{x_1-z(x_2, t, \tau)}h(p, t, \tau)dp.
\end{align*}
Then passing to the limit as $x_1, x_2 \to x$ and using the properties $z(\cdot,t, \tau)$ from the \cref{lnew},  we conclude
\[W_x=h(x-z(x,t, \tau), t, \tau).\]

\noindent\textbf{Step 3:} %In this step we shall prove $W_t=-e^{\beta(t)}f(e^{\beta(t)}w)$. 
Let $t_1<t<t_2$ and $t_1, t_2 \in N_B.$ Now we calculate
\begin{align*}
&W(x,t_1, \tau)-W(x_,t_2, \tau)\\
&= W(z(x, t_1, \tau), x, t_1, \tau)-  W(z(x,t_2,\tau), x, t_2, \tau)\\
&= W(z(x, t_1, \tau), x, t_1, \tau)- W(z(x, t_2, \tau), x, t_1, \tau) + W( z(x, t_2, \tau), x,  t_1, \tau)- W( z(x,t_2,\tau), x, t_2, \tau).
\end{align*}
As the first term is not positive, we get
\begin{align} \label{4.11}
W(x,t_1, \tau)-W(x,t_2, \tau)&\leq W( z(x, t_2, \tau), x, t_1, \tau)-W(z(x,t_2,\tau), x, t_2, \tau)\nonumber\\
&=\int_{\tau}^{t_1} f^*\left(f^{\prime}\left(h(x-z(x, t_2, \tau),t_1,\tau)e^{\beta(\theta)}\right)\right)e^{-\beta(\theta)}d\theta\nonumber\\ 
&- \int_{\tau}^{t_2}f^*\left(f^{\prime}\left(h(x_2-y(x, t_2, \tau),t_2,\tau)e^{\beta(\theta)}\right)\right)e^{-\beta(\theta)}d\theta.
\end{align}
  Dividing by $t_1-t_2$ and passing to the limit as $t_1, t_2 \to t$ in \eqref{4.11}  we get that the right-hand side of the above inequality is the derivative of 
  $$\int_{\tau}^{r} f^*\left(f^{\prime}\left(h(x-z(x, t_2, \tau),r,\tau)e^{\beta(\theta)}\right)\right)e^{-\beta(\theta)}d\theta$$ 
  with respect to $r$ evaluated at the point $t.$ Now we calculate
\begin{align}\label{derivative in t}
  & \frac{d}{dr}\left[\int_{\tau}^{r} f^*\left(f^{\prime}\left(h(x-z(x, t_2, \tau),r,\tau)e^{\beta(\theta)}\right)\right)e^{-\beta(\theta)}d\theta\right]\nonumber\\
  &=\frac{d}{dr} \left[h(x-z(x, t_2,\tau),r,\tau) (x-z(x, t_2,\tau))- \int_{\tau}^{r}  f\left(h(x-z(x, t_2, \tau),r,\tau))e^{\beta(\theta)}\right)e^{-\beta(\theta)}d\theta \right]\nonumber\\
  &=h_r (x-z(x, t_2,\tau),r,\tau) (x-z(x, t_2,\tau)) - f\left(h(x-z(x, t_2, \tau),r,\tau)e^{\beta(r)}\right)e^{-\beta(r)}\nonumber\\
  & - \int_{\tau}^{r}  f\left(h(x-z(x, t_2, \tau),r,\tau))e^{\beta(\theta)}\right) h_r \left(x-z(x, t_2,\tau),r,\tau\right)d\theta.
\end{align}
Therefore from \eqref{4.11} we get
\begin{equation*}
\limsup_{t_1, t_2 \to t}\frac{W(x,t_1, \tau)-W(x_,t_2,\tau)}{t_1 -t_2} \leq - f\Big(h(x-z(x, t, \tau),t,\tau))e^{\beta(t)}\Big)e^{-\beta(t)}.  
\end{equation*}
Similar calculation yields 
\begin{equation*}
\liminf _{t_1, t_2 \to t}\frac{W(x,t_1,\tau)-W(x_,t_2,\tau)}{t_1 -t_2} \geq - f\Big(h(x-z(x, t, \tau),t,\tau))e^{\beta(t)}\Big)e^{-\beta(t)}.  
\end{equation*}
The above two inequalities show that the derivative of $W$ with respect to $t$ exists and 
$$W_t=- f\Big(h(x-z(x, t, \tau),t,\tau))e^{\beta(t)}\Big)e^{-\beta(t)}. $$

\subsubsection{Verification of boundary condition}In this section we verify the boundary condition in the sense of \eqref{boundary condition}. Based upon the characterization of $W(0, t)$ in \cref{5.1}, we classify the boundary points into three types:
 \begin{definition}[\textbf{Classification of boundary points}]\label{classification} Let $t>0$ and $(0,t)$ be a point on the boundary $x=0.$ 
\begin{enumerate}
\item [1.] The point $(0, t)$ is said to be a boundary point of {\it Type-I} if 
$$W(0, t)= -\int_{\tau}^t f((\bar{u}_b(\theta))e^{-\beta(\theta)} d\theta + W(0, \tau)$$
for some $\tau<t.$
\item [2.] The point $(0, t)$ is said to be a boundary point of {\it Type-II} if 
$$W(0, t)= -\int_{\tau}^t f^{*}\left(f^{\prime}\left(h(0, t, \tau)e^{\beta(\theta)}\right)\right)e^{-\beta(\theta)} d\theta + W(0, \tau)$$
for some $\tau<t.$
\item [3.] The point $(0, t)$ is said to be a boundary point of {\it Type-III} if there exists a sequence of points $\{t_n\}$ satisfying $t_n < t_{n+1}$ and converging to $t$ such that
\begin{equation*}
    \begin{aligned}
    W(0, t)= \sum_{n\in \mathbb{N}}\int_{t_{2n-1}}^{t_{2n}}
f^*\left(f^{\prime}\left(h(0, t_{2n}, t_{2n-1})e^{\beta(\theta)}\right)\right)e^{-\beta(\theta)}d\theta 
-\sum_{n\in \mathbb{N}} \int_{t_{2n}}^{t_{2n+1}}f(\bar{u}_b (\theta)) d\theta + W(0, t_1).
    \end{aligned}
\end{equation*}
\end{enumerate}
\end{definition}

Now, we prove the weak form of boundary condition in the sense of Bardos-le Roux- N\'{e}d\'{e}lec.   
\begin{lemma}\label{BDC lemma}
The explicit solution $u$ to \eqref{equation 1}-\eqref{IBD} given in \cref{thm explicit formula}
satisfies the initial condition and a weak form of boundary condition in the sense of \eqref{boundary condition}.
\end{lemma}
\begin{proof}
To verify the initial condition, note that for $x>0,$ there exists an $\kappa>0$ such that for $t<\kappa,$  we have $A(x, t)<B(x,t).$ In that case $u(x,t)$ satisfies the initial condition (see \cite{manish}). Here, we prove the boundary condition. The proof of the boundary condition is divided into several cases depending on the classification of boundary points as well as the relation between $A(x,t)$ and $B(x,t)$ at the boundary. Furthermore, since $\tau_*(x,t)\not \to t$ as $x\to 0+$ always, one must consider the separate cases depending on this fact too. 
%We observe that if a boundary point is of \textit{Type-III}, then for a sequence $x_n \to 0,$ either $\tau_*(x_n, t) \in (t_{2n_k}, t_{2n_k+1}]$ or $\tau_*(x_n, t) \in (t_{2n_k-1}, t_{2n_k}].$ In the first case, we further note that either $\tau_*(x_n, t)\to t$ as $x_n \to 0$ or $\tau_*(x_n, t)\not \to t$ as $x_n \to 0.$ Now one needs to consider two cases: $A(0, t)\leq B(0, t)$ and $B(0, t)<A(0, t).$ If $A(0, t)\leq B(0, t),$ then we show that $f(u(0+, t))\geq f(\bar{u}_b),$ but our proof in this case specifically does not use the fact whether $\tau_*(x_n, t) \to t$ as $x_n \to 0$ or $\tau_*(x_n, t)\not \to t$ as $x_n \to 0.$  But when $B(0, t)< A(0, t),$ using $\tau_*(x_n, t)\to t$ as $x_n \to 0,$ we prove that $u(0+, t)=\bar{u}_b(t).$ In the second case we prove that $f(u(0+, t))=f(\bar{u}_b(t))$ which is independent of the relation between $A(0,t)$ and $B(0,t).$}\\

\noindent{\bf Case 1.} When $A(0, t) \leq B(0, t),$ first we assume that the {\em h-curve} $X(\theta)$ joining $(0, t)$  to $(y^{*}(0,t), 0)$ does not touch $t$-axis in the interval $(0, t)$.
We note that $u(0+, t)= e^{\beta(t)}h(-y^{*}(0,t), t)$ and $$X^{\prime}(\theta)=f^{\prime}\left(h(-y^{*}(0,t), t)e^{\beta(\theta)}\right).$$ 
Since $X(\theta)\geq 0$ and $X(t)=0,$ $X^{\prime}(t)\leq 0.$ This implies $f^{\prime}(u(0+, t))\leq 0.$ Note that
\begin{equation*}
\begin{aligned}
\min_{\tau \in H_B} B(\tau, x, t) \leq &\min_{t\geq t_2\geq t_1\geq 0, y} \int_{t_2}^t e^{-\beta(\theta)} f^*\left(f^{\prime}\left(h(x,t, t_2)e^{\beta(\theta)}\right)\right)d\theta-\int_{t_1}^{t_2}f(\bar{u_b}(\theta))e^{-\beta(\theta)}d\theta\\ &+\int_0^{t_1} e^{-\beta(\theta)}f^*\left(f^{\prime}\left(h(-y,t_1)e^{\beta(\theta)}\right)\right)d\theta +\int_0^y u_0(\theta)d\theta :=\min_{t\geq t_2\geq t_1\geq 0, y}  B(x,y,t,t_1, t_2).
\end{aligned}
\end{equation*}
Here the minimum is taken over the points $(y, t_1, t_2),$ wherever the expression  $B(x,y,t,t_1, t_2)$ is meaningful.
Now as in the proof of \cref{thm explicit formula}, we rewrite the functional $A(y,x,t)$ and $B(x,y,t,t_1, t_2)$ in the following equivalent form:
\begin{align*}
&A(y,x,t)=\int_0^y u_0(\theta)d\theta-\int_{\tau}^t e^{-\beta(\theta)}f\left(h(x-y,\theta)e^{\beta(\theta)}\right)d\theta-\int_0^{\tau}
e^{-\beta(\theta)}f\left(h(x-y,\tau)e^{\beta(\theta)}\right)d\theta\\
&+(x-y)h(x-y,\tau),\\
&B(x,y,t,  t_1, t_2)=\int_0^y u_0(\theta)d\theta-\int_{t_1}^{t_2}e^{-\beta(\theta)}f(\bar{u_b}(\theta))d\theta-\int_{\tau}^{t_1}e^{-\beta(\theta)}f\left(h(-y,\theta)e^{\beta(\theta)}\right)d\theta\\
&-\int_0^{\tau}e^{\beta(\theta)}f\left(h(-y,\tau)e^{\beta(\theta)}\right)d\theta-yh(-y,\tau)+\int_0^x h(\theta, t, t_2)d\theta -\int_{t_2}^t e^{-\beta(\theta)}f\left(h(0,t,t_2)e^{\beta(\theta)}\right)d\theta.
\end{align*}
We further notice that there exists a sequence $\varepsilon_n \to 0$ such that the expression $B(0, y^*(0,t), t, t-\varepsilon_n)$ makes sense.
Since the {\em h-curve} joining the points $(0,t)$ to $(y^*(0,t),t)$ is a minimizer of the functional $A(y,x,t),$ we find
\begin{equation}\label{functional inequality}
    A(y^*(0,t), 0, t)\leq B(0, y^*(0,t), t, t-\varepsilon_n).
\end{equation} 
Since
\[\frac{(f^{\prime})^{-1}(0)}{e^{\displaystyle{\max_{\theta\in [t_2,t]}\beta(\theta)}}}\leq h(0,t,t_2)\leq \frac{(f^{\prime})^{-1}(0)}{e^{\displaystyle{\min_{\theta\in [t_2,t]}\beta(\theta)}}},\]
we have
\begin{equation*}
\lim_{t_2\to t}\int_{t_2}^t e^{-\beta(\theta)}f\left(h(0,t,t_2)e^{\beta(\theta)}\right)d\theta=0.
\end{equation*}
Hence from the inequality \eqref{functional inequality}, we have
\begin{align*}
&\int_0^{y^*(0,t)}u_0(\theta)d\theta-\int_{\tau}^t e^{-\beta(\theta)}f\left(h(-y^*(0,t),\theta)e^{\beta(\theta)}\right)d\theta-\int_0^{\tau}e^{\beta(\theta)}f\left(h(-y^*(0,t),\tau)e^{\beta(\theta)}\right)d\theta\\
&-y^*(0,t)h(-y^*(0,t),\tau)\\ \leq
&\int_0^{y^*(0,t)}u_0(\theta)d\theta-\int_{t-\varepsilon_n}^t e^{-\beta(\theta)}f(\bar{u_b}(\theta))d\theta-\int_{\tau}^{t-\varepsilon_n}e^{-\beta(\theta)}f\left(h(-y^*(0,t),\theta)e^{\beta(\theta)}\right)d\theta\\
&-\int_0^{\tau}e^{\beta(\theta)}f\left(h(-y^*(0,t),\tau)e^{\beta(\theta)}\right)d\theta-y^*(0,t)h(-y^*(0,t),\tau).
\end{align*}
This implies 
\[\int_{t-\varepsilon_n}^t e^{-\beta(\theta)}f(\bar{u_b}(\theta))d\theta \leq \int_{t-\varepsilon_n}^t e^{-\beta(\theta)}f\left(h(-y^*(0,t),\theta)e^{\beta(\theta)}\right)d\theta .\]
Now dividing by $\varepsilon_n$ on the both sides and passing to the limit as $\varepsilon_n \to 0,$ we get
$$f(\bar{u_b}(t))\leq f(u(0+,t)).$$ 
Note that if a boundary point $(0, t)$ is of \textit{type-I}, then $B(0,t)\leq A(0,t).$ Therefore in this case $(0,t)$ can not be of \textit{Type-I}. If $(0,t)$ is of \textit{type-II}, then we use the dynamic programming principle at a suitable time level to conduct the same proof described above. 
In the remaining case, if it is not of \textit{type-I} or \textit{type-II}, then the $h$-curve joining $(0, t)$ to $(y^*(0,t), 0)$ touches the $t$-axis at the points
$\{t_n\}$  such that $t_n<t_{n+1}$ and $t_n\to t.$ In this case, we can take $\varepsilon_n=t-t_n$ and follow the analysis as above.
This completes the proof of this case.\\
\noindent{\bf Case 2.} When $B(0, t) < A(0, t)$ and $\tau_*(x, t)$ does not converge to $t.$ This case is independent of the type of $(0, t).$ Let $\tau_*(x,t) \to t_1$ for some $t_1<t$. Then
$$W(0, t)=  \int_{t_1}^{t} f^*\left(f^{\prime}\left(h(0,t,t_1)e^{\beta(\theta)}\right)\right)e^{-\beta(\theta)}  d\theta + W(0, t_1).$$
Using dynamics programming principle, one can find a level $t_2$ with $t_1<t_2<t,$ such that
$$W(0, t)= \int_{t_2}^{t} f^*\left(f^{\prime}\left(h(0,t,t_1)e^{\beta(\theta)}\right)\right)e^{-\beta(\theta)}  d\theta + W(X(t_2), t_2),$$
where $X$ is the {\em h-curve} joining $(0, t)$ to $(0, t_1).$ Now following the analysis of {\bf Case 1}, we obtain
\begin{align*}
    f^{\prime}(u(0+, t)) \leq 0\,\,\, \text{and}\,\,\, f(\overline{u}_b(t)) \leq f(u(0+, t)).
\end{align*}
\noindent{\bf Case 3.} When $B(0, t) < A(0, t)$ and $\tau_*(x, t)$ converges to $t.$ If $(0,t)$ is of {\it Type-II,} then $\tau_*(x, t)$ does not converge to $t.$ Therefore in this case the boundary points $(0,t)$ are not of \emph{Type-II}. Therefore, this case has the following two subcases.

\noindent{\bf Subcase 1.} If $(0,t)$ is of {\it Type-I}, then for $(x, t)$ near to $(0, t)$, we have
$$W(x, t)=\int_{\tau_*(x,t)}^t f^*\left(f^{\prime}\left(h(x,t,\tau_*(x,t))e^{\beta(\theta)}\right)\right)e^{-\beta(\theta)}  d\theta -\int^{\tau_* (x, t)}_{\tau} f(\overline{u}_b(\theta))e^{-\beta(\theta)} d\theta + W(0, \tau).$$
Given any sequence $x_k \to 0,$  
$\dot{X}(\tau_* (x_k, t))=f^{\prime}\left(h(x_k, t, \tau_*(x_k, t)\right)e^{\beta (\tau_*(x_k, t))})\geq 0,$
where $X(\theta)$ is the {\em h-curve} joining $(x_k, t)$ to $(0,\tau_*(x_k, t)).$
If $(x_k, t)$ satisfies $f^{\prime}\left(h(x_k, t, \tau(x_k, t))e^{\beta (\tau(x_k, t))}\right)>0,$
then there exists a $\delta>0$  such that for $\tau_1 \in (\tau_*(x_k, t)-\delta, \tau_*(x_k, t)+\delta),$ the {\em h-curves} joining $(x_k, t)$ to $(0, \tau_1)$ lies completely in $Q.$ Then
$$B(\tau_1, x, t)=\int_{\tau_1}^t f^*\left(f^{\prime}\left(h(x,t,\tau_1)e^{\beta(\theta)}\right)\right)e^{-\beta(\theta)}  d\theta -\int^{\tau_1}_{\tau} f(\overline{u}_b(\theta))e^{-\beta(\theta)} d\theta + W(0, \tau).$$
attains its minimum at $\tau_1=\tau_* (x_k, t).$ 
Therefore $\frac{\partial}{\partial \tau_1} W(x, t, \tau_1)=0$ at $\tau_1=\tau_* (x_k, t).$ 
This implies $f(\bar{u}_b (\tau_* (x_k, t)))=f\left(h(x_k, t, \tau_* (x_k, t))e^{\beta(\tau_* (x_k, t))}\right).$ In any case, for a sequence $x_k\to 0,$ we have $f^{\prime}(u(0+, t))\geq 0$ and this implies $u(0+, t)= \bar{u}_b(t).$

%\noindent{\bf Subcase 2.} Suppose $(0,t)$ is of {\it Type-II} and there are infinitely many $n$ such that $x_n \to 0$ for which $\tau_*(x_n, t)\in (t_{2n_k}, t_{2n_k +1}],$ then the same analysis as in {\bf Subcase 1} of {\bf Case 1} can be followed. \\
 %This completes the proof of all cases.

\noindent{\bf Subcase 2.} If $(0,t)$ is of {\it Type-III}, then we have two possibilities. If there are infinitely many $k$ such that $x_k \to 0$ for which $\tau_*(x_k, t)\in [t_{2n_k}, t_{2n_k +1}),$ then the same analysis as in {\bf Subcase 1} can be followed to obtain $u(0+, t)= \bar{u}_b(t).$
 
If there are infinitely many $k$ such that $x_k \to 0$ for which $\tau_*(x_k, t)\in (t_{2n_k -1}, t_{2n_k}],$ then $\tau_*(x_k, t)=t_{2n_k}$ because two minimizers can not intersect by the non-intersecting property, \cref{NIP}. Further $f^{\prime}\left(h(x, t, t_{2n_k } )e^{\beta(t_{2n_k})}\right)=0$ as the {\em h-curve} joining $(x_k, t)$ to $(0,t_{2 n_k})$ and  the {\em h-curve} joining $(0,t_{2 n_k})$ to $(0,t_{2 n_{k -1}})$ is the single {\em h-curve} (since minimizer attend by {\em h-curves}) and  this concludes $h(x, t, t_{2n_k} )e^{\beta(t_{2n_k})}=\lambda_f.$ In this case $f(u(0+, t))=f(\lambda_f).$ Then $u(0+, t)=\lambda_f.$ In this case,
$f^{\prime}(u(0+, t))=0$ holds and by the definition of $\bar{u}_b, $ we obtain that $f(\bar{u}_b(t)) \geq f(u(0+, t)).$ Moreover, we have
\begin{align*}
  W(x_k, t)=\int_{t_{2n_k}}^t f^*\left(f^{\prime}\left(h(x_k, t, t_{2n_k})e^{\beta(\theta)}\right)\right)e^{-\beta(\theta)}d\theta+W(0, t_{2n_k})
\end{align*}
and for $x_k\to 0,$ there exists $0< h_{x_k} < x_k$ such that
\begin{align*}
   B(t_{2n_k}+h_{x_k}, x_k, t) \geq B(t_{2n_k}, x_k, t).
\end{align*}
This implies
\begin{align}\label{EQQ2.29}
   &\int_{t_{2n_k}+h_{x_k}}^t f^*\left(f^{\prime}\left(h(x_k, t, t_{2n_k}+h_{x_k})e^{\beta(\theta)}\right)\right)e^{-\beta(\theta)}d\theta\nonumber\\
    &-\int_{t_{2n_k}}^t f^*\left(f^{\prime}\left(h(x_k, t, t_{2n_k})e^{\beta(\theta)}\right)\right)e^{-\beta(\theta)}d\theta \geq W(0, t_{2n_k}+h_{x_k})-W(0, t_{2n_k})
\end{align}
From the definition of $W(0, t_{2n_k}),$ we have $W(0, t_{2n_k}+h_{x_k})-W(0, t_{2n_k})\geq \int_{t_{2n_k}}^{t_{2n_k}+h_{x_k}} f(\bar{u}_b(\theta))e^{-\beta(\theta)}d\theta.$
Hence, dividing by $h_{x_k}$ and passing to the limit as $x_k \to 0$ in the above inequality \eqref{EQQ2.29}, we conclude
\begin{align*}
   f(u(0+, t))\geq f(\bar{u}_b(t)).
\end{align*}
Note that the passage of limit $x_k\to 0$ in \eqref{EQQ2.29} can be executed similarly as in Step 2 of \cref{subsection 1}. Therefore we obtain $f(u(0+, t))= f(\bar{u}_b (t)).$
 \end{proof}
 \begin{remark}
 In the proof of Bardos- le Roux- N\'{e}d\'{e}lec condition, we used the continuity of the boundary data $u_b$ in addition to the BV regularity on $u_b.$
 %Note that we do not need the local Lipschitz regularity to show the boundary condition. 
 For any BV data $u_b$, a variant of the proof used in the above lemma can be carried out to show the Bardos- le Roux- N\'{e}d\'{e}lec condition at the Lebesgue points of $u_b.$
 \end{remark}
 \begin{remark}
Let us comment on regularity assumptions on the boundary data $u_b.$ The BV regularity of $u_b$ is used to estimate the integral $I_2$ in \eqref{lipw2} of \cref{LIP Value}. We observe the following:\\
$\bullet$ If $u_b$ is only bounded measurable and {\em h-curves} meets finitely many times at the boundary $x=0,$ then we can still estimate $I_2$ of \cref{LIP Value}.

\noindent$\bullet$ If $u_b$ is only bounded measurable and {\em h-curve} meets the boundary infinitely many times, then as of now, it is not clear how to treat this case as we have difficulties in estimating $I_2$ of \cref{LIP Value}. However, \cref{example2} shows that in certain cases such a situation can occur where we consider $u_b$ as a step function.
 \end{remark}
 Now we can complete the proof of \cref{thm explicit formula}.
\begin{proof}[Proof of \cref{thm explicit formula}]  To complete the proof, it only remained to show the weak formulation.
In \cref{subsection 1}, we obtained
\begin{align*}
\begin{cases}
        &W_x=h\left(x-z(x,t, \tau), t, \tau\right):=w,\\
        &W_t=-f(e^{\beta(t)}w)e^{-\beta(t)}.
\end{cases}
\end{align*}
Therefore, we have
\begin{align*}
    0&=\int_0^{\infty}\int_0^{\infty}\left[W_t+e^{-\beta(t)}f(e^{\beta(t)}W_x)\right]\varphi_x (x,t) dx dt\\
    &=-\int_0^{\infty}\int_0^{\infty} W \varphi_{t x} + e^{-\beta(t)}f(e^{\beta(t)}W_x) \varphi_x dx dt-\int_0^{\infty}W(x,0)\varphi_x(x,0)dx\\
    &=\int_0^{\infty}\int_0^{\infty}W_x \varphi_t + e^{-\beta(t)}f(e^{\beta(t)}W_x) \varphi_x dx dt+ \int_0^{\infty} W_x(x,0)\varphi(x,0)dx\\
    &=\int_0^{\infty}\int_0^{\infty}w \varphi_t + e^{-\beta(t)}f(e^{\beta(t)}w) \varphi_x dx dt+ \int_0^{\infty} w(x,0)\varphi(x,0)dx.
\end{align*}
Thus $w$ satisfies $w_t+e^{-\beta(t)}f(e^{\beta(t)}w)_x=0$ in the sense of \cref{weak formulation} and since $w(x,0)=u_0(x),$ we conclude that $u$ satisfied the integral identity in \cref{weak formulation}. Moreover, in \cref{BDC lemma}, we showed that $u(\cdot, t)$ satisfies the boundary condition \eqref{boundary condition} for any fixed $t>0.$ This completes the proof.
\end{proof} 
\section{Construction of generalized characteristics}\label{section 4}
In this section, we will interpret the explicit formulae obtained in \cref{thm explicit formula} for the balance laws in terms of the generalized characteristics. 
\subsection{Proof of \cref{GC thm}} Inspired by \cite{neumann2021initial},  we begin with the definition of characteristic triangles.
\begin{definition}[\textbf{Characteristic triangles}]\label{Char triangles}
Let $x\geq 0$, $t>0$ and $A(x,t)$, $B(x,t)$ are given as in \eqref{definition A} and \eqref{definition B}.
\begin{enumerate}
\item[1.] For $A(x,t)< B(x,t),$ we define the characteristics triangle at the point $(x,t)$ as the region bounded by the {\em h-curves} joining the points $(x,t), (y_{*} (x,t), 0)$ and $(y^{*} (x,t), 0).$\\
\item[2.] For $A(x,t)>B(x,t),$ we define the characteristics triangle at the point $(x,t)$ as the region bounded by the {\em h-curves} joining the points $(x,t), (0, \tau_{*} (x,t),$ and $(0, \tau^{*} (x,t)).$\\
\item[3.] For $A(x,t)=B(x,t),$ we define the characteristics triangle at the point $(x,t)$ as the region bounded by the {\em h-curves} joining the points
 $(x,t)$, $(y^{*} (x,t),0)$, $(0, \tau^{*} (x,t))$ which includes the point $(0,0)$, where $(x,t)$ and $(y_{*} (x,t), 0)$ (resp. $(\tau^{*} (x,t), 0)$) are joined by a $h$-curves for initial (resp. boundary) functional.\\
 \item [4.] For $A(0, t) < B(0, t),$ we define the characteristic triangle as the region bounded by the {\em h-curve} $\gamma:[0, t] \to [0, \infty)$ joining $(0,t)$ to $(y^*(0,t), 0)$ and $\gamma(\theta)>0$ for all $\theta\in [0,t).$
\end{enumerate}
We denote the characteristic triangle associated with the point $(x,t)$ by $\Delta(x,t)$.
\end{definition}
%\begin{remark}
%We observe that the last definition of a characteristic triangle, starting from the boundary is in the  almost everywhere sense. We may have points on the boundary from that there exists no $h$-curve joining the points on the $x$-axis which completely lies in the quarter plane. But one can see that such a set is countable. This can be shown by using the non-intersecting property of $h$-curves. The procedure is the following. On the contrary, suppose for all points in the interval  $[t_1, t_2]$ on the boundary, there exists no $h$- curve which joins to the $x$- axis completely lying on the quarter plane. Select a point $b < t_1,$ such that there exists a $h-$curve, say $\gamma_{h_b}$ starting from $b$ to the $x$-axis. Now select two points $b_1, b_2 \in [t_1, t_2]$ with $b_1<b_2$ such that two $h$-curves starting from $b_1, b_2,$ say $\gamma_{h_{b_1}}$ and $\gamma_{h_{b_2}}$ coming outside of the quarter plane, intersecting the boundary at the points $\tilde{b_1}$ and $\tilde{b_2}$ respectively, such that $\tilde{b_2}<b<\tilde{b_1}$  and join the $x-$axis. Then there exists a point $c\in  [b_1, b_2]$ such that the $h$-curve $\gamma_{h_c}$ intersects $\gamma_{h_b}$ inside the quarter plane which is a contradiction.
%\end{remark} 
We list some properties of characteristic triangles in the following lemma.
\begin{lemma}[\textbf{Properties of triangles}]\label{newlemma5.2}
Let $x\geq  0, t>0.$ Then we have the following properties of characteristic triangles.
\begin{enumerate}
        \item Let $t>0$ be fixed, and $x_1$, $x_2 > 0$, $x_1\neq x_2$ but arbitrary. Then the characteristic triangles associated with $(x_1,t)$ and $(x_2,t)$ do not intersect in the interior of $\mathbb{R}_+^2$.
        \item If two characteristic triangles intersect in $\mathbb{R}_+^2$, then one is contained in the other. 
        \item For any time $t_0 >0$ we have
\[
\bigcup_{x \in [0, \infty)}\Delta(x, t_0)=\{(x,t)| x \in [0, \infty), 0 \leq t \leq t_0\}\,.
\]
    \end{enumerate}
\end{lemma}
\begin{proof}
The proof uses \cref{NIP} and \cref{lnew} and follows the similar lines given in \cite{neumann2021initial}.
\end{proof}
Our next goal is to prove \cref{GC thm}. We start by stating a more rigorous version of it.
\begin{theorem}\label{thm4.1}
Let $t_1>0$ be fixed. For each $(x_1,t_1)$ there exists a unique Lipschitz continuous curve $x=\mathtt{X}(t)$ with $\mathtt{X}(t_1)=x_1$ such that the characteristic triangles associated with each point on the curve form an increasing family of triangles and for any point $(x,t), t\geq t_1$ on the curve, we have the following:\\\\
(i) When $A(x,t)<B(x,t)$ we have
\begin{equation*}
\begin{aligned}
\lim_{t^{\prime\prime},t^{\prime}\to t}\frac{\mathtt{X}(t^{\prime\prime})-\mathtt{X}(t^{\prime})}{t^{\prime\prime}-t^{\prime}}=
\begin{cases}
f^{\prime}\left(h(x-y(x,t),t)e^{\beta(t)}\right)\,\,\,\,\, &\textnormal{if}\,\,\, y_*(x,t)=y^*(x,t),\\
\frac{f\left(h(x-y_*(x,t),t)e^{\beta(t)}\right)-f\left(h(x-y^*(x,t),t)e^{\beta(t)}\right)}{h(x-y_*(x,t),t)e^{\beta(t)}-h(x-y^*(x,t),t)e^{\beta(t)}}\,\,\,\,\, &\textnormal{if}\,\,\, y_*(x,t)<y^*(x,t).
\end{cases}
\end{aligned}
\end{equation*}
(ii) When $B(x,t)<A(x,t)$ we have
\begin{align*}
\lim_{t^{\prime\prime},t^{\prime}\to t}\frac{\mathtt{X}(t^{\prime\prime})-\mathtt{X}(t^{\prime})}{t^{\prime\prime}-t^{\prime}}=
\begin{cases}
f^{\prime}\left(h(x,t, \tau(x,t))e^{\beta(t)}\right)\,\,\,\,\, &\textnormal{if}\,\,\, \tau_{*}(x,t)=\tau^*(x,t),\\
\frac{f\left(h(x,t, \tau_{*}(x,t))e^{\beta(t)}\right)-f\big(h(x,t, \tau^*(x,t))e^{\beta(t)}\big)}{h(x,t, \tau_{*}(x,t))e^{\beta(t)}-h(x,t, \tau^*(x,t))e^{\beta(t)}}\,\,\,\,\, &\textnormal{if}\,\,\, \tau_{*}(x,t)<\tau^*(x,t).
\end{cases}
\end{align*}
(iii) When $A(x,t)=B(x,t)$ and $y^*(x,t)\neq0$ or $\tau^*(x,t)\neq 0,$ we have
\begin{align*}
    \lim_{t^{\prime\prime},t^{\prime}\to t}\frac{\mathtt{X}(t^{\prime\prime})-\mathtt{X}(t^{\prime})}{t^{\prime\prime}-t^{\prime}}=\frac{f\left(h(x-y^*(x,t),t)e^{\beta(t)}\right)-f\left(h(x,t, \tau^*(x,t))e^{\beta(t)}\right)}{h(x-y^*(x,t),t)e^{\beta(t)}-h(x,t, \tau^*(x,t))e^{\beta(t)}}.
\end{align*}
(iv) When $A(x,t)=B(x,t)$ and both $y^*(x,t)=0, \tau^*(x,t)=0$ we have
\begin{align*}
    \lim_{t^{\prime\prime},t^{\prime}\to t}\frac{\mathtt{X}(t^{\prime\prime})-\mathtt{X}(t^{\prime})}{t^{\prime\prime}-t^{\prime}}=f^{\prime}\left(h(x,t)e^{\beta(t)}\right).
\end{align*}
\end{theorem}
\begin{proof}
 At any level $t>t_1,$ by \cref{newlemma5.2} there exists a unique characteristics triangle containing the point $(x_1, t_1)$ with apex at some point denoted as $(\mathtt{X}(t), t).$  This procedure gives the existence and uniqueness of the curve $\mathtt{X}(t).$

Let $t<t^{\prime}<t^{\prime\prime}$ and $\mathtt{X}(t)=x,\mathtt{X}(t^{\prime})=x^{\prime},\mathtt{X}(t^{\prime\prime})=x^{\prime\prime}.$ For the sake of concreteness, henceforth, without loss of any generality, we assume that the {\em h-curve} joining any two points in $Q$ lies completely in $Q.$ In the case of the initial value problem, note that the functional $A(y_*(x^{\prime}, t^{\prime}), x^{\prime\prime}, t^{\prime \prime})$ in the inequality \eqref{IF inequality} is well defined as $y_*(x^{\prime}, t^{\prime}) \in H_A(x^{\prime\prime}, t^{\prime\prime}).$ Indeed, if the {\em h-curve} joining $(x^{\prime\prime}, t^{\prime\prime})$ to $(y_*(x^{\prime}, t^{\prime}), 0)$ leaves the domain $Q$ and comes back, then it has to meet the {\em h-curve} joining $(x^{\prime\prime}, t^{\prime\prime})$ to $(y_*(x^{\prime\prime},t^{\prime\prime}), 0)$ at say $(x_1, t_1)$ point. Hence we get a contradiction to the uniqueness of {\em h-curve} joining any two points in $Q,$ as $(x^{\prime\prime}, t^{\prime \prime})$ and  $(x_1, t_1)$ is joined by two different {\em h-curves}. The case of the boundary functional in \eqref{BF inequality} can also be treated in the same way using the dynamic programming principle, \cref{DPP}.

To prove $(i),$ first we consider the case $y_*(x,t)=y^*(x,t).$ In this case the slope $\frac{x^{\prime\prime}-x^{\prime}}{t^{\prime\prime}-t^{\prime}}$ lies between the slope of the {\em h-curve} joining the points $\{(x^{\prime\prime},t^{\prime\prime}), (y^*(x^{\prime\prime},t^{\prime\prime}), 0)\}$ and $\{(x^{\prime\prime},t^{\prime\prime}), (y_*(x^{\prime\prime},t^{\prime\prime}), 0)\}.$ The {\em h-curve} through the points $\{(x^{\prime\prime},t^{\prime\prime}), (y_*(x^{\prime\prime},t^{\prime\prime}), 0)\}$ is
\[X_1(\xi)=y_*(x^{\prime\prime},t^{\prime\prime})+\int_0^{\xi}f^{\prime}\left(h(x^{\prime\prime}-y_*(x^{\prime\prime},t^{\prime\prime}),t^{\prime\prime})e^{\beta(\theta)}\right)d\theta,\] and {\em h-curve} through the points $\{(x^{\prime\prime},t^{\prime\prime}), (y^*(x^{\prime\prime},t^{\prime\prime}), 0)\}$ is 
\[X_2(\xi)=y^*(x^{\prime\prime},t^{\prime\prime})+\int_0^{\xi}f^{\prime}\left(h(x^{\prime\prime}-y^*\left(x^{\prime\prime},t^{\prime\prime}),t^{\prime\prime}\right)e^{\beta(\theta)}\right)d\theta.\]
Hence 
\begin{equation}\label{equ58}
\frac{X_1(t^{\prime\prime})-X_1(t^{\prime})}{t^{\prime\prime}-t^{\prime}}\geq \frac{x^{\prime\prime}-x^{\prime}}{t^{\prime\prime}-t^{\prime}} \geq \frac{X_2(t^{\prime\prime})-X_2(t^{\prime})}{t^{\prime\prime}-t^{\prime}}, 
\end{equation}
i.e.,
\[\frac{\int_{t^{\prime}}^{t^{\prime\prime}}f^{\prime}\left(h(x-y_*(x^{\prime\prime},t^{\prime\prime}),t^{\prime\prime})e^{\beta(\theta)}\right)d\theta}{t^{\prime\prime}-t^{\prime}}\geq \frac{x^{\prime\prime}-x^{\prime}}{t^{\prime\prime}-t^{\prime}}\geq \frac{\int_{t^{\prime}}^{t^{\prime\prime}}f^{\prime}\left(h(x-y^*(x^{\prime\prime},t^{\prime\prime}),t^{\prime\prime})e^{\beta(\theta)}\right)d\theta}{t^{\prime\prime}-t^{\prime}}.\]
Now passing to the limit as $t^{\prime \prime}, t^{\prime} \to t$ and  using the monotonicity and semicontinuity of $y_*, y^*$ we obtain the first identity of $(i).$ Now consider the case $y_*(x,t)<y^*(x,t).$ From the definition of $A(y,x,t)$ we have
\begin{align}\label{IF inequality}
A(y^*(x^{\prime\prime},t^{\prime\prime}),x^{\prime\prime},t^{\prime\prime})-A(y_*(x^{\prime},t^{\prime}),x^{\prime\prime},t^{\prime\prime})\leq A(y^*(x^{\prime\prime},t^{\prime\prime}),x^{\prime},t^{\prime})-A(y_*(x^{\prime},t^{\prime}),x^{\prime},t^{\prime})
\end{align}
The above inequality can be rewritten as
\begin{equation}\label{equation4.3}
    \begin{aligned}
   &\big[A(y^*(x^{\prime\prime},t^{\prime\prime}),x^{\prime\prime},t^{\prime\prime})-A(y^*(x^{\prime\prime},t^{\prime\prime}),x^{\prime\prime},t^{\prime})\big]-\big[A(y_*(x^{\prime},t^{\prime}),x^{\prime\prime},t^{\prime\prime})-A(y_*(x^{\prime},t^{\prime}),x^{\prime\prime},t^{\prime})\big]\\
   &\leq \big[A(y^*(x^{\prime\prime},t^{\prime\prime}),x^{\prime},t^{\prime})-A(y^*(x^{\prime\prime},t^{\prime\prime}),x^{\prime\prime},t^{\prime})\big]-\big[A(y_*(x^{\prime},t^{\prime}),x^{\prime},t^{\prime})-A(y_*(x^{\prime},t^{\prime}),x^{\prime\prime},t^{\prime})\big].
    \end{aligned}
\end{equation}
We expand the above inequality step by step. Firstly, using \eqref{equation2.4} we have
\begin{equation}\label{equation4.4}
\begin{aligned}
   &A(y^*(x^{\prime\prime},t^{\prime\prime}),x^{\prime\prime},t^{\prime\prime})-A(y^*(x^{\prime\prime},t^{\prime\prime}),x^{\prime\prime},t^{\prime})\\
   &=-\int_{x^{\prime\prime}}^{y^*(x^{\prime\prime},t^{\prime\prime})}h(x^{\prime\prime}-\eta, t^{\prime\prime})d\eta-\int_0^{t^{\prime\prime}}e^{-\beta(\eta)}f\left(h(0,t^{\prime\prime})e^{\beta(\eta)}\right)d\eta\\
   &+\int_{x^{\prime\prime}}^{y^*(x^{\prime\prime},t^{\prime\prime})}h(x^{\prime\prime}-\eta, t^{\prime})d\eta+\int_0^{t^{\prime}}e^{-\beta(\eta)}f\left(h(0,t^{\prime})e^{\beta(\eta)}\right)d\eta.
\end{aligned}
\end{equation}
Using a similar calculation of \eqref{equation2.4}, one can obtain, for $0<\tau<t,$ 
\begin{align*}
    &\int_x^y h(x-\eta,t)dz +\int_0 ^t e^{-\beta(\eta)}f\left(h(0, t)e^{\beta(\eta)}\right)d\eta\\
&= \int_{\tau}^t e^{-\beta(\eta)}f\left(h(x-y,\eta)e^{\beta(\eta)}\right)d\eta+\int_0^{\tau}
e^{-\beta(\eta)}f\left(h(x-y,\tau)e^{\beta(\eta)}\right)d\eta-(x-y)h(x-y,\tau).
\end{align*}

Using the above equation in \eqref{equation4.4}, we get
\[A(y^*(x^{\prime\prime},t^{\prime\prime}),x^{\prime\prime},t^{\prime\prime})-A(y^*(x^{\prime\prime},t^{\prime\prime}),x^{\prime\prime},t^{\prime})=-\int_{t^{\prime}}^{t^{\prime\prime}}e^{-\beta(\eta)}f\left(h(x^{\prime\prime}-y^*(x^{\prime\prime},t^{\prime\prime}),\eta)e^{\beta(\eta)}\right)d\eta.\]
Similarly,
\[A(y_*(x^{\prime},t^{\prime}),x^{\prime\prime},t^{\prime\prime})-A(y_*(x^{\prime},t^{\prime}),x^{\prime\prime},t^{\prime})=-\int_{t^{\prime}}^{t^{\prime\prime}}e^{-\beta(\eta)}f\left(h(x^{\prime\prime}-y_*(x^{\prime},t^{\prime}),\eta)e^{\beta(\eta)}\right)d\eta.\]
Thus we find
\begin{align*}
       &\left[A(y^*(x^{\prime\prime},t^{\prime\prime}),x^{\prime\prime},t^{\prime\prime})-A(y^*(x^{\prime\prime},t^{\prime\prime}),x^{\prime\prime},t^{\prime})\right]-\left[A(y_*(x^{\prime},t^{\prime}),x^{\prime\prime},t^{\prime\prime})-A(y_*(x^{\prime},t^{\prime}),x^{\prime\prime},t^{\prime})\right]\\
       &=\int_{t^{\prime}}^{t^{\prime\prime}}e^{-\beta(\eta)}\left[f\left(h(x^{\prime\prime}-y_*(x^{\prime},t^{\prime}),\eta)e^{\beta(\eta)}\right)-f\left(h(x^{\prime\prime}-y^*(x^{\prime\prime},t^{\prime\prime}),\eta)e^{\beta(\eta)}\right)\right]d\eta.
    \end{align*}
and after a change of variable
\begin{equation*}
    \begin{aligned}
    &\left[A(y^*(x^{\prime\prime},t^{\prime\prime}),x^{\prime},t^{\prime})-A(y^*(x^{\prime\prime},t^{\prime\prime}),x^{\prime\prime},t^{\prime})\right]-\left[A(y_*(x^{\prime},t^{\prime}),x^{\prime},t^{\prime})-A(y_*(x^{\prime},t^{\prime}),x^{\prime\prime},t^{\prime})\right]\\
    &=\int_{x^{\prime}-y^*(x^{\prime\prime},t^{\prime\prime})}^{x^{\prime\prime}-y^*(x^{\prime\prime},t^{\prime\prime})}h(z,t^{\prime})dz-\int_{x^{\prime}-y_*(x^{\prime},t^{\prime})}^{x^{\prime\prime}-y_*(x^{\prime},t^{\prime})}h(z,t^{\prime})dz.
     \end{aligned}
\end{equation*}
Therefore from the inequality \eqref{equation4.3} we get
\begin{equation}\label{equation4.7}
   \frac{\frac{1}{t^{\prime\prime}-t^{\prime}}\int_{t^{\prime}}^{t^{\prime\prime}}e^{-\beta(\eta)}\left[f\left(h(x^{\prime\prime}-y_*(x^{\prime},t^{\prime}),\eta)e^{\beta(\eta)}\right)-f\left(h(x^{\prime\prime}-y^*(x^{\prime\prime},t^{\prime\prime}),\eta)e^{\beta(\eta)}\right)\right]d\eta}{\frac{1}{x^{\prime\prime}-x^{\prime}}\Big[\int_{x^{\prime}-y^*(x^{\prime\prime},t^{\prime\prime})}^{x^{\prime\prime}-y^*(x^{\prime\prime},t^{\prime\prime})}h(z,t^{\prime})dz-\int_{x^{\prime}-y_*(x^{\prime},t^{\prime})}^{x^{\prime\prime}-y_*(x^{\prime},t^{\prime})}h(z,t^{\prime})dz\Big]} \leq \frac{x^{\prime\prime}-x^{\prime}}{t^{\prime\prime}-t^{\prime}}.
\end{equation}
On the other hand, considering the inequality
\begin{equation*}
    \begin{aligned}
   &\big[A(y_*(x^{\prime\prime},t^{\prime\prime}),x^{\prime\prime},t^{\prime\prime})-A(y_*(x^{\prime\prime},t^{\prime\prime}),x^{\prime\prime},t^{\prime})\big]-\big[A(y^*(x^{\prime},t^{\prime}),x^{\prime\prime},t^{\prime\prime})-A(y^*(x^{\prime},t^{\prime}),x^{\prime\prime},t^{\prime})\big]\\
   &\leq \big[A(y_*(x^{\prime\prime},t^{\prime\prime}),x^{\prime},t^{\prime})-A(y_*(x^{\prime\prime},t^{\prime\prime}),x^{\prime\prime},t^{\prime})\big]-\big[A(y^*(x^{\prime},t^{\prime}),x^{\prime},t^{\prime})-A(y^*(x^{\prime},t^{\prime}),x^{\prime\prime},t^{\prime})\big]
    \end{aligned}
\end{equation*}
and simplifying it as before, we get
\begin{equation}\label{equation4.9}
   \frac{\frac{1}{t^{\prime\prime}-t^{\prime}}\int_{t^{\prime}}^{t^{\prime\prime}}e^{-\beta(\eta)}\Big[f(h(x^{\prime\prime}-y_*(x^{\prime},t^{\prime}),\eta)e^{\beta(\eta)})-f(h(x^{\prime\prime}-y^*(x^{\prime\prime},t^{\prime\prime}),\eta)e^{\beta(\eta)})\Big]d\eta}{\frac{1}{x^{\prime\prime}-x^{\prime}}\Big[\int_{x^{\prime}-y^*(x^{\prime\prime},t^{\prime\prime})}^{x^{\prime\prime}-y^*(x^{\prime\prime},t^{\prime\prime})}h(z,t^{\prime})dz-\int_{x^{\prime}-y_*(x^{\prime},t^{\prime})}^{x^{\prime\prime}-y_*(x^{\prime},t^{\prime})}h(z,t^{\prime})dz\Big]} \geq \frac{x^{\prime\prime}-x^{\prime}}{t^{\prime\prime}-t^{\prime}}.
\end{equation}
Passing to the limit as $t^{\prime\prime}, t^{\prime} \to t$ in \eqref{equation4.7}-\eqref{equation4.9} we obtain the second identity of $(i).$\\
For $(ii)$ we observe that the slope $\frac{x^{\prime\prime}-x^{\prime}}{t^{\prime\prime}-t^{\prime}}$ lies between the slope of the  {\em h-curve} joining the points $\{(x^{\prime\prime}, t^{\prime\prime}), (0, \tau_{*}(x^{\prime\prime}, t^{\prime\prime}))\}$ and the points $\{(x^{\prime\prime}, t^{\prime\prime}), (0, \tau^*(x^{\prime\prime}, t^{\prime\prime}))\}.$ The {\em h-curve} through $\{(x^{\prime\prime}, t^{\prime\prime}), (0, \tau_{*}(x^{\prime\prime}, t^{\prime\prime}))\}$ is
\[Y_1(\xi)=\int_{\tau_{*}(x^{\prime\prime}, t^{\prime\prime})}^{\xi} f^{\prime}\left(h(x^{\prime\prime},t^{\prime\prime},\tau_{*}(x^{\prime\prime}, t^{\prime\prime}) )e^{\beta(\theta)}\right)d\theta,\,\,\, Y_1(t^{\prime\prime})=x^{\prime\prime}\]
and  {\em h-curve} through $\{(x^{\prime\prime}, t^{\prime\prime}), (0, \tau^*(x^{\prime\prime}, t^{\prime\prime}))\}$ is 
\[Y_2(\xi)=\int_{\tau^*(x^{\prime\prime}, t^{\prime\prime})}^{\xi} f^{\prime}\left(h(x^{\prime\prime},t^{\prime\prime},\tau^*(x^{\prime\prime}, t^{\prime\prime}) )e^{\beta(\theta)}\right)d\theta, \,\,\, Y_2(t^{\prime\prime})=x^{\prime\prime}.\]
Hence 
\[\frac{Y_2(t^{\prime\prime})-Y_2(t^{\prime})}{t^{\prime\prime}-t^{\prime}}\geq \frac{x^{\prime\prime}-x^{\prime}}{t^{\prime\prime}-t^{\prime}} \geq \frac{Y_1(t^{\prime\prime})-Y_1(t^{\prime})}{t^{\prime\prime}-t^{\prime}}.\]
This implies
\[\frac{\int_{t^{\prime}}^{t^{\prime\prime}}f^{\prime}\left(h(x^{\prime\prime},t^{\prime\prime},\tau^*(x^{\prime\prime}, t^{\prime\prime}) )e^{\beta(\theta)}\right)d\theta}{t^{\prime\prime}-t^{\prime}}\geq\frac{x^{\prime\prime}-x^{\prime}}{t^{\prime\prime}-t^{\prime}} \geq \frac{\int_{t^{\prime}}^{t^{\prime\prime}}f^{\prime}\left(h(x^{\prime\prime},t^{\prime\prime},\tau_{*}(x^{\prime\prime}, t^{\prime\prime}) )e^{\beta(\theta)}\right)d\theta}{t^{\prime\prime}-t^{\prime}}. \]
Passing to the limit as $t^{\prime\prime},t^{\prime} \to t,$ we obtain the first identity of $(ii).$ To prove the second identity for the case $\tau_{*}(x,t)<\tau^*(x,t),$ observe that
\begin{align}\label{BF inequality}
B(\tau^*(x^{\prime\prime},t^{\prime\prime}), x^{\prime\prime},t^{\prime\prime})-B(\tau_{*}(x^{\prime},t^{\prime}), x^{\prime\prime},t^{\prime\prime})\leq B(\tau^*(x^{\prime\prime},t^{\prime\prime}), x^{\prime},t^{\prime})-B(\tau_{*}(x^{\prime},t^{\prime}), x^{\prime},t^{\prime}).
\end{align}
The above inequality can be rewritten as 
\begin{equation}\label{equ4.9}
    \begin{aligned}
    &\left[B( \tau^*(x^{\prime\prime},t^{\prime\prime}), x^{\prime\prime},t^{\prime\prime})-B(\tau^*(x^{\prime\prime},t^{\prime\prime}), x^{\prime\prime},t^{\prime})\right]
    -\left[B(\tau_*(x^{\prime},t^{\prime}), x^{\prime\prime},t^{\prime\prime})-B(\tau_*(x^{\prime},t^{\prime}),x^{\prime\prime},t^{\prime})\right]\\
    &\leq \left[B(\tau^*(x^{\prime\prime},t^{\prime\prime}), x^{\prime},t^{\prime})-B(\tau^*(x^{\prime\prime},t^{\prime\prime}), x^{\prime\prime},t^{\prime})\right]
    -\left[B(\tau_{*}(x^{\prime},t^{\prime}), x^{\prime}, t^{\prime})-B(\tau_{*}(x^{\prime},t^{\prime}), x^{\prime\prime},
    t^{\prime})\right].
    \end{aligned}
\end{equation}
Firstly, using \eqref{derivative in t} we have
\begin{align*}
    &B(\tau^*(x^{\prime\prime},t^{\prime\prime}), x^{\prime\prime},t^{\prime\prime})-B(\tau^*(x^{\prime\prime},t^{\prime\prime}), x^{\prime\prime},t^{\prime})
    =-\int_{t^{\prime}}^{t^{\prime\prime}}e^{-\beta(\eta)}f\left(h(x^{\prime\prime},\eta, \tau^*(x^{\prime\prime},t^{\prime\prime}))e^{\beta(\eta)}\right)d\eta,
\end{align*}
and similarly
\begin{equation*}
    \begin{aligned}
    &B(\tau_{*}(x^{\prime},t^{\prime}), x^{\prime\prime},t^{\prime\prime})-B(\tau_{*}(x^{\prime},t^{\prime}), x^{\prime\prime},t^{\prime})
    &=-\int_{t^{\prime}}^{t^{\prime\prime}}e^{-\beta(\eta)}f\left(h(x^{\prime\prime},\eta, \tau_{*}(x^{\prime},t^{\prime}))e^{\beta(\eta)}\right)d\eta.
    \end{aligned}
\end{equation*}
Hence L.H.S of the inequality \eqref{equ4.9} imply
\begin{align*}
    &\left[B(\tau^*(x^{\prime\prime},t^{\prime\prime}),x^{\prime\prime},t^{\prime\prime})-B(\tau^*(x^{\prime\prime},t^{\prime\prime}), x^{\prime\prime},t^{\prime})\right]
    -\left[B(\tau_{*}(x^{\prime},t^{\prime}), x^{\prime\prime},t^{\prime\prime})-B(\tau_{*}(x^{\prime},t^{\prime}), x^{\prime\prime},t^{\prime})\right]\\
    &=\int_{t^{\prime}}^{t^{\prime\prime}}e^{-\beta(\eta)}\left[f\left(h(x^{\prime\prime},\eta, \tau_{*}(x^{\prime},t^{\prime}))e^{\beta(\eta)}\right)-f\left(h(x^{\prime\prime},\eta, \tau^*(x^{\prime\prime},t^{\prime\prime}))e^{\beta(\eta)}\right)\right]d\eta.
    \end{align*}
Now using a similar calculation as \eqref{derivative in t}, for $t_2<\theta<t$ one obtains
\begin{equation*}
    \begin{aligned}
    &\int_{t_2}^t e^{-\beta(\theta)}f^*\Big(f^{\prime}\left(h(x,t,t_2)e^{\beta(\theta)}\right)\Big)d\theta\\
   & =\int_0^x h(\eta,t, t_2)d\eta-\int_{t_2}^te^{-\beta(\theta)}f\left(h(0,t,t_2)e^{\beta(\theta)}\right)d\theta.
    \end{aligned}
\end{equation*}
Applying the above equation in the R.H.S of \eqref{equ4.9} gives
\begin{equation*}
    \begin{aligned}
    &\big[B(\tau^*(x^{\prime\prime},t^{\prime\prime}), x^{\prime},t^{\prime})-B(\tau^*(x^{\prime\prime},t^{\prime\prime}), x^{\prime\prime},t^{\prime})\big]
    -\big[B(\tau_{*}(x^{\prime},t^{\prime}), x^{\prime}, t^{\prime})-B(\tau_{*}(x^{\prime},t^{\prime}), x^{\prime\prime},
    t^{\prime})\big]\\
    &=\int_{x^{\prime}}^{x^{\prime\prime}}h(\eta,t^{\prime}, \tau_{*}(x^{\prime},t^{\prime}))-h(\eta,t^{\prime}, \tau^*(x^{\prime\prime},t^{\prime\prime}))d\eta.
    \end{aligned}
\end{equation*}
This implies
\begin{equation}\label{equation4.10}
    \begin{aligned}
    \frac{\frac{1}{t^{\prime\prime}-t^{\prime}}\int_{t^{\prime}}^{t^{\prime\prime}}e^{-\beta(\eta)}\Big[f(h(x^{\prime\prime},\eta, \tau_{*}(x^{\prime},t^{\prime}))e^{\beta(\eta)})-f(h(x^{\prime\prime},\eta, \tau^*(x^{\prime\prime},t^{\prime\prime}))e^{\beta(\eta)})\Big]d\eta}{\frac{1}{x^{\prime\prime}-x^{\prime}}\int_{x^{\prime}}^{x^{\prime\prime}}h(\eta,t^{\prime}, \tau_{*}(x^{\prime},t^{\prime}))-h(\eta,t^{\prime}, \tau^*(x^{\prime\prime},t^{\prime\prime}))d\eta}\leq \frac{x^{\prime\prime}-x^{\prime}}{t^{\prime\prime}-t^{\prime}}.
    \end{aligned}
\end{equation}
Similarly from inequality
\begin{align*}
    &\left[B(\tau_{*}(x^{\prime\prime},t^{\prime\prime}), x^{\prime\prime},t^{\prime\prime})-B(\tau_{*}(x^{\prime\prime},t^{\prime\prime}), x^{\prime\prime},t^{\prime})\right]
    -\left[B(\tau^*(x^{\prime},t^{\prime}), x^{\prime\prime},t^{\prime\prime})-B(\tau^*(x^{\prime},t^{\prime}), x^{\prime\prime},t^{\prime})\right]\\
    &\leq \left[B(\tau_*(x^{\prime\prime},t^{\prime\prime}), x^{\prime},t^{\prime})-B(\tau_{*}(x^{\prime\prime},t^{\prime\prime}), x^{\prime\prime},t^{\prime})\right]
    -\left[B(\tau^{*}(x^{\prime},t^{\prime}), x^{\prime},t^{\prime})-B(\tau^*(x^{\prime},t^{\prime}), x^{\prime\prime},t^{\prime})\right]
\end{align*}
we get 
\begin{align}\label{equation4.11}
\frac{\frac{1}{t^{\prime\prime}-t^{\prime}}\int_{t^{\prime}}^{t^{\prime\prime}}e^{-\beta(\eta)}\Big[f(h(x^{\prime\prime},\eta, \tau_{*}(x^{\prime},t^{\prime}))e^{\beta(\eta)})-f(h(x^{\prime\prime},\eta, \tau^*(x^{\prime\prime},t^{\prime\prime}))e^{\beta(\eta)})\Big]d\eta}{\frac{1}{x^{\prime\prime}-x^{\prime}}\int_{x^{\prime}}^{x^{\prime\prime}}h(\eta,t^{\prime}, \tau_{*}(x^{\prime},t^{\prime}))-h(\eta,t^{\prime}, \tau^*(x^{\prime\prime},t^{\prime\prime}))d\eta}\geq \frac{x^{\prime\prime}-x^{\prime}}{t^{\prime\prime}-t^{\prime}}.
\end{align}
Passing to the limit $t^{\prime\prime},t^{\prime} \to t$ in \eqref{equation4.10}-\eqref{equation4.11}, we obtain the second identity of $(ii).$\\
For case $(iii),$ we observe the inequality
\begin{align*}
&B(\tau^*(x^{\prime\prime},t^{\prime\prime}), x^{\prime\prime},t^{\prime\prime})-A(y^*(x^{\prime},t^{\prime}),x^{\prime\prime},t^{\prime\prime})\leq B(\tau^*(x^{\prime\prime},t^{\prime\prime}), x^{\prime},t^{\prime})-A(y^*(x^{\prime},t^{\prime}),x^{\prime},t^{\prime}).
\end{align*}
Following the previous strategy the above inequality can be written as
\begin{align*}
&\big[B(\tau^*(x^{\prime\prime},t^{\prime\prime}), x^{\prime\prime},t^{\prime\prime}) -B(\tau^*(x^{\prime\prime},t^{\prime\prime}), x^{\prime\prime},t^{\prime})\big]-\big[A(y^*(x^{\prime},t^{\prime}),x^{\prime\prime},t^{\prime\prime})-A(y^*(x^{\prime},t^{\prime}),x^{\prime\prime},t^{\prime})\big]\\
&\leq \big[B(\tau^*(x^{\prime\prime},t^{\prime\prime}), x^{\prime},t^{\prime})-B(\tau^*(x^{\prime\prime},t^{\prime\prime}), x^{\prime\prime},t^{\prime})\big]-\big[A(y^*(x^{\prime},t^{\prime}),x^{\prime},t^{\prime})-A(y^*(x^{\prime},t^{\prime}),x^{\prime\prime},t^{\prime})\big].
\end{align*}
From the previous analysis, it follows directly that
\begin{align}\label{equation4.13}
\frac{\frac{1}{t^{\prime\prime}-t^{\prime}}\int_{t^{\prime}}^{t^{\prime\prime}}e^{-\beta(\eta)}\Big[f(h(x^{\prime\prime}-y^*(x^{\prime},t^{\prime}),\eta)e^{\beta(\eta)})-f(h(x^{\prime\prime},\eta, \tau^*(x^{\prime\prime},t^{\prime\prime}))e^{\beta(\eta)})\Big]d\eta}{\frac{1}{x^{\prime\prime}-x^{\prime}}\int_{x^{\prime}-y^*(x^{\prime},t^{\prime})}^{x^{\prime\prime}-y^*(x^{\prime},t^{\prime})}h(z,t^{\prime})d\eta-\frac{1}{x^{\prime\prime}-x^{\prime}}\int_{x^{\prime}}^{x^{\prime\prime}}h(\eta,t^{\prime}, \tau^*(x^{\prime},t^{\prime}))d\eta}\leq \frac{x^{\prime\prime}-x^{\prime}}{t^{\prime\prime}-t^{\prime}}.
\end{align}
 Similarly from the inequality 
\begin{align*}
A(y^*(x^{\prime\prime},t^{\prime\prime}),x^{\prime\prime},t^{\prime\prime})-B(\tau^*(x^{\prime},t^{\prime}), x^{\prime\prime},t^{\prime\prime})\leq A(y^*(x^{\prime\prime},t^{\prime\prime}),x^{\prime},t^{\prime})-B(\tau^*(x^{\prime},t^{\prime}), x^{\prime},t^{\prime})
\end{align*}
and hence from
\begin{align*}
&\big[A(y^*(x^{\prime\prime},t^{\prime\prime}),x^{\prime\prime},t^{\prime\prime})-A(y^*(x^{\prime\prime},t^{\prime\prime}),x^{\prime\prime},t^{\prime})\big]
 -\big[B(\tau^*(x^{\prime},t^{\prime}), x^{\prime\prime},t^{\prime\prime})-B(\tau^*(x^{\prime},t^{\prime}), x^{\prime\prime},t^{\prime})\big]\\&\leq \big[A(y^*(x^{\prime\prime},t^{\prime\prime}),x^{\prime},t^{\prime})-A(y^*(x^{\prime\prime},t^{\prime\prime}),x^{\prime\prime},t^{\prime})\big]-\big[B(\tau^*(x^{\prime},t^{\prime}), x^{\prime},t^{\prime})-B(\tau^*(x^{\prime},t^{\prime}), x^{\prime\prime},t^{\prime})\big]
\end{align*}
  we get
  \begin{equation}\label{equation4.14}
      \begin{aligned}
        \frac{\frac{1}{t^{\prime\prime}-t^{\prime}}\int_{t^{\prime}}^{t^{\prime\prime}}e^{-\beta(\eta)}\Big[f(h(x^{\prime\prime}-y^*(x^{\prime\prime},t^{\prime\prime}),\eta)e^{\beta(\eta)})-f(h(x^{\prime\prime},\eta, \tau^*(x^{\prime\prime},t^{\prime\prime}))e^{\beta(\eta)})\Big]d\eta}{\frac{1}{x^{\prime\prime}-x^{\prime}}\int_{x^{\prime}-y^*(x^{\prime\prime},t^{\prime\prime})}^{x^{\prime\prime}-y^*(x^{\prime\prime},t^{\prime\prime})}h(z,t^{\prime})d\eta-\frac{1}{x^{\prime\prime}-x^{\prime}}\int_{x^{\prime}}^{x^{\prime\prime}}h(\eta,t^{\prime}, \tau^*(x^{\prime},t^{\prime}))d\eta}\geq \frac{x^{\prime\prime}-x^{\prime}}{t^{\prime\prime}-t^{\prime}}.
      \end{aligned}
  \end{equation}
 Now passing to the limit as $t^{\prime\prime},t^{\prime} \to t$ in \eqref{equation4.13}-\eqref{equation4.14} we get the identity of $(iii).$\\
 Proof of $(iv)$ can be obtained from the following observation. When $A(x,t)=B(x,t),$
 \[\frac{Y_2(t^{\prime\prime})-Y_2(t^{\prime})}{t^{\prime\prime}-t^{\prime}}\geq \frac{x^{\prime\prime}-x^{\prime}}{t^{\prime\prime}-t^{\prime}} \geq \frac{X_2(t^{\prime\prime})-X_2(t^{\prime})}{t^{\prime\prime}-t^{\prime}}.\]
 Now from the definition of $X_2(\xi),Y_2(\xi)$ and using the fact $y^*(x,t)=\tau^*(x,t)=0$ we get $(iv).$
\end{proof}
In the next lemma, we show that the constructed curve actually starts from the $x$-axis or $t$-axis. The proof can be completed following exactly the same way as in \cite{neumann2021initial} and by appropriately replacing the characteristic lines by {\em h-curves}. Hence we omit it.
%Lemma~\ref{lem: curves} to the initial- or boundary manifold.\\
\begin{lemma}\label{lem:contcurves}
There is a countable set $\mathcal{S}$ of points on the $x$- and $t$-axis with the following properties.
\begin{enumerate}
\item \label{defcurvesX}For all $(\eta,0)\not\in \mathcal{S}$ there is a unique Lipschitz continuous curve $x=\mathtt{X}(\eta,t)$, $t\geq0$, such that $\mathtt{X}(\eta,0)=\eta$ and the characteristic triangles associated to points on the curve form an increasing family of sets.
\item \label{defcurvesY}For all $(0,\eta)\not\in \mathcal{S}$  there is a unique Lipschitz continuous curve $x=\mathtt{Y}(\eta,t)$, $t\geq\eta$, such that $\mathtt{Y}(\eta,\eta)=0$ and the characteristic triangles associated to points on the curve form an increasing family of sets.
\end{enumerate}
Further, for all $\eta>0$ such that $(\eta,0)$ and $(0,\eta)$ do not belong to $\mathcal{S}$,
\begin{equation*}
\begin{aligned}
\tfrac{\partial}{\partial t}\mathtt{X}(\eta, t)&= f^{\prime}\left(u(\mathtt{X}(\eta, t), t)\right) &&\text{for almost all $t>0$} \\
\tfrac{\partial}{\partial t}\mathtt{Y}(\eta, t)&= f^{\prime}\left(u(\mathtt{Y}(\eta, t), t)\right) &&\text{for almost all $t>\eta$}\,,
\end{aligned}
\end{equation*}
where the right-hand side is a measurable function.
\end{lemma}
\begin{remark}
We constructed the curve $\mathtt{X}(t)$ in the quarter plane $x>0, t>0.$ If the curve reaches $x=0,$ then for that point $(0,t),$ it falls precisely under the case $A(0,t) \leq B(0,t)$ and we show the boundary condition by Bardos- le Roux- N\'{e}d\'{e}lec is satisfied, since $f^{\prime}(u(0+,t)) \leq 0.$
\end{remark}
\begin{remark}
Note that, for a fixed $t>0,$ if $A(x,t)=B(x,t)$ in an interval $[a(t), b(t)]$ then $ \tau^*(x,t)=\tau_{*}(x,t)=0,\,\, y_*(x,t)=y^*(x,t)=0.$ Indeed,  if for any point $(x_0, t_0) \in [a(t), b(t)],$  $\tau^*(x_0, t_0) >0,$ then by the non-intersecting property, \cref{NIP} we arrive at a contradiction. This property corresponds to the rarefaction solution.
\end{remark}
Next, we show that the constructed solution is entropy admissible in the sense of Lax.
\begin{theorem}[\textbf{Entropy condition}]
\label{Entropy condition}
Let $(x, t)\in Q$ be an interior point in the quarter plane and also be a point of discontinuity of $u(x,t)$ lying on the generalized characteristic   $\mathtt{X}(t)$. If  $\mathtt{X}$ is differentiable at $t,$ then 
\begin{equation}\label{entropy}
    f^{\prime}(u(x-,t))> \dot{\mathtt{X}}(t) > f^{\prime}(u(x+, t)). 
\end{equation}
Consequently, at almost every point of discontinuity of $u(x,t)$ along the corresponding generalized characteristic curve, the above inequality holds.
\end{theorem}
\begin{proof}
We need to prove \eqref{entropy} for $A(x,t)<B(x,t),$ $A(x,t)>B(x,t)$ and $A(x,t)=B(x,t).$ 
When $A(x,t)< B(x,t)$ and $y_*(x,t)<y^*(x,t),$ observe that $u(x-,t)=h(x-y_*(x,t),t)e^{\beta(t)}$ and $u(x+,t)=h(x-y^*(x,t),t)e^{\beta(t)}.$ Let $t<t^{\prime}<t^{\prime\prime}$ and $\mathtt{X}(t)=x,\mathtt{X}(t^{\prime})=x^{\prime},\mathtt{X}(t^{\prime\prime})=x^{\prime\prime}.$ The characteristic curve through the points $\{(x^{\prime\prime}, t^{\prime\prime}), (y_*(x^{\prime\prime}, t^{\prime\prime}), 0)\}$ is given by
\[X_1(\xi)=y_*(x^{\prime\prime},t^{\prime\prime})+\int_0^{\xi}f^{\prime}\left(h(x^{\prime\prime}-y_*(x^{\prime\prime},t^{\prime\prime}),t^{\prime\prime})e^{\beta(\theta)}\right)d\theta.\]
Also the characteristic curve through the points $\{(x^{\prime\prime}, t^{\prime\prime}), (y^*(x^{\prime\prime}, t^{\prime\prime}), 0)\}$ is given by
\[X_2(\xi)=y^*(x^{\prime\prime},t^{\prime\prime})+\int_0^{\xi}f^{\prime}\left(h(x^{\prime\prime}-y^*(x^{\prime\prime},t^{\prime\prime}),t^{\prime\prime})e^{\beta(\theta)}\right)d\theta.\]
Hence from \eqref{equ58} we have
\begin{equation}\label{equation5.13sec5}
    \frac{\int_{t^{\prime}}^{t^{\prime\prime}}f^{\prime}\big(h(x-y_*(x^{\prime\prime},t^{\prime\prime}),t^{\prime\prime})e^{\beta(\theta)}\big)d\theta}{t^{\prime\prime}-t^{\prime}}> \frac{x^{\prime\prime}-x^{\prime}}{t^{\prime\prime}-t^{\prime}}>\frac{\int_{t^{\prime}}^{t^{\prime\prime}}f^{\prime}\big(h(x-y^*(x^{\prime\prime},t^{\prime\prime}),t^{\prime\prime})e^{\beta(\theta)}\big)d\theta}{t^{\prime\prime}-t^{\prime}}.
\end{equation}
Now passing to the limit as $t^{\prime\prime}, t^{\prime} \to t$ in \eqref{equation5.13sec5} and using semicontinuty property of $y_*$,$y^*$ in \cref{lnew} we get the desired inequality \eqref{entropy}. 
 When $B(x,t)<A(x,t),$ we have $u(x-,t)=h(x,t, \tau^*(x,t))e^{\beta(t)}$ and $u(x+, t)=h(x, t, \tau_*(x,t))e^{\beta(t)}$ and the rest of the proof will be similar. For the case $A(x,t)=B(x,t),$ we have $u(x-,t)=h(x,t, \tau^*(x,t))e^{\beta(t)}$ and $u(x+, t)=h(x, t, y^*(x,t))e^{\beta(t)}$ and  \eqref{entropy} can be similarly proven. This completes the proof.
\end{proof}
\vspace{.5cm}
\subsection{Proof of \cref{UES}} In this section, our aim is to prove \cref{UES}. First, we recall some well-known definitions from  \cite{Bressanbook} below. We consider the homogeneous equation
\begin{align}\label{eqq3.15}
    w_t+F(w)_x=0, \,\,\,\, \text{where}\,\,\, w:=e^{-\beta(t)}u\,\,\, \text{and}\,\,\, F(w):= e^{-\beta(t)}f(e^{\beta(t)}w).
\end{align}
\begin{definition}[\textbf{Entropy flux pairs}]
    A $C^1$-function $\eta: \RR \to \RR$ is called an \textit{entropy} for \eqref{eqq3.15}, with \textit{entropy flux} $q: \RR\to \RR,$ if 
    \begin{align}\label{entropy flux pairs}
        \eta^{\prime}(w)\cdot F^{\prime}(w)=q^{\prime}(w)\,\,\, \text{for}\,\,\, w \in \RR.
    \end{align}
\end{definition}
\begin{definition}[\textbf{Kruzkov entropy inequality}] A weak solution $w \in L^{\infty}(\RR)\cap BV_{loc}(\RR)$ of \eqref{eqq3.15} is said to satisfy \textit{Kruzkov entropy inequality} if for every entropy pairs $(\eta, q)$ and for every non-negative $\varphi \in C^1_0([0, \infty)\times(0, \infty))$ the following integral inequality holds:
\begin{align}\label{KEI}
    \int_0^{\infty}\int_0^{\infty}\left(\eta(w)\frac{\partial \varphi}{\partial t}+q(w)\frac{\partial \varphi}{\partial t}\right)dxdt \geq -\int_0^{\infty}q(w(0+, t))\varphi(0, t)dt
\end{align}
\end{definition}
\begin{proof}[Proof of \cref{UES}] The proof is very similar to \cites{lefloch, adimurthi2000conservation}. For the sake of completeness, we only provide a sketch of the proof.
\\
\noindent\textbf{Step 1.} It is known that for every $k\in \RR,$ one can choose 
$\eta (w)=|w-k|$ and $q(w)=\left(F(w)-F(k)\right)sgn (w-k)$ which satisfy \eqref{entropy flux pairs}. Then from \eqref{KEI} we have
\begin{align}\label{eqq3.18}
    \int_0^{\infty}\int_0^{\infty}\left(|w(x,t)-k|\frac{\partial \varphi}{\partial t}+F(w(x,t), k)\frac{\partial \varphi}{\partial t}\right)dxdt \geq -\int_0^{\infty}F(w(0+, t), k)\varphi(0, t)dt
\end{align}
where \[F(w(x,t), k)=\frac{F(w(x,t))-F(k)}{(w(x,t)-k)}|w(x,t)-k|.\] Moreover, following the same calculations of \cite[page 35]{adimurthi2000conservation}, one can deduce that if $w$ is $C^1$ except a discrete set of Lipschitz curves, then the Lax entropy condition \eqref{entropy} implies \eqref{eqq3.18}.\\
\noindent\textbf{Step 2.} Next we use Kruzkov's idea of doubling the variables. Let $\varphi \geq 0$ be a compactly supported $C^1$ function of $x,t, y, s$ and taking $k=e^{-\beta(s)}v(y,s),$ and following the same proof of \cite[page 30-33]{adimurthi2000conservation} we obtain to the inequality (2.11) of \cite[page 33]{adimurthi2000conservation} which reads
\begin{align}\label{eqq3.19}
    \int_0^{\infty}\varphi^{\prime}(t)\int_{\max\{a+Mt, 0\}}^{b-Mt} e^{-\beta(t)}|u(x,t)-v(x,t)|dxdt \nonumber\\
    \geq -\int_0^{\max\{\frac{-a}{M}, 0\}}F\left(e^{-\beta(t)}u(0+, t), e^{-\beta(t)}v(0+, t)\right)\varphi(t)dt
\end{align}
where $b-a \geq 2MT$ and $0\leq \varphi\in C^{\infty}_0(0, T).$ If $a \geq 0,$ then $\max\{-\frac{-a}{M}, 0\}=0$ and hence $t\mapsto \int_{a+Mt}^{b-Mt}e^{-\beta(t)}|u(x,t)-v(x,t)|dx$ is a decreasing function. This implies
\begin{align*}
    \int_{a+Mt}^{b-Mt}e^{-\beta(t)}|u(x,t)-v(x,t)|dx \leq \int_a^b |u_0(x)-v_0(x)|dx
\end{align*}
which gives $u=v$ as $u_0(x)=v_0(x).$ Now let $a<0<b,, a+b\geq 0, c<0<-a, c-a\leq 0, T_0=\frac{-a}{M}=\frac{d}{M}.$ With this choice of $a,b,c, T_0 < \min\left\{\frac{b-a}{2M}, \frac{-(c+a)}{2M}\right\}$ and $\varphi \in C^{\infty}_0(0, T_0),$ from \eqref{eqq3.19} we get
\begin{align}\label{eqq3.20}
    \int_0^{T_0}\varphi^{\prime}(t)\int_{c+Mt}^{b-Mt}e^{-\beta(t)}|u(x,t)-v(x,t)|dxdt \geq -\int_0^{T_0}F\left(e^{-\beta(t)}u(0+, t), e^{-\beta(t)}v(0+, t)\right)\varphi(t)dt.
\end{align} 
Now it is enough to show that 
\begin{align}\label{eqq3.21}
    F\left(e^{-\beta(t)}u(0+, t), e^{-\beta(t)}v(0+, t)\right)=\frac{f(u(0+,t))-f(v(0+, t))}{u(0+,t)-v(0+, t)}e^{-\beta(t)}|u(0+, t)-v(0+, t)|\leq 0.
\end{align}
Without loss of any generality we assume $u(0+, t)\geq v(0+, t)$ and consider following four cases:\\
1. $f^{\prime}(u(0+, t))> 0$ and $f^{\prime}(v(0+, t))> 0,$\\
2. $f^{\prime}(u(0+, t))> 0$ and $f^{\prime}(v(0+, t))\leq 0,$\\
3. $f^{\prime}(u(0+, t))\leq 0$ and $f^{\prime}(v(0+, t)) > 0,$\\
4. $f^{\prime}(u(0+, t))\leq 0$ and $f^{\prime}(v(0+, t))\leq 0.$\\
In the first case, from the boundary condition \eqref{boundary condition} we get $u(0+, t)=\bar{u}_b(t)=v(0+, t)$ and hence \eqref{eqq3.21} is proved. In the second case, using boundary condition \eqref{boundary condition} we obtain
\[f(u(0+, t))-f(v(0+, t))\leq f(\bar{u}_b(t))-f(\bar{u}_b(t))=0.\] Note that since we assumed $u(0+, t)\geq v(0+, t)$ and $f$ is a $C^2$ convex function, we have $f^{\prime}(u(0+, t)) \geq f^{\prime}(v(0+, t)).$ Therefore the third case can not occur. In the fourth case, using the mean value theorem, we obtain
\begin{align*}
  \frac{f(u(0+,t))-f(v(0+, t))}{u(0+,t)-v(0+, t)}=f^{\prime}(\xi(t))\leq f^{\prime}(u(0+, t))\leq 0.  
\end{align*}
This completes the proof of all four cases. The rest of the proof can be completed by using the inequality in \eqref{eqq3.20} and for details see \cite[page 34]{adimurthi2000conservation}. Hence we have $u\equiv v.$
\end{proof}

\subsection*{Acknowledgement} The authors gratefully acknowledge the comments and suggestions made by the anonymous referee to improve the manuscript.

\vspace{.5cm}
\noindent\textbf{Conflict of interest.} On behalf of all authors, the corresponding author states that there is no conflict of interest.

\vspace{.5cm}
\noindent\textbf{Data availability.} This manuscript has no associated data.
%\begin{thebibliography}{10}
\bibliographystyle{plain}	
\bibliography{Reference}

% \bib, bibdiv, biblist are defined by the amsrefs package.
\begin{bibdiv}
\begin{biblist}

\bib{manish}{article}{
      author={Adimurthi},
      author={Singh, Manish},
      author={Veerappa~Gowda, G.~D.},
       title={Lax-{O}le\u{\i}nik explicit formula and structure theory for
  balance laws},
        date={2020},
        ISSN={0022-0396},
     journal={J. Differential Equations},
      volume={268},
      number={11},
       pages={6517\ndash 6575},
         url={https://doi.org/10.1016/j.jde.2019.11.044},
      review={\MR{4075555}},
}

\bib{Bardos}{article}{
      author={Bardos, C.},
      author={le~Roux, A.~Y.},
      author={N\'{e}d\'{e}lec, J.-C.},
       title={First order quasilinear equations with boundary conditions},
        date={1979},
        ISSN={0360-5302},
     journal={Comm. Partial Differential Equations},
      volume={4},
      number={9},
       pages={1017\ndash 1034},
         url={https://doi.org/10.1080/03605307908820117},
      review={\MR{542510}},
}

\bib{sharpe1911problem}{article}{
      author={Sharpe, Francis~R},
      author={Lotka, Alfred~J},
       title={L. a problem in age-distribution},
        date={1911},
     journal={The London, Edinburgh, and Dublin Philosophical Magazine and
  Journal of Science},
      volume={21},
      number={124},
       pages={435\ndash 438},
}

\bib{mackey1993multistability}{article}{
      author={Mackey, MC},
      author={Rey, A},
       title={Multistability and boundary layer development in a transport
  equation with retarded arguments},
        date={1993},
     journal={Can. Appl. Math. Quart},
      volume={1},
       pages={1\ndash 21},
}

\bib{gyllenberg1982nonlinear}{article}{
      author={Gyllenberg, Mats},
       title={Nonlinear age-dependent population dynamics in continuously
  propagated bacterial cultures},
        date={1982},
     journal={Mathematical Biosciences},
      volume={62},
      number={1},
       pages={45\ndash 74},
}

\bib{Leflochbook}{book}{
      author={LeFloch, Philippe~G.},
       title={Hyperbolic systems of conservation laws},
      series={Lectures in Mathematics ETH Z\"{u}rich},
   publisher={Birkh\"{a}user Verlag, Basel},
        date={2002},
        ISBN={3-7643-6687-7},
         url={https://doi.org/10.1007/978-3-0348-8150-0},
        note={The theory of classical and nonclassical shock waves},
      review={\MR{1927887}},
}

\bib{serrebook}{book}{
      author={Serre, Denis},
       title={Systems of conservation laws. 1},
   publisher={Cambridge University Press, Cambridge},
        date={1999},
        ISBN={0-521-58233-4},
         url={https://doi.org/10.1017/CBO9780511612374},
        note={Hyperbolicity, entropies, shock waves, Translated from the 1996
  French original by I. N. Sneddon},
      review={\MR{1707279}},
}

\bib{Bressanbook}{book}{
      author={Bressan, Alberto},
       title={Hyperbolic systems of conservation laws},
      series={Oxford Lecture Series in Mathematics and its Applications},
   publisher={Oxford University Press, Oxford},
        date={2000},
      volume={20},
        ISBN={0-19-850700-3},
        note={The one-dimensional Cauchy problem},
      review={\MR{1816648}},
}

\bib{hopf}{article}{
      author={Hopf, Eberhard},
       title={The partial differential equation {$u_t+uu_x=\mu u_{xx}$}},
        date={1950},
        ISSN={0010-3640},
     journal={Comm. Pure Appl. Math.},
      volume={3},
       pages={201\ndash 230},
         url={https://doi.org/10.1002/cpa.3160030302},
      review={\MR{47234}},
}

\bib{lax}{article}{
      author={Lax, P.~D.},
       title={Hyperbolic systems of conservation laws. {II}},
        date={1957},
        ISSN={0010-3640},
     journal={Comm. Pure Appl. Math.},
      volume={10},
       pages={537\ndash 566},
         url={https://doi.org/10.1002/cpa.3160100406},
      review={\MR{93653}},
}

\bib{Joseph88}{article}{
      author={Joseph, K.~T.},
       title={Burgers' equation in the quarter plane, a formula for the weak
  limit},
        date={1988},
        ISSN={0010-3640},
     journal={Comm. Pure Appl. Math.},
      volume={41},
      number={2},
       pages={133\ndash 149},
         url={https://doi.org/10.1002/cpa.3160410202},
      review={\MR{924681}},
}

\bib{joseph}{article}{
      author={Joseph, K.~T.},
      author={Veerappa~Gowda, G.~D.},
       title={Explicit formula for the solution of convex conservation laws
  with boundary condition},
        date={1991},
        ISSN={0012-7094},
     journal={Duke Math. J.},
      volume={62},
      number={2},
       pages={401\ndash 416},
         url={https://doi.org/10.1215/S0012-7094-91-06216-2},
      review={\MR{1104530}},
}

\bib{lefloch}{article}{
      author={LeFloch, Philippe},
       title={Explicit formula for scalar nonlinear conservation laws with
  boundary condition},
        date={1988},
        ISSN={0170-4214},
     journal={Math. Methods Appl. Sci.},
      volume={10},
      number={3},
       pages={265\ndash 287},
         url={https://doi.org/10.1002/mma.1670100305},
      review={\MR{949657}},
}

\bib{F10}{article}{
      author={Frankowska, H\'{e}l\`ene},
       title={On {L}e{F}loch's solutions to the initial-boundary value problem
  for scalar conservation laws},
        date={2010},
        ISSN={0219-8916,1793-6993},
     journal={J. Hyperbolic Differ. Equ.},
      volume={7},
      number={3},
       pages={503\ndash 543},
         url={https://doi.org/10.1142/S0219891610002219},
      review={\MR{2735819}},
}

\bib{MR457947}{article}{
      author={Dafermos, C.~M.},
       title={Generalized characteristics and the structure of solutions of
  hyperbolic conservation laws},
        date={1977},
        ISSN={0022-2518},
     journal={Indiana Univ. Math. J.},
      volume={26},
      number={6},
       pages={1097\ndash 1119},
         url={https://doi.org/10.1512/iumj.1977.26.26088},
      review={\MR{457947}},
}

\bib{CR05}{article}{
      author={Colombo, Rinaldo~M.},
      author={Rosini, Massimiliano~D.},
       title={Well posedness of balance laws with boundary},
        date={2005},
        ISSN={0022-247X},
     journal={J. Math. Anal. Appl.},
      volume={311},
      number={2},
       pages={683\ndash 702},
         url={https://doi.org/10.1016/j.jmaa.2005.03.008},
      review={\MR{2168427}},
}

\bib{CR07}{article}{
      author={Colombo, Rinaldo~M.},
      author={Rosini, Massimiliano~D.},
       title={Well posedness of balance laws with non-characteristic boundary},
        date={2007},
        ISSN={0392-4041},
     journal={Boll. Unione Mat. Ital. Sez. B Artic. Ric. Mat. (8)},
      volume={10},
      number={3, bis},
       pages={875\ndash 894},
      review={\MR{2507903}},
}

\bib{CG10}{article}{
      author={Colombo, Rinaldo~M.},
      author={Guerra, Graziano},
       title={On general balance laws with boundary},
        date={2010},
        ISSN={0022-0396},
     journal={J. Differential Equations},
      volume={248},
      number={5},
       pages={1017\ndash 1043},
         url={https://doi.org/10.1016/j.jde.2009.12.002},
      review={\MR{2592880}},
}

\bib{CR15}{article}{
      author={Colombo, Rinaldo~M.},
      author={Rossi, Elena},
       title={Rigorous estimates on balance laws in bounded domains},
        date={2015},
        ISSN={0252-9602},
     journal={Acta Math. Sci. Ser. B (Engl. Ed.)},
      volume={35},
      number={4},
       pages={906\ndash 944},
         url={https://doi.org/10.1016/S0252-9602(15)30028-X},
      review={\MR{3355359}},
}

\bib{ABBCN23}{article}{
      author={Ancona, Fabio},
      author={Bianchini, Stefano},
      author={Bressan, Alberto},
      author={Colombo, Rinaldo~M.},
      author={Nguyen, Khai~T.},
       title={Examples and conjectures on the regularity of solutions to
  balance laws},
        date={2023},
        ISSN={1552-4485},
     journal={Quart. Appl. Math.},
      volume={81},
      number={4},
       pages={433\ndash 454},
         url={https://doi.org/10.1090/qam/1647},
}

\bib{MR1924285}{article}{
      author={Dias, Jo\~{a}o-Paulo},
      author={LeFloch, Philippe~G.},
       title={Some existence results for conservation laws with source-term},
        date={2002},
        ISSN={0170-4214},
     journal={Math. Methods Appl. Sci.},
      volume={25},
      number={13},
       pages={1149\ndash 1160},
         url={https://doi.org/10.1002/mma.332},
      review={\MR{1924285}},
}

\bib{MR913412}{article}{
      author={Liu, Tai-Ping},
       title={Nonlinear resonance for quasilinear hyperbolic equation},
        date={1987},
        ISSN={0022-2488},
     journal={J. Math. Phys.},
      volume={28},
      number={11},
       pages={2593\ndash 2602},
         url={https://doi.org/10.1063/1.527751},
      review={\MR{913412}},
}

\bib{MR895101}{article}{
      author={Dafermos, C.~M.},
       title={Trend to steady state in a conservation law with spatial
  inhomogeneity},
        date={1987},
        ISSN={0033-569X},
     journal={Quart. Appl. Math.},
      volume={45},
      number={2},
       pages={313\ndash 319},
         url={https://doi.org/10.1090/qam/895101},
      review={\MR{895101}},
}

\bib{MR1184059}{article}{
      author={Lyberopoulos, Athanasios~N.},
       title={Large-time structure of solutions of scalar conservation laws
  without convexity in the presence of a linear source field},
        date={1992},
        ISSN={0022-0396},
     journal={J. Differential Equations},
      volume={99},
      number={2},
       pages={342\ndash 380},
         url={https://doi.org/10.1016/0022-0396(92)90026-J},
      review={\MR{1184059}},
}

\bib{MR1286920}{article}{
      author={Lyberopoulos, Athanasios~N.},
       title={A {P}oincar\'{e}-{B}endixson theorem for scalar balance laws},
        date={1994},
        ISSN={0308-2105},
     journal={Proc. Roy. Soc. Edinburgh Sect. A},
      volume={124},
      number={3},
       pages={589\ndash 607},
         url={https://doi.org/10.1017/S0308210500028791},
      review={\MR{1286920}},
}

\bib{MR1245070}{article}{
      author={Fan, Hai~Tao},
      author={Hale, Jack~K.},
       title={Large time behavior in inhomogeneous conservation laws},
        date={1993},
        ISSN={0003-9527},
     journal={Arch. Rational Mech. Anal.},
      volume={125},
      number={3},
       pages={201\ndash 216},
         url={https://doi.org/10.1007/BF00383219},
      review={\MR{1245070}},
}

\bib{MR1270661}{article}{
      author={Fan, Hai~Tao},
      author={Hale, Jack~K.},
       title={Attractors in inhomogeneous conservation laws and parabolic
  regularizations},
        date={1995},
        ISSN={0002-9947},
     journal={Trans. Amer. Math. Soc.},
      volume={347},
      number={4},
       pages={1239\ndash 1254},
         url={https://doi.org/10.2307/2154808},
      review={\MR{1270661}},
}

\bib{MR1693629}{article}{
      author={H\"{a}rterich, J\"{o}rg},
       title={Heteroclinic orbits between rotating waves in hyperbolic balance
  laws},
        date={1999},
        ISSN={0308-2105},
     journal={Proc. Roy. Soc. Edinburgh Sect. A},
      volume={129},
      number={3},
       pages={519\ndash 538},
         url={https://doi.org/10.1017/S0308210500021491},
      review={\MR{1693629}},
}

\bib{MR2172699}{article}{
      author={Ehrt, Julia},
      author={H\"{a}rterich, J\"{o}rg},
       title={Asymptotic behavior of spatially inhomogeneous balance laws},
        date={2005},
        ISSN={0219-8916},
     journal={J. Hyperbolic Differ. Equ.},
      volume={2},
      number={3},
       pages={645\ndash 672},
         url={https://doi.org/10.1142/S0219891605000579},
      review={\MR{2172699}},
}

\bib{MR1751427}{article}{
      author={Mascia, Corrado},
      author={Sinestrari, Carlo},
       title={The perturbed {R}iemann problem for a balance law},
        date={1997},
        ISSN={1079-9389},
     journal={Adv. Differential Equations},
      volume={2},
      number={5},
       pages={779\ndash 810},
      review={\MR{1751427}},
}

\bib{MR1432096}{article}{
      author={Sinestrari, Carlo},
       title={Instability of discontinuous traveling waves for hyperbolic
  balance laws},
        date={1997},
        ISSN={0022-0396},
     journal={J. Differential Equations},
      volume={134},
      number={2},
       pages={269\ndash 285},
         url={https://doi.org/10.1006/jdeq.1996.3223},
      review={\MR{1432096}},
}

\bib{MR1427730}{article}{
      author={Sinestrari, Carlo},
       title={The {R}iemann problem for an inhomogeneous conservation law
  without convexity},
        date={1997},
        ISSN={0036-1410},
     journal={SIAM J. Math. Anal.},
      volume={28},
      number={1},
       pages={109\ndash 135},
         url={https://doi.org/10.1137/S003614109427446X},
      review={\MR{1427730}},
}

\bib{MR1808451}{incollection}{
      author={Fan, Haitao},
      author={Jin, Shi},
      author={Teng, Zhen-huan},
       title={Zero reaction limit for hyperbolic conservation laws with source
  terms},
        date={2000},
      volume={168},
       pages={270\ndash 294},
         url={https://doi.org/10.1006/jdeq.2000.3887},
      review={\MR{1808451}},
}

\bib{neumann2021initial}{article}{
      author={Neumann, Lukas},
      author={Oberguggenberger, Michael},
      author={Sahoo, Manas~R.},
      author={Sen, Abhrojyoti},
       title={Initial-boundary value problem for 1d pressureless gas dynamics},
        date={2022},
        ISSN={0022-0396},
     journal={Journal of Differential Equations},
      volume={316},
       pages={687\ndash 725},
  url={https://www.sciencedirect.com/science/article/pii/S0022039622000766},
      review={\MR{MR4379304}},
}

\bib{adimurthi2000conservation}{article}{
      author={Adimurthi},
      author={Gowda, G. D.~Veerappa},
       title={Conservation law with discontinuous flux},
        date={2003},
        ISSN={0023-608X},
     journal={J. Math. Kyoto Univ.},
      volume={43},
      number={1},
       pages={27\ndash 70},
         url={https://doi.org/10.1215/kjm/1250283740},
      review={\MR{2028700}},
}

\end{biblist}
\end{bibdiv}

\end{document}